\documentclass[10pt]{amsart}
\usepackage[all,arc]{xy}
\usepackage{tikz-cd}
\usepackage{tikz}
\usetikzlibrary{arrows,decorations}

\usepackage{adjustbox}
\usepackage{enumitem}
\usepackage{mathrsfs}
\usepackage{multicol}
\usepackage{xcolor,colortbl}
\usepackage{color,soul}
\usepackage{amssymb,amsfonts,amsmath,nccmath,tabularx,multirow}

\setlist[enumerate]{leftmargin=2.5em}
\setlist[itemize]{leftmargin=2.5em}

\newcommand{\nocontentsline}[3]{}
\newcommand{\tocless}[2]{\bgroup\let\addcontentsline=\nocontentsline#1{#2}\egroup}

\usepackage{xr-hyper}
\usepackage{wrapfig}
\makeatletter
\newcommand*{\addFileDependency}[1]{% argument=file name and extension
  \typeout{(#1)}
  \@addtofilelist{#1}
  \IfFileExists{#1}{}{\typeout{No file #1.}}
}
\makeatother

\newcommand*{\myexternaldocument}[1]{%
    \externaldocument{#1}%
    \addFileDependency{#1.tex}%
    \addFileDependency{#1.aux}%
}

\myexternaldocument{MTarXivV2_forMC}
\usepackage
[pdfauthor={Catherine Hsu, Preston Wake, and Carl Wang-Erickson},
 pdftitle={Explicit non-Gorenstein R=T via rank bounds II: Computation},
 bookmarks=false,hidelinks]
{hyperref}

\newcounter{qcounter}

\newcommand\define{\newcommand}

\define\rk{\mathrm{rk}}

\define\isoto{\xrightarrow{\sim}}
\define\onto{\twoheadrightarrow}

\newcommand{\ttmat}[4]{\left( \begin{array}{cc}
#1 & #2 \\
#3 & #4
\end{array}
\right)}

\newcommand{\Z}{\mathbb{Z}}
\newcommand{\Q}{\mathbb{Q}}

\newcommand{\F}{\mathbb{F}}

\newcommand{\p}{\mathfrak{p}}
\newcommand{\m}{\mathfrak{m}}

\newcommand{\cL}{\mathcal{L}}

\newcommand{\Hom}{\mathrm{Hom}}
\newcommand{\Gal}{\mathrm{Gal}}
\newcommand{\Nm}{\mathrm{Nm}}

\newcommand{\Ext}{\mathrm{Ext}}

\newcommand{\Fr}{\mathrm{Fr}}

\newcommand{\ab}{\mathrm{ab}}

\define\ord{{\mathrm{ord}}}

\define\GL{{\mathrm{GL}}}

\define{\Fitt}{\mathrm{Fitt}}
\define{\Ann}{\mathrm{Ann}}

\define{\Fpgpring}{\mathbb{F}_p[\langle\sigma\rangle]}

\newtheorem{thm}{Theorem}[subsection] 

\newtheorem*{thm*}{Theorem}

\newtheorem{cor}[thm]{Corollary}
\newtheorem{prop}[thm]{Proposition}
\newtheorem{lem}[thm]{Lemma}

\theoremstyle{definition}
\newtheorem{defn}[thm]{Definition}

\newtheorem{eg}[thm]{Example}

\newtheorem{assump}[thm]{Assumption}

\theoremstyle{remark}
\newtheorem{rem}[thm]{Remark}

\newcommand{\ra}{\rightarrow}

\newcommand{\risom}{\buildrel\sim\over\ra}
\newcommand{\rinj}{\hookrightarrow}

\newcommand{\rsurj}{\twoheadrightarrow}

\newcommand{\bT}{\mathbb{T}}

\newcommand{\cO}{\mathcal{O}}

\newcommand{\fl}{{\mathrm{flat}}}

\newcommand{\Qz}{\mathbb{Q}(\zeta_p)}

\usepackage{color}

\newcommand{\bro}{{\bar\rho}}

\newcommand{\sm}[4]{\ensuremath{\big(\begin{smallmatrix}#1 & #2 \\ #3 & #4\end{smallmatrix}\big)}}

\newcommand{\up}[1]{^{(#1)}}
\newcommand{\upp}[1]{^{(#1)'}}

\usepackage{bbm}
\newcommand{\Bone}{\mathbbm{1}}

\newcommand{\cand}{_\mathrm{cand}}
\newcommand{\adj}{_\mathrm{adj}}

\newcommand{\oQ}{\overline{\Q}}

\newcommand{\ep}{\epsilon}

\makeatletter
\let\c@equation\c@thm
\makeatother

\numberwithin{equation}{subsection}

\usepackage[ruled,vlined,linesnumbered]{algorithm2e}
\SetKw{KwNotIn}{not in}
\SetKw{Kwin}{in}
\SetKw{KwBreak}{break}
\SetKw{KwAssert}{assert}
\SetKw{KwOr}{or}
\SetKw{KwForEvery}{for every}
\SetKw{KwAnd}{and}

\title[$R=\bT$ via rank bounds II]{Explicit non-Gorenstein $R=\bT$ via rank bounds II: Computational aspects}

\author{Catherine Hsu}
\address{Department of Mathematics \& Statistics, Swarthmore College,  Swarthmore, PA 19081, USA}
\email{chsu2@swarthmore.edu}

\author{Preston Wake}
\address{Department of Mathematics, Michigan State University, East Lansing, MI 48824, USA}
\email{wakepres@msu.edu}

\author{Carl Wang-Erickson}
\address{Department of Mathematics, University of Pittsburgh, Pittsburgh, PA 15260, USA}
\email{carl.wang-erickson@pitt.edu}

\begin{document}

\begin{abstract}
This is the second in a pair of papers about residually reducible Galois deformation rings with non-optimal level. In the first paper, we proved a Galois-theoretic criterion for the deformation ring to be as small as possible. This paper focuses on the computations needed to verify this criterion. We adapt a technique developed by Sharifi to compute number fields with twisted-Heisenberg Galois group and prescribed ramification, and compute the splitting behavior of primes in these extensions.
\end{abstract}

\maketitle

\tableofcontents

\section{Introduction}

\subsection{Summary}
\label{subsec: summary2}
In this series of two papers, we prove, under some hypotheses, an integral $R=\bT$ theorem for the mod-$p$ Galois representation $\bro = 1 \oplus \omega$, where $\bT$ is the Hecke algebra acting on weight $2$ modular forms of level $N=\ell_0\ell_1$, $p\geq 5$ is a prime number, $\omega$ is the mod-$p$ cyclotomic character, and $R$ is a universal pseudodeformation ring for $\overline{\rho}$. We are concerned with the case where
\begin{itemize}
    \item $\ell_0$ is a prime with $\ell_0 \equiv 1 \pmod{p}$, and
    \item $\ell_1$ is a prime with $\ell_1 \not \equiv \pm 1\pmod{p}$ such that $\ell_1$ is a $p$th power modulo $\ell_0$.
    \item There is a unique weight 2 cusp form of level $\ell_0$ that is congruent to the Eisenstein series modulo $p$.
\end{itemize}
See the introduction of Part I (\cite{part1}) for a discussion of why this particular setup is interesting from the point of view of Galois representations.
This second paper is focused on the computations needed to verify the hypotheses of the $R=\bT$ theorem proven in Part I.

The method we use to prove $R=\bT$ is new. As with the standard approach, we start with a surjection $R \onto \bT$ and show that if $R$ is ``small enough," this surjection must be an isomorphism. Standard methods use tangent space computations to show that $R$ is ``small enough," but these techniques are not enough in our setting because of the failure of the Gorenstein property. Instead, we show that $R$ is ``small enough" by bounding the dimension of $R/pR$ (as a vector space) under certain conditions.

In Part I of this pair of papers, we prove that 
\[
\dim_{\F_p}R/pR \le 3\Longleftrightarrow \dim_{\F_p}R/pR = 3\Longleftrightarrow R=\mathbb{T},
\] 
and then prove that $\dim_{\F_p}R/pR > 3$ if and only if certain Galois cochains exist and satisfy specific local conditions about their restrictions to decomposition groups at $\ell_0$, $\ell_1$, and $p$. Broadly, the main steps of this paper are as follows. 
\begin{itemize}
    \item Translate the cochain existence problem into a problem about the existence of extensions of $\Q(\zeta_p)$ with prescribed (nilpotent $p$-group) Galois group, and translate the local conditions into conditions on splitting behavior of primes in these extensions.
    \item Describe extensions with the desired Galois groups explicitly as iterated Kummer extensions, following Sharifi's construction of generalized Heisenberg extensions \cite{sharifithesis,sharifi2007,LLSWW2020}. We refer to these as \emph{twisted-Heisenberg extensions}. 
    \item Adjust the extensions constructed in the previous step so that they have the desired local properties. This involves understanding the local behavior of certain global cochains, which we achieve using a tame analog of the Gross-Stark conjecture, developed in \cite{WWE3,wake2020eisenstein}.
    \item Use Kummer theory to express the splitting behavior of primes in these extensions in terms of conditions on the Kummer generators.
    \item Perform computations in Sage \cite{SAGE} using the unit/S-unit interface to the unit/S-unit groups computed in Pari/GP \cite{PARI2}.
    \item Establish number-theoretic characterizations of these extension fields.
\end{itemize}
Using these computations, we find many explicit examples where the conditions for $\dim_{\F_p}R/pR \le 3$ are satisfied and conclude that $R=\bT$ in these cases. Moreover, we find examples where $\dim_{\F_p}R/pR >3$, and we compute that $\mathrm{rank}_{\Z_p}\bT>3$ in each of these cases, which is consistent with our $R=\bT$ conjecture.

Although the focus of this paper is on computing bounds for $\dim_{\F_p}R/pR$, we expect that some of the techniques developed here will be of independent interest. We expect that the same methods can be used to compute bounds on dimensions of residually-reducible deformation rings in other contexts. We also hope that this paper can serve as a guide for further computation and exploration in generalized Heisenberg extensions of number fields.

\subsection{Main results} 
\label{subsec: P2 main results}

We begin by formulating our main results in terms of splitting conditions in certain unipotent $p$-extensions of $\Q(\zeta_p)$ determined by $N=\ell_0\ell_1$. In the description below, we use $C_p$ to denote a cyclic order $p$ group. 

Let $K =\Q(\zeta_p,\ell_1^{1/p})$, and let $L/\Q(\zeta_p)$ be the $\omega^{-1}$-isotypic $C_p$-extension such that $(1-\zeta_p)$ splits and only the primes over $\ell_0$ ramify. (For the existence and uniqueness of $L/\Q(\zeta_p)$, see, e.g., \cite[Lem.\ 3.9]{CE2005}.) 
To state the main result of this paper, we require two special $C_p$-extensions of $K$ defined in \S\ref{subsec: P2 deducing main}, which we denote by $K'/K$ and $K''/K$. These $C_p$-extensions are characterized by certain Galois-theoretic and splitting conditions, including the ``$\omega^i$-isotypic'' condition that is defined in Definition \ref{defn: Delta isotypic}. In particular, the Galois-theoretic conditions also involve a tower of $C_p$-extensions of $L$ in which each extension in the tower is constructed by composing the previous subextension with an extension of $K$.

Diagrammatically, letting $M = KL$, we have: 
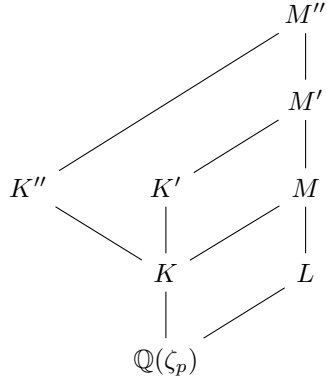
\[
\begin{tikzcd}[every arrow/.append style={-}]
&&M''\arrow{d}\arrow{ddll}\\
&&M'\arrow{d}\arrow{dl}\\
K''\arrow{dr}&K'\arrow{d}&M\arrow{d}\arrow{dl}\\
& K\arrow{d}&L\arrow{dl}\\
&\Q(\zeta_p)&
\end{tikzcd}
\]

In this setup, $M'/M$ is the unique $C_p$-extension such that 
\begin{itemize}
\item $M'/\Q$ is Galois and $M'/M$ is $\omega^0$-isotypic
\item $M'/M$ is unramified
\item the primes of $M$ over $\ell_0$ split in $M'/M$.
\end{itemize}
Then $K'/K$ is characterized up to isomorphism by being a $C_p$-extension of $K$ contained in $M'$ but not equal to $M$. See Proposition \ref{prop: characterize M'} for details. 

Likewise, assuming that the primes over $\ell_1$ split in $K'/K$ (equivalently, in $M'/M$), we can construct and identify another $C_p$-extension $M''/M'$. It is the unique $C_p$-extension of $M'$ such that
\begin{itemize}
\item $M''/\Q$ is Galois and $M''/M'$ is $\omega$-isotypic
\item the conductor of $M''/M'$ divides (resp.\ is equal to) $\m^\fl := \prod_{v \mid p} v^2$; that is, the product of the squares of the primes of $M'$ over $p$
\item primes of $M'$ over $\ell_0$ split in $M''/M'$. 
\end{itemize}
Then $K''/K$ is characterized up to isomorphism by being a $C_p$-extension of $K$ contained in $M''$, not contained in $M'$, and being a member of an isomorphism class (of subfields of $M''$) of cardinality $p$. See Proposition \ref{prop: characterize M''} for details.

The first main result of this paper gives splitting conditions in the $C_p$-extensions $K'/K$ and $K''/K$ for when $\dim_{\F_p}R/pR>3$. 
\begin{thm}[{Theorem \ref{thm: main}}]
\label{thm: main intro} 
We have $\dim_{\F_p}R/pR>3$ if and only if the following conditions hold:
    \begin{enumerate}[label=(\roman*)]
        \item all primes of $K$ over $\ell_1$ split in $K'/K$;
        \item there exists some prime of $K$ over $\ell_0$ that splits in both $K'/K$ and $K''/K$.
    \end{enumerate}
In particular, when $\dim_{\F_p}R/pR=3$, we have $R=\mathbb{T}$.
\end{thm}
The second main result of this paper is an algorithm that computes whether conditions $(i)$ and $(ii)$ in Theorem \ref{thm: main} hold. Indeed, since $K'/K$ and $K''/K$ are both $C_p$-extensions, each can be constructed by adjoining the $p$th root of an $S$-unit in $K$, where we have taken $S$ to be the set of primes of $K$ dividing $Np$. In particular, Kummer theory provides a computationally feasible way to check conditions $(i)$ and $(ii)$ even when the degrees of $K'/\Q$ and $K''/\Q$ are large, i.e., of degree $p^2(p-1)\geq 100$. The main components of our algorithm are given in \S\ref{sec:adjustments}, and the entire program, implemented using Sage \cite{SAGE}, can be found online at \url{https://github.com/cmhsu2012/RR3}. Our program is efficient enough that we have run it for some small values of $p$ and many values of $N$. 

Here is a sample result of our calculations. For a detailed discussion of all computed examples, see \S\ref{sec:data}.

\begin{thm}
\label{thm: main computations}
Let $p=5$ and $\ell_0=11$. Then for 
\[
\ell_1 =23,67,263,307,373,397,593,857,967,1013,
\] 
condition $(i)$ of Theorem \ref{thm: main} holds, but condition $(ii)$ does not. In particular, for these values of $\ell_1$, the $\F_p$-dimension of $R/pR$ equals $3$ and $R\cong\bT$. 

For $\ell_1 =43,197,683,727$, conditions $(i)$ and $(ii)$ of Theorem \ref{thm: main} both hold. Consequently, the $\F_p$-dimension of $R/pR$ exceeds $3$ for these values of $\ell_1$. 
\end{thm}

\begin{rem}
For the values of $p$ and $N$ where we found $\dim_{\F_p}R/pR > 3$, we also computed $\dim_{\F_p}\bT/p\bT>3$. This is consistent with our conjecture that $R\cong\bT$. 
\end{rem}

For the remainder of this introduction, we outline how conditions $(i)$ and $(ii)$ in Theorem \ref{thm: main} arise. Specifically, we explain how the $C_p$-extensions $K'/K$ and $K''/K$ are the splitting fields of certain Galois cochains and then show how to compute explicit $S$-units corresponding to these cochains using Kolyvagin derivative operators.  To describe this precisely, we summarize the necessary Galois cohomological framework established in Part I.

\subsection{Differential equations and the rank of $R$}
\label{subsec:diffeq and R}

Let $R$ be the ring representing deformations of the pseudorepresentation $\omega \oplus 1$ that have determinant equal to the cyclotomic character and that are finite-flat at $p$, unramified-or-Steinberg at the primes $\ell_0, \ell_1$ that divide $N$, and unramified outside $Np$. Note that $R$ only plays a motivational role in this paper, so we do not discuss this definition further -- see Part I for more information.

As we explain in the introduction of Part I, there is a particular first-order deformation $\rho_1$ of $\bro$ such that $\dim_{\F_p}R/pR \le 3$ unless a deformation $\rho_2$ of $\rho_1$ satisfying certain local conditions exists. 

There is a helpful analogy between this problem and boundary value problems in the theory of differential equations. The existence of $\rho_2$ is analogous to the existence of a general solution to the system of differential equations, and satisfying the local conditions is analogous to the existence of a solution to the boundary value problem. We will now use this analogy to frame our main results from Part I.

\subsubsection{The system of equations defining $\rho_1$}
The starting point for our study of $\dim_{\F_p}R/pR$ is the representation $\rho_1$ of $G_\Q := \Gal(\oQ/\Q)$. As an input, we start with two cocycles:
\begin{itemize}
    \item $b\up1$ represents the Kummer class of $\ell_1$ in $H^1(\Z[1/Np],\F_p(1))$, and
    \item $c\up1$ represents a nontrivial class in $H^1(\Z[1/Np],\F_p(-1))$ that is ramified only at $\ell_0$ (such a class is unique up to scaling).
\end{itemize}
If we wanted to represent $\rho_1$ as a matrix with values in $\GL_2(\F_p[\ep]/(\ep^2))$, there are two choices:
\[
\ttmat{\omega(1+a\up1 \epsilon)}{\ep b\up1}{\omega (c\up1+\ep c\up2)}{1+d\up1 \epsilon}
\mbox{ or }
\ttmat{\omega(1+a\up1 \epsilon)}{ b\up1+\ep b\up2}{\omega \ep c\up1}{1+d\up1 \epsilon}.
\]
In other words, we have to choose either the upper-right or lower-left entry to be a multiple of $\ep$. From the point of view of pseudorepresentations, both choices give the same answer because they have the same trace and determinant. To obviate this choice, we write $\rho_1$ as
\[
\rho_1 = \ttmat{\omega(1+a\up1 \epsilon)}{b\up1}{\omega c\up1}{1+d\up1 \epsilon}
\]
to refer to the pseudorepresentation obtained by either of these choices. (This ad hoc definition can be made more formal using the theory of \emph{Generalized Matrix Algebras (GMA)}---see [Part I, \S\ref{sssec: CH reps and GMA reps}] for the definition of GMAs or [Part I, \S\ref{subsec: 1-reducible}] for the 1-reducible GMAs relevant here.)

The determinant of $\rho_1$ is $\omega(1+(a\up1+d\up1-b\up1c\up1)\ep)$.
Since we require our deformations to have cyclotomic determinant, we must have $d\up1=b\up1c\up1-a\up1$. Since $b\up1$ and $c\up1$ are fixed, the data of $\rho_1$ is equivalent to the data of a cochain $a\up1$. For $\rho_1$ to be a homomorphism, it is equivalent that $a\up1$ satisfy the differential equation
\begin{equation}
    \label{eq:diffeq of a1}
    -\hspace{-.075cm}da\up1 = b\up1 \smile c\up1.
\end{equation}
In order for this equation to have a solution, it is equivalent that $\ell_1$ be a $p$th power modulo $\ell_0$ (see [Part I, Lem.\ \ref{lem: log ell1 is zero}]), which we assume. Hence \eqref{eq:diffeq of a1} has a solution.

Note that the solution $a\up1$ to \eqref{eq:diffeq of a1} is not unique, but any two solutions differ by a cocycle. In order for $\rho_1$ to satisfy the local conditions defining $R$, we must impose a condition on the ramification of $a\up1$ at $p$ (it can be shown that $\rho_1$ satisfies the conditions at $\ell_0$ and $\ell_1$ for all choices of $a\up1$---see [Part I, Lem.\ \ref{lem: D1 construction}]). Still, the solution with this condition is not unique: any two solutions differ by a cocycle that is unramified at $p$.

\subsubsection{The system of equations defining $\rho_2$}
\label{subsub:rho2}
Next we want to deform $\rho_1$ to a pseudorepresentation $\rho_2$ with coefficients in $\F_p[\ep]/(\ep^3)$.
We write our desired deformation as
\begin{equation}
\label{eq:rho2}
    \rho_2 = \ttmat{\omega(1+a\up1 \epsilon+a\up2 \epsilon^2)}{b\up1+b\up2\epsilon}{\omega (c\up1+c\up2 \epsilon)}{1+d\up1 \epsilon+d\up2 \epsilon^2}
\end{equation}
with the same convention as for $\rho_1$ that one or the other of the upper-right or lower-left entries should be multiplied by $\ep$ (or using the 1-reducible GMAs of [Part I, \S4.1]).

Just as with $d\up1$, in order that $\det(\rho_2)=\omega$ we must have $d\up2 =b\up1 c\up2+b\up2 c\up1 -a\up1 d\up1 -a\up2$. The data of the deformation $\rho_2$ is the remaining cochains $a\up2$, $b\up2$, and $c\up2$. These cochains must satisfy the following system of equations in order for $\rho_2$ to be a homomorphism:
\begin{align}
\label{eq:diffeq for a2}
    -d a\up2 &= a\up1 \smile a\up1 + b\up1 \smile c\up2 +b\up2 \smile c\up1 \\
\label{eq:diffeq for b2}
    -d b \up2 &= a\up1 \smile b\up1 + b\up1\smile d\up1 \\
\label{eq:diffeq for c2}
    -d c\up2 &= c\up1 \smile a\up1 + d\up1 \smile c\up1.
\end{align}
This system is much more complex than \eqref{eq:diffeq of a1}, due to the coupling of equations. However, we find that:
\begin{itemize}
    \item Equation \eqref{eq:diffeq for c2} has a solution for a unique value of $a\up1$ (among those that satisfy \eqref{eq:diffeq of a1} and the condition on ramification at $p$). This value is characterized by the condition that $a\up1|_{\ell_0}$ be in the image of the cup product 
    \[
    H^0(\Q_{\ell_0},\F_p(1)) \xrightarrow{\smile c\up1|_{\ell_0}} H^1(\Q_{\ell_0},\F_p).
    \]
    From now on, we fix $a\up1$ to be that solution, and we define $\alpha \in H^0(\Q_{\ell_0},\F_p(1))$ such that $\alpha \smile c\up1|_{\ell_0} = a\up1|_{\ell_0}$.
    \item Equation \eqref{eq:diffeq for b2} has a solution if and only if $a\up1|_{\ell_1}=0$.
    \item Supposing that \eqref{eq:diffeq for b2} has a solution, equation \eqref{eq:diffeq for a2} has a solution only for certain values of $b\up2$. These values are characterized by the condition that $b\up2|_{\ell_0}$ be in the image of the cup product 
    \[
    H^0(\Q_{\ell_0},\F_p(2)) \xrightarrow{\smile c\up1|_{\ell_0}} H^1(\Q_{\ell_0},\F_p(1)).
    \]
     For such a choice of $b\up2$, we define $\beta\in H^0(\Q_{\ell_0},\F_p(2))$ such that $\beta \smile c\up1|_{\ell_0} = b\up2|_{\ell_0}$.
\end{itemize}
In summary, we see that $\rho_2$ exists if and only if $a\up1|_{\ell_1}=0$. Moreover, when this is the case, there are invariants $\alpha \in H^0(\Q_{\ell_0},\F_p(1))$ and $\beta\in H^0(\Q_{\ell_0},\F_p(2))$.

Now assume that $\rho_2$ exists. In order for $\rho_2$ to satisfy the local conditions defining $R$, there are conditions that $a\up2$, $b\up2$, and $c\up2$ be finite-flat at $p$, and an additional condition at $\ell_0$ (No additional condition at $\ell_1$ is necessary---we show that the condition on $\rho_2$ at $\ell_1$ is satisfied for all choices of $a\up2$, $b\up2$, and $c\up2$.) We show that the finite-flat conditions can always be satisfied, and from now on we fix $b\up2$ to be a solution satisfying the finite-flat condition and fix $\beta$ to satisfy $\beta \smile c\up1|_{\ell_0} = b\up2|_{\ell_0}$ for this choice. 
The extra condition at $\ell_0$ can be expressed in terms of $\alpha$ and $\beta$:
\begin{itemize}
    \item $\alpha^2+\beta=0$ in $H^0(\Q_{\ell_0},\F_p(2))$.
\end{itemize}

\begin{rem}
In the theory of differential equations, an \emph{inverse problem} is one where the system of equations and the solution are given, and the unknown is the boundary values. This is analogous to our situation, in that we use the equation \eqref{eq:diffeq for c2} to find the correct local conditions for $a\up1$. 
\end{rem}

We summarize the important information as follows:
\begin{itemize}
    \item $a\up1$ is the unique solution to \eqref{eq:diffeq of a1} such that
    \begin{itemize}
        \item $a\up1$ satisfies a finite-flat condition at $p$, and
        \item $a\up1|_{\ell_0} =\alpha \smile c\up1|_{\ell_0}$ for some unique $\alpha \in H^0(\Q_{\ell_0},\F_p(1))$.
    \end{itemize}
    \item $b\up2$ is a solution to \eqref{eq:diffeq for b2} (if it exists) such that
    \begin{itemize}
        \item $b\up2$ satisfies a finite-flat condition at $p$, and
        \item $b\up2|_{\ell_0} =\beta \smile c\up1|_{\ell_0}$ for some unique $\beta \in H^0(\Q_{\ell_0},\F_p(2))$.
    \end{itemize}
\end{itemize}
Moreover, we know that $b\up2$ exists if and only if $a\up1|_{\ell_1}=0$. The following theorem is the main result of Part I. 
\begin{thm}[Part I, Theorem \ref{P1: thm: main with a1}]
\label{thm: main of part 1}
We have $\dim_{\F_p}R/pR>3$ if and only if the following conditions hold:
\begin{enumerate}[label = (\roman*)]
    \item $a\up1|_{\ell_1}=0$
    \item $\alpha^2+\beta=0$ in $H^0(\Q_{\ell_0},\F_p(2))$.
\end{enumerate}
\end{thm}

By realizing $K'/K$ as the splitting field of $a\up1|_{G_K}$ and $K''/K$ as the splitting field of $b\up2|_{G_K}$, we can translate Theorem \ref{thm: main of part 1} into the language of Theorem \ref{thm: main intro}. It remains to compute the $S$-units corresponding to  $a\up1|_{G_K}$ and $b\up2|_{G_K}$.

\subsection{Computing $S$-units to solve cup product and Massey product equations}
In the theory of differential equations, one technique used for solving boundary value problems is to first find a particular solution to the equation, then adjust that solution so it satisfies the boundary conditions. We take a similar two-step approach to computing $a\up1$ and $b\up2$:
\begin{description}
\item[Step 1] Find cochains that solve \eqref{eq:diffeq of a1} and \eqref{eq:diffeq for b2}. We denote these by $a\up1\cand$ and $b\up2\cand$, for \emph{candidate solutions}.
\item[Step 2] Find \emph{local adjustments} needed to make the candidate solutions satisfy the necessary local conditions. These local adjustments are global cocycles $a\adj$ and $b\adj$ such that $a\up1=a\up1\cand+a\adj$ and $b\up2=b\up2\cand+b\adj$ satisfy the desired local properties. 
\end{description}
As we described above, we do not compute with cochains directly, but rather with their $S$-units associated by Kummer theory. For the purposes of this introduction, we will abuse notation and conflate these two.

\subsubsection{Computing candidate solutions}
To compute our candidate solutions, we start with an alternate interpretation of the equations \eqref{eq:diffeq of a1} and \eqref{eq:diffeq for b2}. A solution $a\up1\cand$ to \eqref{eq:diffeq of a1} gives a twisted-Heisenberg extension of $\Q$, which is cut out by the following upper-triangular 3-dimensional representation of $G_\Q$: 
\begin{equation}
    \label{eq: 3d abc}
\begin{pmatrix}
\omega & b\up1 & \omega a\up1\cand \\
0 & 1 & \omega c\up1 \\
0 & 0 & \omega
\end{pmatrix}.
\end{equation}
In his thesis, Sharifi gave a way to find such extensions using Kummer theory \cite{sharifithesis}. The idea is to start by interpreting $c\up1$ as a unit in $\Q(\zeta_p)$. Since $b\up1 \cup c\up1$ vanishes in cohomology, the Hasse norm theorem implies that $c\up1$ is a norm from $K$; let $\gamma \in K^\times$ be such that $N_{K/\Q(\zeta_p)}\gamma=c\up1$. Then Sharifi proves that $a\up1\cand$ can be obtained as a kind of Kolyvagin derivative of $\gamma$ \cite[Proposition 2.6]{sharifithesis}.

There is a similar interpretation of \eqref{eq:diffeq for b2}. Namely, a solution to \eqref{eq:diffeq for b2} gives a twisted-Heisenberg extension of $\Q$ one dimension greater, cut out by
\begin{equation}
\label{eq: 4d abcd}
\begin{pmatrix}
\omega & b\up1 & \omega a\up1\cand & b\up2\cand\\
0 & 1 & \omega c\up1 &  d\up1\cand\\
0 & 0 & \omega & b\up1 \\
0 & 0 & 0 & 1
\end{pmatrix},
\end{equation}
where we define $d\up1\cand=b\up1 c\up1-a\up1\cand$. The obstruction to the existence of a cochain $b\up2\cand$ is a generalization of the cup product called the \emph{triple Massey product} $(b\up1,c\up1,b\up1)$. With this notation, \eqref{eq:diffeq for b2} can be rewritten as 
\[
-d b\up2 = (b\up1,c\up1,b\up1).
\]
Sharifi generalized his results from cup products to higher \emph{cyclic} Massey products \cite{sharifi2007,LLSWW2020}, which include triple Massey products of the form $(b\up1,b\up1,c\up1)$, but not $(b\up1,c\up1,b\up1)$. The result is that a solution to the equation
\[
-dZ = (b\up1,b\up1,c\up1)
\]
can be obtained as a second Kolyvagin derivative of $\gamma$ \cite[Thm.\ 4.3]{sharifi2007} (where, as above, $\gamma \in K^\times$ satisfies $N_{K/\Q(\zeta_p)}\gamma=c\up1$). 
We show, using commutativity relations for Massey products, that a solution $b\up2\cand$ can be derived from such a $Z$.

\subsubsection{Computing local adjustments}
There are two types of local adjustments that need to be made:
\begin{itemize}
    \item local adjustments at $p$ to ensure finite-flatness,
    \item local adjustments at $\ell_0$, used in defining $\alpha$ and $\beta$.
\end{itemize}
For $a\up1$, the finite-flat condition translates to the extension $K'/K$
being unramified at $p$. This can be achieved by multiplying $a\up1\cand$ by an appropriate $p$th root of unity. For $b\up2$, the finite-flat condition boils down to a \emph{peu ramife\'e} condition as in \cite[\S2.4]{serre1987}. If $K/\Q$ is tamely ramified at $p$, this amounts to $b\up2\cand$ being prime-to-$p$ (as an $S$-unit), which is automatic by our Kolyvagin derivative construction. If $K/\Q$ is wildly ramified, the condition is slightly more involved, but can be achieved by multiplying $b\up2\cand$ by an appropriate power of $p$.

The local adjustment at $\ell_0$ is more interesting because it involves the cocycle $c\up1$. Unlike $b\up1$, the splitting field of $c\up1$ is not easy to write down. Even if we do compute it, checking the condition globally would involve working with the compositum of $K$ and the splitting field of $c\up1$. Instead, we take advantage of the fact that the local condition only involves the local restriction $c\up1|_{\ell_0}$, not the global cocycle. The structure of this local restriction is known by a tame version of the Gross--Stark conjecture \cite{wake2020eisenstein}. This result computes slope of $c\up1|_{\ell_0}$ (with respect to a canonical basis of $H^1(\Q_{\ell_0},\F_p(-1))$) in terms of an analytic invariant called the \emph{Mazur--Tate derivative} $\zeta_\mathrm{MT}'$ (of a family of Dirichlet $L$-functions). The quantity $\zeta_\mathrm{MT}'$ is completely explicit and easily computed, so this allows us to explicitly compute $c\up1|_{\ell_0}$ up to scalar, and find the adjustments purely locally.

\subsection{Organization of the paper} 
In Section \ref{sec: recollect part 1}, we summarize the relevant constructions from Part I; this section also contains a subsection (\S\ref{subsec: finite-flat}) of new material that focuses on formulating the finite-flat condition for $b\up2$ in language that is explicit enough for computations. In Section \ref{sec: sharifi}, we construct our candidate solutions, including the background material required to apply Sharifi's methods in our setting. In Section \ref{sec:adjustments}, we prove the main result of this paper (Theorem \ref{thm: main}) and give explicit algorithms for checking whether the splitting conditions in Theorem \ref{thm: main} hold. In Section \ref{sec:data}, we present a broad selection of computed examples that illustrate our main result. Lastly, in Section \ref{sec: P2 ANT}, we establish the characterizations of $K'/K$ and $K''/K$ that appear in \S\ref{subsec: P2 main results} as part of this paper's main result; we also use this content to state a terminal result (Theorem \ref{thm: HL P2}) of the pair of papers. 

\subsection{Acknowledgements}
The first-named author thanks the University of Bristol and the Heilbronn Institute for Mathematical Research for its partial support of this project. The second-named author was supported in part by NSF grant DMS-1901867, and would like to thank his coauthors on the paper \cite{LLSWW2020}; that paper inspired many of the ideas used here about how to compute Massey products. The third-named author was supported in part by Simons Foundation award 846912, and thanks the Department of Mathematics of Imperial College London for its partial support of this project from its Mathematics Platform Grant. We also thank John Cremona for several insightful conversations about computational aspects of this project as well as the referee for their helpful comments and suggestions. This research was supported in part by the University of Pittsburgh Center for Research Computing and Swarthmore College through the computing resources provided.

\section{Key Galois cochains and their local properties}\label{sec: recollect part 1}

The purpose of this section is to recall the constructions of Part I as a point of departure. In \S\ref{subsec: finite-flat}, there will also be some new content that makes these constructions more amenable to computation. We begin with notation and conventions to make the Galois cochains featured in \S\ref{subsec:diffeq and R} precise. 

\subsection{Assumptions, notation, and conventions}

The following statement of assumptions, which are in force throughout this paper, recapitulates the assumptions given at the outset of \S\ref{subsec: summary2} in a slightly more precise form. We write 
\[
\bT \rsurj \bT_{\ell_0} 
\]
for the surjection of Hecke algebras, from level $N$ to $\ell_0$, as described in Part I, \S\ref{P1: subsec: modular forms}.

\begin{assump}
\label{assump: main}
Assume $p\geq 5$ is prime. Throughout the paper, we specialize to level $N=\ell_0\ell_1$, where $\ell_0$ and $\ell_1$ are primes such that
\begin{enumerate}
    \item $\ell_0 \equiv 1 \pmod{p}$
    \item  $\ell_1 \not \equiv 0, \pm 1 \pmod{p}$ 
    \item $\ell_1$ is a $p$th power modulo $\ell_0$
    \item $\rk_{\Z_p}\bT_{\ell_0} = 2$, which is easily checked via a criterion Merel \cite{merel1996}, c.f., Part I, Remark \ref{P1: rem: Merel}
\end{enumerate}
\end{assump}

The assumption that $\rk_{\Z_p} \bT_{\ell_0} = 2$ is equivalent to the assumption that there is a unique Eisenstein-congruent cusp form at level $\ell_0$, which is the assumption we used in the introduction (\S\ref{subsec: summary2}). 

Actions of and functions on profinite groups are presumed continuous without further comment. This includes cochains on Galois groups, which are thought of as functions on a finite self-product of the group, which we will establish notation for shortly. 

For $q = \ell_0, \ell_1, p$, we fix embeddings of algebraic closures $\oQ \rinj \oQ_q$, inducing inclusions of decomposition groups $G_q := \Gal(\oQ_q/\Q_q) \rinj G_{\Q,Np}$, where $G_{\Q,Np}$ is the Galois group of the maximal algebraic extension of $\Q$ ramified only at places dividing $Np\infty$. Write $I_q \subset G_q$ for the inertia subgroup, and $\Fr_q \in G_q$ for a lift of the arithmetic Frobenius element of $G_q/I_q$ to $G_q$.

A primitive $p$th root of unity $\zeta \in \oQ$ plays an important role by inducing an isomorphism between the group of $p$th roots of unity $\mu_p \subset \oQ^\times$ and $\F_p(1)$, which we write for the representation of $G_{\Q,Np}$ on the 1-dimensional $\F_p$-vector space $\F_p$ with action by the modulo $p$ cyclotomic character $\omega$. Likewise, $\zeta$ induces isomorphisms $\mu_p^{\otimes i} \risom \F_p(i)$ for all $i \in \Z$, where $\F_p(i)$ denotes $\F_p(1)^{\otimes i}$. 

\begin{defn}
Let $i \in \Z$ and $j \in \Z_{\geq 0}$. We use $C^j(\Z[1/Np], \F_p(i))$ to denote $j$-cochains of $G_{\Q,Np}$ valued in $\F_p(i)$. Likewise, when $Z^j$, $B^j$, or $H^j$ replaces ``$C^j$'' in this notation, we are referring to cocycles, coboundaries, and cohomology, respectively. When $Y \in Z^j(-)$ is a cocycle, we write $[Y] \in H^j(-)$ for its associated cohomology class. 

Similarly, when $q$ is a prime, we write $C^j(\Q_q, \F_p(i))$ for local cochain groups. We have restriction maps 
\[
(-)\vert_q : H^j(\Z[1/Np], \F_p(i)) \to H^j(\Q_q, \F_p(i)), \quad 
C^j(\Z[1/Np], \F_p(i)) \to C^j(\Q_q, \F_p(i))
\]
for $q = \ell_0, \ell_1, p$. We say that a cohomology class (resp.\ cochain) is 
\begin{itemize}
    \item \emph{unramified at $q$} when its further restriction to $H^j(\Q^\mathrm{ur}_q, \F_p(i))$ \\(resp.\ $C^j(\Q^\mathrm{ur}_q, \F_p(i))$) vanishes, and
    \item \emph{splits at $q$} when it vanishes under these map.
\end{itemize} 

We have cup product maps
\begin{equation}\label{eq: cup products}
\begin{split}
\smile &: C^j(-, \F_p(i)) \times C^{j'}(-, \F_p(i')) \to C^{j+j'}(-, \F_p(i+i')),\\
\cup &: H^j(-, \F_p(i)) \times H^{j'}(-, \F_p(i')) \to H^{j+j'}(-, \F_p(i+i')).
\end{split}
\end{equation}
\end{defn}

\subsection{Pinning data}
\label{subsec: cocycles} Because many of the constructions in Part I depend in subtle ways on additional choices that we refer to as \emph{pinning data}, we recall this information here.

\begin{defn}
    \label{defn: pinning} 
    We refer to the following choices a \emph{pinning data}. 
    \begin{itemize}
        \item For each $q \in \{\ell_0, \ell_1, p\}$, an embedding $\oQ \rinj \oQ_q$
        \item a primitive $p$th root of unity $\zeta_p \in \oQ$
        \item for $i=0,1$, a $p$th root $\ell_i^{1/p} \in \oQ$ of $\ell_i$, such that, if possible, the image of $\ell_1^{1/p}$ in $\oQ_p$, under the fixed embedding, is in $\Q_p$. (See Lemma \ref{lem: b1 ramification at p} for when this is possible.)
    \end{itemize}
\end{defn}

 We fix a choice of pinning data for the entire paper. Notice that our choice defines a decomposition subgroup of $q$ in $G_{\Q,Np}$ for each prime $q$ dividing $Np$ as well as an isomorphism between this subgroup and $G_q$. Likewise, for any number field $F \subset \oQ$, the induced embedding $F \rinj \oQ_q$ singles out a prime of $F$ lying over $q$, which we call the \emph{distinguished prime of $F$ over $q$}. 

We are now ready to write down precise descriptions of the key Galois cochains from Part I, overviewed in \S\ref{subsec:diffeq and R}. Indeed, recall the canonical isomorphism 
\begin{equation}
    \label{eq: kummer}
    \Z[1/Np]^\times \otimes_\Z \F_p \isoto H^1(\Z[1/Np],\mu_p)
\end{equation}
of Kummer theory, which sends an element $n \in \Z[1/Np]^\times \otimes_\Z \F_p$ to the class of the cocycle $G_{\Q,Np} \ni \sigma \mapsto \frac{\sigma (n^{1/p})}{n^{1/p}}$ for a choice $n^{1/p} \in \oQ$ of $p$th root of $n$. We call this element of $H^1(\Z[1/Np],\mu_p)$ the \emph{Kummer class of $n$} and call any cocycle in this class a \emph{Kummer cocycle of $n$}. Because $\zeta_p \not\in \Q$, each Kummer cocycle of $n$ is given by $\sigma \mapsto \frac{\sigma (n^{1/p})}{n^{1/p}}$ for a unique choice $n^{1/p} \in \oQ$ of $p$th root of $n$. We use the isomorphism $\F_p(1) \cong \mu_p$ induced by the choice of $\zeta_p$ in the pinning data (Definition \ref{defn: pinning}) to think of Kummer classes and cocycles as having coefficients in $\F_p(1)$.

We fix the following cohomology classes and cocycles throughout the paper:
\begin{defn} \hfill
    \label{defn: pinned cocycles}
    \begin{itemize}
        \item Let 
        \[
         b_0 , b_1, b_p \in H^1(\Z[1/Np], \F_p(1))
        \]
        be the Kummer classes of $\ell_0$, $\ell_1$, and $p$, respectively.
        \item For $i=0,1$, let
        \[
        \gamma_0 \in I_{\ell_0}, \ \gamma_1 \in I_{\ell_1}
        \]
        be a lift of a generator of the maximal pro-$p$ quotient of the tame quotient of $I_{\ell_i}$ and fixed such that $b_i(\gamma_i)=1 \in \F_p(1)$.\footnote{Note that $b_i|_{I_{\ell_i}}: I_{\ell_i} \to \F_p(1)$ is a well-defined homomorphism because $I_{\ell_i}$ acts trivially on $\F_p(1))$.}
        \item Let
        \[
        b_0\up1, b_1\up1 \in Z^1(\Z[1/Np],\F_p(1))
        \]
        be the Kummer cocycles associated to $p$th roots $\ell_0^{1/p}$ and $\ell_1^{1/p}$, respectively, chosen in our pinning data (Definition \ref{defn: pinning}). Let $b\up1=b_1\up1$.

        \item Let 
        \[
        c\up1 \in Z^1(\Z[1/Np], \F_p(-1))
        \]
        with cohomology class $c_0=[c\up1]$ be the unique cocycle such that
    \begin{enumerate}[label=(\roman*)]
        \item $c_0$ is ramified exactly at $\ell_0$,
        \item $c_0|_p=0$,
        \item $c\up1(\gamma_0) = 1$, and
        \item $c\up1\vert_{\ell_1} = 0$.
    \end{enumerate}
    Note that properties (i)-(iii) specify $c_0$ uniquely, and property (iv) specifies $c\up1$ uniquely. See Part I, Definition \ref{P1: defn: pinned cocycles} for more details.
    \item Let 
    \[
    a_0, a_p \in Z^1(\Z[1/Np],\F_p)
    \]
    be non-zero homomorphisms ramified exactly at $\ell_0$ and at $p$, respectively, and such that $a_0(\gamma_0)=1$. This determines $a_0$ uniquely and determines $a_p$ up to $\F_p^\times$-scaling (which is sufficient for our purposes). 
    \end{itemize}
\end{defn}

We note that the choices of $b\up1$, $c\up1$, and $a_0$ depend only on the pinning data of Definition \ref{defn: pinning}. Additionally, while we define all of the cohomology classes and cocycles appearing in the contructions of Part I for the sake of completeness, the content of this paper primarily requires the cocycles $b\up1,c\up1$, and $a_0$.

\subsection{The solution $a\up1$ to differential equation \eqref{eq:diffeq of a1} and the invariant $\alpha$}

We move on to Galois cochains that are not cocyles, but satisfy differential equations whose origin was discussed in \S\ref{subsec:diffeq and R}. Then, the crucial numerical invariant $\alpha^2 + \beta \in H^0(\Q_{\ell_0},\F_p(2))$ will be derived from them, which is proven to be canonical -- that is, independent of the pinning data and therefore only dependent on a choice of $p$ and $N$ satisfying Assumption \ref{assump: main} -- in Part I, Theorem \ref{thm: main invariant}.

By definition of $c\up1$, its associated cohomology class $c_0$ splits at $p$. This means that there exists some $x_{c} \in \F_p(-1)$ such that
\[
c\up1\vert_p = dx_{c}.
\]
Concretely, for any $\tau \in G_p$, we have $c_0\up1(\tau) = (\omega^{-1}(\tau)-1)x_{c}$. This $x_{c}$ is determined by the pinning data. 

We will use $x_{c}$ to normalize the solution to our first differential equation \eqref{eq:diffeq of a1}, which is 
\[
-dX = b\up1 \smile c\up1,
\]
where $X$ is an unknown in $C^1(\Z[1/Np], \F_p)$. 

In the statement of Proposition \ref{prop: produce a1} below, we use the fact that a primitive $p$th root of unity in $\Q_{\ell_0}$ exists becuase $\ell_0 \equiv 1 \pmod{p}$ and supplies a $\F_p$-basis for $H^0(\Q_{\ell_0}, \F_p(1))$ under the natural isomorphism $\mu_p(\Q_{\ell_0}) \cong H^0(\Q_{\ell_0}, \F_p(1))$. 
\begin{prop}[Part I, Lemma \ref{P1: lem: produce a1}]
\label{prop: produce a1}
There exists a unique solution $a\up1 \in C^1(\Z[1/Np],\F_p)$ of the differential equation \eqref{eq:diffeq of a1} that satisfies the local conditions
\begin{enumerate}[label=(\alph*)]
\item $(a\up1 +b_1\up1 \smile x_{c})|_{I_p} = 0$ in $Z^1(\Q_p^\mathrm{nr}, \F_p)$
\item $a\up1|_{\ell_0}$ is on the line in $Z^1(\Q_{\ell_0}, \F_p) \cong H^1(\Q_{\ell_0}, \F_p)$ spanned by $\zeta \cup c_0|_{\ell_0}$ under the cup product
\[
H^0(\Q_{\ell_0}, \F_p(1)) \times H^1(\Q_{\ell_0}, \F_p(-1)) \to H^1(\Q_{\ell_0}, \F_p),
\] 
where $\zeta$ is a choice of $\F_p$-basis of $H^0(\Q_{\ell_0}, \F_p(1))$.
\end{enumerate}
This cochain $a\up1$ is uniquely determined by the pinning data. 
\end{prop}

Condition (a) is a finite-flat condition at $p$, as discussed in Part I, \S\ref{P1: subsubsec: finite-flat}. 

\begin{rem}
    We point out a phenomenon that will frequently appear in various forms. Even though $a\up1$ is a global cochain that is not a global cocycle, its restriction $a\up1\vert_{\ell_0}$ to $G_{\ell_0}$ is a cocycle because the factor $b\up1$ of the differential equation splits at $\ell_0$. We will often encounter situations where a global non-cocycle is a local cocycle, allowing us to impose arithmetic conditions on it. 
\end{rem}

We now define an important numerical invariant.
\begin{defn}
\label{defn: alpha}
Let $\alpha \in \mu_p(\Q_{\ell_0}) \cong H^0(\Q_{\ell_0}, \F_p(1))$ be the unique solution to
\[
[a\up1|_{\ell_0}] = \alpha \cup c_0|_{\ell_0}.
\]
By Proposition \ref{prop: produce a1}, $\alpha$ depends only on the pinning data of Definition \ref{defn: pinning}. 
\end{defn}

\subsection{The solution $b\up2$ to differential equation \eqref{eq:diffeq for b2} and the invariant $\beta$}

In preparation to discuss the differential equation solved by $b\up2$, we recall the following well-known analogue, for general odd primes $p$, of the ramification behavior of the prime $2$ in quadratic (degree $2$) number fields. As far as notation, recall that $b_1$ denotes the Kummer class of $\ell_1$, which contains the Kummer cocycle $b\up1$.
\begin{lem}[Part I, Lemma \ref{P1: lem: ram of ell1 at p}]
    \label{lem: b1 ramification at p}
    The following conditions are equivalent. 
    \begin{enumerate}
        \item $\ell_1^{p-1} \equiv 1 \pmod{p^2}$
        \item $b_1$ is unramified at $p$
        \item $p$ is tamely ramified in $\Q(\ell_1^{1/p})/\Q$.
    \end{enumerate}
\end{lem} 

We recall differential equation \eqref{eq:diffeq for b2}, which has the form
\[
-dY = a\up1 \smile b\up1 + b\up1 \smile d\up1,
\]
where we have let $d\up1 = b\up1c\up1 - a\up1$. 

\begin{prop}[Part I, Lemma \ref{P1: lem: exists 2nd-order 1-reducible} and Proposition \ref{P1: prop: exists unique beta}]
\label{prop: produce b2}
There exists a solution $b\up2 \in C^1(\Z[1/Np],\F_p(1))$ of the differential equation \eqref{eq:diffeq for b2} if and only if $a\up1\vert_{\ell_1} = 0$ in $Z^1(\Q_{\ell_1}, \F_p)$. If any such solution exists, then there also exists a solution $b\up2$ satisfying the local conditions that 
\begin{enumerate}[label=(\alph*)]
\item $b\up2|_{\ell_0}$ is a cocycle on the line in $Z^1(\Q_{\ell_0}, \F_p(1)) \cong H^1(\Q_{\ell_0}, \F_p(1))$ spanned by $\zeta' \cup c_0|_{\ell_0}$ under the cup product
\[
H^0(\Q_{\ell_0}, \F_p(2)) \times H^1(\Q_{\ell_0}, \F_p(-1)) \to H^1(\Q_{\ell_0}, \F_p(1)),
\] 
where $\zeta'$ is a choice of $\F_p$-basis of $H^0(\Q_{\ell_0}, \F_p(2))$
\item there exists some $\rho_2$ as in \eqref{eq:rho2} such that $\rho_2\vert_p$ is finite-flat and $b\up2$ is a coordinate of $\rho_2$ (as in \eqref{eq:rho2}).
\end{enumerate}
There exists a single $\beta \in H^0(\Q_{\ell_0}, \F_p(2)) \cong \F_p(2)$ such that, for any of the $b\up2$ satisfying these local conditions, 
\[
b\up2|_{\ell_0} = \beta \smile c\up1|_{\ell_0}.
\]
Moreover, the set of solutions of \eqref{eq:diffeq for b2} satisfying these two local conditions is contained in a torsor under the subspace of $Z^1(\Z[1/Np], \F_p(1))$ spanned by coboundaries $B^1(\Z[1/Np], \F_p(1))$ and the cocycle $b\up1$. 
\end{prop}

\begin{proof}
The first claim of the proposition is exactly Lemma \ref{lem: a1 vanish at ell1 when S large} of Part I. 

Applying that first claim, the existence of a solution $b\up2$ of \eqref{eq:diffeq for b2} is sufficient, by Proposition \ref{P1: prop: exists unique beta}, to deduce that 
\begin{itemize}
    \item there exists a $\rho_2$ as in \eqref{eq:rho2} that is finite-flat at $p$, meaning that its $b\up2$-coordinate satisfies (b); and 
    \item its $b\up2$-coordinate also satisfies (a), according to Lemma \ref{P1: lem: exists 2nd-order 1-reducible}(2) of Part I.
\end{itemize}
Finally, the claim about containment in a torsor is in Part I, Proposition \ref{P1: prop: exists unique beta}. 
\end{proof}

\begin{defn}
    \label{defn: beta} Whenever a deformation $\rho_2$ of $\rho_1$ exists, i.e., when $a\up1\vert_{\ell_1} = 0$ by Proposition \ref{prop: produce b2}, we define $\beta \in H^0(\Q_{\ell_0},\F_p(2))$ to be the unique element (depending only on the pinning data) that satisfies 
    \[
    b\up2|_{\ell_0} = \beta \smile c\up1|_{\ell_0}.
    \]
    for any choice of $b\up2$ satisfying the local conditions (a) and (b) Proposition \ref{prop: produce b2}.
\end{defn}

\subsection{Local slope and a zeta-value}\label{subsec: local slope and a zeta-value}
The purpose of this section is to relate the class of the cocycle $c\up1$ to a Mazur--Tate $\zeta$-function $\xi$, following \cite[\S8.3]{wake2020eisenstein}.
This formula expresses the slope of the line spanned by the global class $c\up1$ in the 2-dimension local Galois cohomology group $H^1(\Q_{\ell_0},\F_p(-1))$ as a ratio of classical and tame $L$-values. This can be thought of as a tame analog of the ($p$-adic) Gross--Stark formula.

We consider the basis $\{a_0|_{\ell_0}, \lambda\}$ of $H^1(\Q_{\ell_0},\F_p)$, where $a_0$ is as in Defintion \ref{defn: pinned cocycles} and $\lambda$ is the unique unramified character sending $\Fr_{\ell_0}$ to 1. We also let $\zeta_\mathrm{MT}' \in \F_p$ denote the element
\begin{equation}\label{eq: mazur-tate derivative}
\zeta_\mathrm{MT}' = \frac{1}{2} \sum_{i=1}^{\ell_0-1} B_2(i)\log_{\ell_0}(i),
\end{equation}
where $B_2(x)$ is the second Bernoulli polynomial.
This element can be seen as the derivative of the Mazur-Tate type $L$-funciton $\chi \mapsto L(\chi,1)$ for Dirichlet characters $\chi$ of modulus $N$ and $p$-power order (see \cite[\S1.5]{WWE3} or \cite[\S4.1]{wake2020eisenstein}). Note that we have $\zeta_\mathrm{MT}' \ne 0$ due to our assumption that there is a unique cusp form of level $\ell_0$ that is congruent to the Eisenstein series modulo $p$, according to \cite[Thm.\ 1.5.2]{WWE3}. 

\begin{rem}
We remark that this is closely related to Merel's result \cite[Thm.\ 2]{merel1996} characterizing the uniqueness of the cusp form in terms of \emph{Merel's number}, defined in Part I, \S\ref{sssec: intro interpretation}, equation \eqref{eq: Merel number}. Indeed, $\zeta_\mathrm{MT}'$ vanishes if and only if \emph{Merel's number} does; a direct link between the two quantities was proved in \cite[Lem.\ 12.3.1]{WWE3} (see also \cite{lecouturier2018}). 
\end{rem}

We have the following description of the line spanned by $c\up1 |_{\ell_0}$.

\begin{thm}[{\cite[Prop.\ 8.3.2]{wake2020eisenstein}}]
\label{thm:slope}
Let $\zeta \in H^0(\Q_{\ell_0}, \F_p(1)) \cong \F_p(1)$ be a basis. The line in $H^1(\Q_{\ell_0},\F_p)$ spanned by $\zeta \smile c\up1|_{\ell_0}$ contains the element
\[
\zeta_\mathrm{MT}' \lambda + \frac{1}{6} a_0|_{\ell_0}.
\]
\end{thm}

This is computationally significant because $\zeta_\mathrm{MT}'$, $\lambda$, and $a_0|_{\ell_0}$ are easy to compute even though $c\up1$ is not.

\subsection{Explicit formulation of the finite-flat condition on $b\up2\vert_p$}\label{subsec: finite-flat}

In this section, we shift from recollections to new content. It is readily apparent that the finite-flat condition of Proposition \ref{prop: produce b2} is too inexplicit for computation; the goal of this section is to remedy that problem.

For orientation, we recall the ``basis change'' argument of [Part I, \S\ref{P1: subsec: finite-flat on rho2}], which is a first step in making the finite-flat condition testable using Kummer theory. We let $\rho'_2$ be a conjugate of $\rho_2$ (as in \eqref{eq:rho2}) by  generalized matrix such that $\rho'_2\vert_p$ is upper-triangular. This conjugation was achieved with a lower-triangular generalized matrix of the form
\[
\ttmat{1}{0}{x_c + \epsilon x'_c}{1},
\]
where $x_c$ was prescribed in Proposition \ref{prop: produce a1}. We write $\rho'_2$ using the coordinate system
\[
\rho_2'=\ttmat{\omega (1+a\upp1\epsilon + a\upp2\epsilon^2) }{b\upp1 + b\upp2 \epsilon}{\omega(c\upp1 + c\upp2 \epsilon)}{1+d\upp1 \epsilon + d\upp2\epsilon^2},
\]
and, from Proposition \ref{prop: produce a1}, we have an unramified homomorphism
\[
a\upp1\vert_p = (a\up1 + b\up1 \smile x_c\big)\vert_p : G_p \to \F_p.
\]
Also, we have $b\upp2 = b\up2$ and $b\upp1=b\up1$ because the conjugation is lower-triangular. 

We have the following simplified form of $\rho'_2\vert_p$, namely,
\[
\rho'_2\vert_p = \ttmat{\omega (1+a\upp1\epsilon + a\upp2\epsilon^2) }{b\up1 + b\up2 \epsilon}{0}{1-a\upp1 \epsilon + d\upp2\epsilon^2}\bigg\vert_p.
\]
The constant determinant property implies that
\[
d\upp1 = -a\upp1, \qquad d\upp2 = (a\upp1)^2-a\upp2. 
\]
We observe that the differential equation imposed on $b\upp2 = b\up2$ by the homomorphism property of $\rho'_2$ is
\[
-db\upp2 = a\upp1 \smile b\upp1 + b\upp1 \smile d\upp1,
\]
which simplifies upon restriction to $G_p$ as
\[
-db\up2\vert_p = a\upp1\vert_p \smile b\up1\vert_p + b\up1\vert_p \smile -a\upp1\vert_p. 
\]

In Part I, we translate the finite-flat property of the GMA representation $\rho'_2\vert_p$ to a typical (matrix-valued) representation. For convenience, we write 
\[
\chi_2 := (1 + \ep a\upp1 + \ep^2 a\upp2)\vert_p : G_p \to \F_p[\ep]/(\ep^3)^\times, \text{ and } \chi_1 = 1 + \ep a\upp1\vert_p : G_p \to \F_p[\ep_1]^\times,
\]
so that $\chi_1 = (\chi_2 \mod \ep^2)$ and 
\[
\rho_2'\vert_p = \ttmat{\omega \chi_2}{b\up1 + \ep b\up2}{0}{\chi_2^{-1}}. 
\]
We also establish notation for the homomorphism 
\begin{equation}
\label{eq: P2 form eta1}
\eta_1 = \ttmat{\omega(1 + \ep a\upp1)}{b\up1 + \ep b\up2}{\ep c\upp1}{1 + d\upp1} : G_{\Q,Np} \to \GL_2(\F_p[\ep]/(\ep^2))
\end{equation}
assembled from the coordinates of $\rho_2'$, as in [Part I, Def.\ \ref{defn: etas}]. Note that $\eta_1 \vert_p$ is upper-triangular, extending $\chi_1^{-1}$ by $\omega\chi_1$, since we have arranged that $c\upp1\vert_p = 0$. 

\begin{prop}[Part I, Lemma \ref{lem: flat rho2 step 2}]
    \label{P2: prop: ref flat rho2 step 2}
    Assume the deformation $\rho_2$ of $\rho_1$ exists. Then $\rho_2$ is finite-flat at $p$ if and only if both the following (a) and (b) hold. 
    \begin{enumerate}[label=(\alph*)]
        \item the homomorphism $\eta_1$ is finite-flat at $p$ 
    \item $\chi_2 : G_p \to \F_p[\ep_2]$ is unramified. 
    \end{enumerate}
    Moreover, if $\rho_2$ satisfies (a), then there exists some deformation $\rho_{2,\mathrm{new}}$ of $\rho_1$ that is finite-flat at $p$ and such that $\eta_1$ equals the $\eta_{1, \mathrm{new}}$ formed from $\rho_{2,\mathrm{new}}$ via \eqref{eq: P2 form eta1}. 
\end{prop}

\begin{proof}
The first claim is the equivalence $(1) \Rightarrow (3)$ of Part I, Lemma \ref{lem: flat rho2 step 2}. The second claim follows from the proof of Part I, Lemma \ref{lem: flat rho2 step 4}: the proof shows that once we have a $\rho_2$ whose induced $\eta_1$ is finite-flat at $p$, then it is possible to adjust at most the $a\up2$ and $d\up2$-coordinates to form $\rho_{2,\mathrm{new}}$ from $\rho_2$ such that $\rho_{2,\mathrm{new}}$ is finite-flat at $p$ and such that $\eta_{1,\mathrm{new}}$ (assembled from the coordinates of $\rho_{2,\mathrm{new}}$) equals $\eta_1$. 
\end{proof}

Next, we want to understand when $\eta_1\vert_p : G_p \to \GL_2(\F_p[\ep]/(\ep^2))$, or equivalently the extension class $B := [b\up1 + \ep b\up2] \in \Ext^1_{\F_p[\ep_1][G_p]}(\chi_1^{-1}, \omega\chi_1)$, is finite flat. Indeed, note that $\chi_1^{-1}$ and $\omega\chi_1$ are finite-flat because $\chi_1$ is unramified. 
This issue is especially straightforward when $a\upp1\vert_p = 0$, making $\chi_1$ trivial and $b\up2\vert_p$ a cocycle. Otherwise, $a\upp1\vert_p$ cuts out the unique unramified cyclic degree $p$ extension $\Q_{p^p}/\Q_p$, and $b\up2\vert_p$ is not a cocycle on $G_p$. In order to apply Kummer theory in the latter case, we restrict $b\up2$ to the absolute Galois group of $\Q_{p^p}$. We write
\[
b\up2\vert_{p^p} \in Z^1(\Q_{p^p}, \F_p(1)), \text{ so that } [b\up2\vert_{p^p}] \in H^1(\Q_{p^p}, \F_p(1)) \cong \Q_{p^p}^\times/(\Q_{p^p}^\times)^p,
\]
where the final isomorphism comes from Kummer theory.
\begin{prop}
    \label{prop: b2 f-f unramified test}
    The finite-flatness of $\eta_1\vert_p$ is characterized by Kummer theory.
    \begin{enumerate}[label = (\alph*)]
        \item When $\chi_1 = 1$, $\eta_1$ is finite-flat at $p$ if and only if the class 
        \[
        [b\up2\vert_p] \in H^1(\Q_p, \F_p(1)) \cong \Q_p^\times/(\Q_p^\times)^p
        \]
        of the cocycle $b\up2\vert_p$ is in the line spanned by the Kummer class of $1+p$. 
        \item When $\chi_1 \neq 1$, $\eta_1$ is finite-flat at $p$ if and only if the class
        \[
        [b\up2\vert_{p^p}] \in H^1(\Q_{p^p}, \F_p(1)) \cong \Q_{p^p}^\times/(\Q_{p^p}^\times)^p
        \]
        is in the subspace spanned by the subgroup of units in $\Z_{p^p}^\times$ that are $1 \pmod{p}$. 
    \end{enumerate}
\end{prop}

\begin{proof}
    Simple application of Part I, Lemma \ref{P1: lem: local Kummer finite-flat}. 
\end{proof}

While the condition of Proposition \ref{prop: b2 f-f unramified test} is arguably explicit, our application requires the finite-flatness of $\eta_1$ to be tested using Kummer theory in the extension field of $\Q_p$ \emph{cut out} by $b\up1$, meaning that it is the minimal extension $A/\Q_p$ such that $b\up1\vert_{G_A} = 0$. It is important to understand how this condition behaves, which we describe in this lemma with proof omitted. 
\begin{lem}
Let $x \in Z^1(\Q_p, \F_p(i))$ and assume $(p-1) \nmid i$. If $x \in B^1(\Q_p, \F_p(i))$ is non-zero, $\Q_p(\zeta_p)/\Q_p$ is the extension cut out by $x$. Otherwise, letting 
\[
\rho_x = \ttmat{\omega^i}{x}{0}{1} : G_p \to \GL_2(\F_p), 
\]
the finite extension of $\Q_p$ cut out by $x$ is the subfield of $\oQ_p^{\ker \rho_x}$ that is fixed by the image of $\sm{\omega^i}{}{}{1}$ under the isomorphism $\Gal(\oQ_p^{\ker \rho_x}/\Q_p) \cong \mathrm{image}(\rho_x)$. 
\end{lem}

\begin{lem}
The extension of $\Q_p$ cut out by $b\up1\vert_p$ falls into two isomorphism classes, 
\[
\Q_p(\ell_1^{1/p}) \simeq 
\left\{
\begin{array}{ll}
\Q_p & \text{ if } \ell_1^{p-1} \equiv 1 \pmod{p^2} \\
\Q_p((1+p)^{1/p}) & \text{ if } \ell_1^{p-1} \not\equiv 1 \pmod{p^2}
\end{array}
\right.
\]
where the latter is totally ramified over $\Q_p$.
\end{lem}

\begin{proof}
To prove the claim, we apply the condition for the ramification of $b_1$ determined in Lemma \ref{lem: b1 ramification at p}, and note that $b_1$ is ramified at $p$ if and only if it is non-trivial at $p$.

First we address the ramified case. We recall from [Part I, Lem.\ \ref{P1: lem: local Kummer finite-flat}] that $H^1(\Q_p, \F_p(1))^\fl$ is 1-dimensional, being spanned by the Kummer class of $1+p$. Therefore, because $b_1\vert_p \in H^1(\Q_p, \F_p(1))$ is always contained in the finite-flat subspace, the extension of $\Q_p$ cut out by $\sm{\omega}{b\up1}{0}{1}$ is $\Q_p(\zeta_p, (1+p)^{1/p})$. By Galois theory, all of its subfields of degree $p$ over $\Q_p$ are mutually isomorphic, and therefore isomorphic to $\Q_p((1+p)^{1/p})$.

In the unramified case, our restriction on $b\up1$ in Definition \ref{defn: pinned cocycles} implies that $b\up1\vert_p = 0$, making it cut out the trivial extension of $\Q_p$. 
\end{proof}

When $\Q_p(\ell_1^{1/p}) = \Q_p$, the finite-flatness of $\eta_1$ can be tested just as easily as in Proposition \ref{prop: b2 f-f unramified test}. However, when $\Q_p(\ell_1^{1/p})/\Q_p$ is totally ramified, the test becomes more difficult: Lemma \ref{P1: lem: local Kummer finite-flat} of Part I does not apply over ramified extensions of $\Q_p$. This difficulty can be circumvented when $a\upp1\vert_p = 0$, because $b\up2\vert_p$ is a cocycle. Hence the remaining difficulty is concentrated in the case where
\[
\tag{$\star$} a\upp1\vert_p \neq 0 \text{ and } b\up1\vert_{I_p} \neq 0. 
\] 

The following proposition, which ends up being an exercise in Kummer theory, addresses this most difficult case $(\star)$. Let $F := \Q_p((1+p)^{1/p})$ and fix an isomorphism between $F$ and $\Q_p(\ell_1^{1/p})$. Let $\pi \in F$ denote a uniformizer. 
\begin{lem}\label{lem: b1 ramified f-f test}
In case $(\star)$, the following conditions are equivalent. 
\begin{enumerate}
        \item $[b\up2\vert_{\Q_{p^p}}] \in H^1(\Q_{p^p}, \F_p(1))$ corresponds to a $p$-unit under the isomorphism $H^1(\Q_{p^p}, \F_p(1)) \cong \Q_{p^p}^\times/(\Q_{p^p}^\times)^p$ of Kummer theory
        \item $[b\up2\vert_{F}] \in H^1(F, \F_p(1))$ corresponds to a $\pi$-unit that is $1 \pmod{\pi^2}$ under the isomorphism 
        \[
        H^1(F, \F_p(1)) \cong F^\times/(F^\times)^p
        \]
        of Kummer theory.
    \end{enumerate}
\end{lem}

\begin{proof}
    Let $M$ denote the composite field of $F$ and $\Q_{p^p}$, $M := \Q_{p^p}((1+p)^{1/p})$. We choose a uniformizer of $F$ and $M$, 
    \[
    \pi := (1+p)^{1/p} - 1 \in F. 
    \]
    We have a commutative diagram with exact rows and columns
    \[
    \xymatrix{
    & 0 \ar[d] & 0 \ar[d] \\ 
     & \langle1+p\rangle \ar[d] \ar[r]^\sim & \langle 1 + p \rangle \ar[d] \\
    0 \ar[r] & \Q_p^\times/(\Q_p^\times)^p \ar[r] \ar[d] & \Q_{p^p}^\times/(\Q_{p^p}^\times)^p \ar[d] \\
    0 \ar[r] & F^\times/(F^\times)^p \ar[r] & M^\times/(M^\times)^p
    }
    \]
    where both horizontal inclusions arise from taking invariants of an equivariant Galois action of $\Gal(\Q_{p^p}/\Q_p) \cong \Gal(M/F)$. 
    
    In case $(\star)$, the fact that $b\up2\vert_p$ is a cocycle upon restriction to both $G_F$ and $G_{\Q_{p^p}}$ means that the classes induced by these cocycles are sent, via Kummer theory, to elements $y \in \Q_{p^p}^\times/(\Q_{p^p}^\times)^p$ and $z \in F^\times/(F^\times)^p$ that map to the same element of $M^\times/(M^\times)^p$. Thus the $\Gal(\Q_{p^p}/\Q_p)$-action on $y$ is trivial. 
    
    In order to apply Proposition \ref{prop: b2 f-f unramified test}(b), it will be useful to know what the coordinates of $p$ are in the decomposition
    \[
     F^\times/(F^\times)^p \cong \langle \pi \rangle \oplus (1 + \pi \cO_F)/(1 + \pi \cO_F)^p. 
    \]
    The key calculation is the equality
    \[
    \frac{\pi^p}{p} = 1 + (1+p)^{1/p} - \frac{1}{p}\sum_{i=2}^{p-1} (-1)^i (1+p)^{i/p} \in \cO_F^\times, 
    \]
    which means that $p$ has trivial $\langle \pi\rangle$-part in the decomposition of $F^\times/(F^\times)^p$. Moreover, its image in $(1 + \pi \cO_F)/(1 + \pi \cO_F)^p$ is non-trivial modulo the image of $(1 + \pi^2 \cO_F)$. 
    
    Using the natural isomorphism
    \[
    \Q_{p^p}^\times/(\Q_{p^p}^\times)^p \cong \langle p \rangle \oplus (1 + p \Z_{p^p})/(1 + p \Z_{p^p})^p
    \]
    and the Galois-equivariant exp-log isomorphism $(1 + p \Z_{p^p}, \cdot) \cong (p\Z_{p^p}, +)$, we find that $y \in \langle 1+p\rangle$ if and only if its $\langle p\rangle$-coordinate, under the summand, vanishes. 
    
    Putting the above two facts together, we observe that the equivalence of (1) and (2) follows from proving that the pullback to $F^\times/(F^\times)^p$ of the image of $(1 + p \Z_{p^p})/(1 + p \Z_{p^p})^p$ in $M^\times/(M^\times)^p$ lies in the span of $1+\pi^2 \cO_F$. This is easily verified by viewing $1 + p\Z_{p^p}$ and $1 + \pi^2 \cO_F$ within $1 + \pi^2 \cO_M$ and using the fact that we have a logarithm isomorphism $(1 + \pi^2\cO_M,\cdot) \isoto (\pi^2\cO_M, +)$.  
\end{proof}

\section{Sharifi's explicit solutions to Massey product differential equations}\label{sec: sharifi}
We now turn our focus to developing a method for computing whether the cohomological conditions $(i)$ and $(ii)$ in Theorem \ref{thm: main of part 1} hold. Central to our approach is work of Sharifi giving explicit solutions to cup product and certain Massey product equations \cite{sharifithesis,sharifi2007}. More specifically, from Propositions \ref{prop: produce a1} and \ref{prop: produce b2}, we know that there exist solutions of the differential equations given by
\begin{equation}
\label{eq:da1 and db2}
    \begin{split}
-da\up1\cand &= b\up 1\smile c\up 1,\\
-db\up2\cand &= a\up 1\smile b\up1 + b\up1\smile d\up 1,
\end{split}
\end{equation}
but it is unclear how to write down even a single such solution.

The key observations required to produce explict soltuions to these equations are, first, that we can derive our desired solutions from related differential equations involving a special type of Massey product, \textit{cyclic Massey products}, and, second, that cyclic Massey products for absolute Galois groups can be solved via Kolyvagin derivative operators using Sharifi's theory \cite{sharifi2007}.

\subsection{Cup products and triple Massey products} Let $G$ be a profinite group and let $\chi_1, \chi_2, \chi_3 \in H^1(G,\F_p)$. 
For $n>1$, let $U_n(\F_p) \subset \GL_n(\F_p)$ denote the subgroup of upper-triangular unipotent matrices, and let $Z_n(\F_p) \subset U_n(\F_p)$ denote the center. Since $U_2(\F_p) \cong \F_p$ we can think of $\chi_i$ as a representation $\sm{1}{\chi_i}{0}{1}$ of $G$ with values in $U_2(\F_p)$

The cup product $\chi_i \cup \chi_j$ obstructs the extension of the representations $(\chi_i,\chi_j)$, with values in $U_3(\F_p)/Z_3(\F_p) \cong U_2(\F_p) \times U_2(\F_p)$, to a representation with values in $U_3(\F_p)$. We can represent this in matrix form as 
\[
\begin{pmatrix}
1 & \chi_i & * \\
0 & 1 & \chi_j \\
0 & 0 & 1
\end{pmatrix} : G \to U_3(\F_p)/Z_3(\F_p).
\]
Lifting this to $U_3(\F_p)$ is equivalent to finding a cochain $\nu:G\to \F_p$ (to go in the ``$*$"-entry) such that
\[
-d\nu = \chi_i \smile \chi_j.
\]

\begin{eg}
\label{eg:cup square}
It is well known that the ring $H^*(G,\F_p)$ is graded-commutative. In particular, for $\chi \in H^1(G,\F_p)$, the cup-square of $\chi$ is zero: $\chi \cup \chi =0$. We can realize the vanishing explicitly, defining a 1-cochain ${\chi \choose 2} = \frac{1}{2}( \chi^2 -\chi)$ and calculating  
\[
-d {\chi \choose 2} = \chi \smile \chi.
\]
The corresponding representation $G \to U_3(\F_p)$ is given by the composition of $\chi: G \to \F_p$ with the map $\F_p \to U_3(\F_p)$ sending $1$ to  the standard Jordan block matrix 
\[
\begin{pmatrix}
1 & 1 & 0 \\
0 & 1 & 1 \\
0 & 0 & 1
\end{pmatrix}.
\]
\end{eg}

Triple Massey products can be interpreted as an obstruction to extending two ``overlapping" representations with values in $U_3(\F_p)$ to a single representation with values in $U_4(\F_p)$. In particular, the ``overlapping'' representations
\[
\begin{pmatrix}
1 & \chi_1 & \kappa_{1,3} \\
0 & 1 & \chi_2 \\
0 & 0 & 1
\end{pmatrix},
\begin{pmatrix}
1 & \chi_2 & \kappa_{2,4} \\
0 & 1 & \chi_3 \\
0 & 0 & 1
\end{pmatrix}
: G \to U_3(\F_p).
\]
constitute an \emph{extension problem}, which we may represent as a homomorphism
\[
\begin{pmatrix}
1 & \chi_1 & \kappa_{1,3} & * \\
0 & 1 & \chi_2 & \kappa_{2,4} \\
0 & 0 & 1 & \chi_3 \\
0 & 0 & 0 & 1
\end{pmatrix}: G \to U_4(\F_p)/Z_4(\F_p).
\]
 A \emph{solution to the extension problem} is a $U_4(\F_p)$-valued representation where the missing ``$\ast$'' is filled in by a 1-cochain $\nu : G \to \F_p$. One readily calculates the differential equation that $\nu$ must solve, namely, 
\[
-d\nu = \chi_1 \smile \kappa_{2,4} + \kappa_{1,3} \smile \chi_3.
\]
The quantity on the right-hand side is \emph{a priori} a $2$-cocycle, and its cohomology class is an instance of the triple Massey product of the triple $(\chi_1,\chi_2,\chi_3)$. 

We formalize the foregoing discussion in the following definition.

\begin{defn}
\label{defn: triple Massey}
A \emph{defining system for the triple Massey product} $(\chi_1,\chi_2,\chi_3)$ is a pair of cochains $\kappa_{1,3}, \kappa_{2,4}$ such that
\begin{align*}
    -d \kappa_{1,3} & = \chi_1 \smile \chi_2 \\
    -d\kappa_{2,4} & = \chi_2 \smile \chi_3.
\end{align*}
The \emph{triple Massey product} $(\chi_1,\chi_2,\chi_3) \in H^2(G,\F_p)$ with respect to the defining system $\kappa_{1,3}, \kappa_{2,4}$ is the class of the 2-cocycle
\[
\chi_1 \smile \kappa_{2,4} + \kappa_{1,3} \smile \chi_3.
\]
\end{defn}

\begin{rem}
    \label{rem: Massey discussion}
    Note that the triple Massey product $(\chi_1,\chi_2,\chi_3)$ depends on the defining system. The usual notion of a Massey product, which we are departing from in Definition \ref{defn: triple Massey}, is a \emph{multi-valued} product where all possible defining systems are used to produce the values. We hope the reader does not mind our non-standard use of terminology, which allows us to avoid grammatical contortions. 
\end{rem}

\begin{eg}
Restricting the cocycles $b\up1$ and $c\up1$ to $G_{\Q(\zeta_p)}$, we see that the equation \eqref{eq:da1 and db2} is expressing $-db\cand\up2$ as the triple Massey product $(b\up1,c\up1,b\up1)$ with respect to the defining system $a\up1,\,d\up1$.
\end{eg}

\subsection{Commutativity relations}
Cup products in $H^\bullet(G,\F_p)$ are known to be skew-symmetric. Similarly, Massey products, when considered as multivalued functions as in Remark \ref{rem: Massey discussion}, are known to satisfy certain commutativity relations, cf., \cite[\S3]{may1969} or \cite[\S2]{kraines1966}. For triple Massey products, these relations are
\begin{align*}
    &(\chi_1,\chi_2,\chi_3)- (\chi_3,\chi_2, \chi_1) =0 \\
    &(\chi_1,\chi_2,\chi_3)+(\chi_2,\chi_3,\chi_1)+(\chi_3,\chi_1,\chi_2)=0.
\end{align*}
When $\chi_2=\chi_1$, we can combine these relations to obtain:
\[
(\chi_1,\chi_3,\chi_1)+2(\chi_1,\chi_1,\chi_3)=0.
\]
In particular, if $(\chi_1,\chi_1,\chi_3)$ vanishes, then $(\chi_1,\chi_3,\chi_1)$ also vanishes. 

For our purposes, we require an ``enhanced version'' of these relations, which keeps track of the 1-cochains realizing vanishing Massey products as their coboundaries, and also relates the various defining systems. Recall from (\ref{eq: cup products}) that $\cup$ denotes a cup product on cohomology while $\smile$ denotes a cup product on cochains. We extract the following result from the proofs of the commutativity relations. 
\begin{lem}
\label{lem:commutativity}
Let $\chi_1, \chi_2 \in H^1(G,\F_p)$ and suppose $\chi_1 \cup \chi_2=0$. Let $\kappa:G \to \F_p$ be a cochain such that $-d\kappa = \chi_1 \smile \chi_2$, and let ${\chi_1 \choose 2}$ be as in Example \ref{eg:cup square}.

Define
\[
\kappa_{1,3}={\chi_1 \choose 2}, \quad \kappa_{2,4} =\kappa
\]
and
\[
\kappa'_{1,3}=\kappa, \quad \kappa'_{2,4} =\chi_1\chi_2-\kappa.
\]
Then
\begin{itemize}
    \item $(\kappa_{1,3}, \kappa_{2,4})$ is a defining system for the Massey product $(\chi_1,\chi_1,\chi_2)$, 
    \item $(\kappa'_{1,3}, \kappa'_{2,4})$ is a defining system for the Massey product $(\chi_1,\chi_2,\chi_1)$.
    \item If $\nu: G \to \F_p$ satisfies 
    \[
    -d\nu = \chi_1 \smile \kappa + {\chi_1 \choose 2} \smile \chi_2
    \]
    then $\nu' = \chi_1\kappa -2 \nu -\kappa$ satisfies
    \[
    -d\nu' = \chi_1 \smile \kappa'_{2,4} + \kappa'_{1,3} \smile \chi_1.
    \]
\end{itemize}
\end{lem}
\begin{proof}
Direct computations show that
\[
-d(\chi_1\chi_2) = \chi_1 \smile \chi_2 + \chi_2 \smile \chi_1.
\]
and that
\[
-d(\chi_1\kappa) = \chi_1 \smile \kappa + \kappa \smile \chi_1 +\chi_1^2 \smile \chi_2 + \chi_1 \smile \chi_1\chi_2.
\]
The result follows by rearranging.
\end{proof}

\subsection{Cyclic triple Massey products} 
We now narrow our discussion to triple Massey products of the form $(\chi_1,\chi_1,\chi_2)$, which we will call \emph{cyclic} triple Massey products when we use only certain defining systems. In this special case, we will relate them to a connecting homomorphism in Galois cohomology. 

As motivation, first we recall the following well-known interpretation of the cup product $\chi_1 \cup \chi_2$, for fixed $\chi_1$, as a connecting homomorphism in Galois cohomology. 
\begin{eg}
Let $\chi_1\in H^1(G,\F_p)$ be non-trivial, let $H=G/\ker(\chi_1)$ be the coimage of $\chi_1$, and let $I \subset \F_p[H]$ be the augmentation ideal. Let $h \in H$ be the unique generator with $\chi_1(h)=1$ and let $X=[h]-1 \in I$. Taking $X$ as a generator of $I$ and writing $k_0+k_1X=(k_1,k_0)^T$, $G$ acts on $\F_p[H]/I^2$ via the matrix $\sm{1}{\chi_1}{0}{1}$. In particular, we can identify $I/I^2$ with $\F_p$ as $G$-modules (with trivial $G$-action) and have an exact sequence
\begin{equation}
\label{eq: mod I^2}
  0 \to \F_p \to \F_p[H]/I^2 \to \F_p \to 0
\end{equation}
of $G$-modules. The class of the extension \eqref{eq: mod I^2} in $\Ext^1_{\F_p[G]}(\F_p,\F_p) = H^1(G,\F_p)$ is $\chi_1$, so the connecting map
\[
H^1(G,\F_p) \to H^2(G,\F_p)
\]
sends $\chi_2$ to $\chi_1 \cup \chi_2$.
\end{eg}

We now consider an analogous result for cyclic triple Massey products. In this case, we have the homomorphism
\[
\begin{pmatrix}
1 & \chi_1 & {\chi_1 \choose 2} \\
0 & 1 & \chi_1 \\
0 & 0 & 1
\end{pmatrix}: G \to U_3(\F_p)
\]
given by composing $\chi_1: G \to \F_p$ with the homomorphism $\F_p \to U_3(\F_p)$ sending $1$ to the standard Jordan block matrix in $U_3(\F_p)$. This cochain ${\chi_1 \choose 2}$ gives part of the data for a defining system, and we call a defining system \emph{proper} if it includes ${\chi_1 \choose 2}$. In particular, a proper defining system is determined by a single cochain $\kappa: G \to \F_p$ such that
\begin{equation}
\label{eq:dkappa}
    -\hspace{-.075cm}d\kappa = \chi_1 \smile \chi_2.
\end{equation}

Just as for \eqref{eq: mod I^2}, the generator $X$ of $I$ identifies $I^2/I^3$ with $\F_p$ as $G$-modules, and we have an exact sequence
\begin{equation}
\label{eq: mod I^3}
  0 \to \F_p \to \F_p[H]/I^3 \to \F_p[H]/I^2 \to 0
\end{equation}
of $G$-modules.

\begin{lem}
\label{lem:proper defining}
The preimage of $\chi_2$ under the map
\[
H^1(G, \F_p[H]/I^2) \to H^1(G,\F_p)
\]
is in bijection with the set of proper defining systems of the Massey product $(\chi_1,\chi_1,\chi_2)$. Moreover, the connecting map 
\[
\partial :H^1(G, \F_p[H]/I^2) \to H^2(G,\F_p)
\]
for the sequence \eqref{eq: mod I^3} sends a proper defining system to the associated Massey product.
\end{lem}
\begin{proof}
This is a special case of \cite[Theorem 3.3.4]{LLSWW2020}. We only review the construction here. A cocycle $H^1(G, \F_p[H]/I^2)$ mapping to $\chi_2$ in $H^1(G,\F_p)$ can be written as $\chi_2 + \kappa X$ for some cochain $\kappa$. The cocycle condition exactly amounts to \eqref{eq:dkappa}, so it is equivalent to a proper defining system. The statement about the connecting map is a straightforward computation.
\end{proof}

\begin{defn}
We call Massey products $(\chi_1,\chi_1,\chi_2)$ with proper defining systems \emph{cyclic (triple) Massey products} because of the appearence of the cyclic group $H$ in the lemma. (Other groups are considered in \cite{LLSWW2020}.) 
\end{defn}
Changing the defining system for a cyclic triple Massey product changes the value in a simple way. 

\begin{lem}
\label{lem:change proper def sys}
Let $\kappa, \kappa': G \to \F_p$ be two proper defining systems for the cyclic triple Massey product $(\chi_1,\chi_1,\chi_2)$, and let $(\chi_1,\chi_1,\chi_2)_{\kappa}, (\chi_1,\chi_1,\chi_2)_{\kappa'} \in H^2(G,\F_p)$ be the corresponding values. Then
\[
(\chi_1,\chi_1,\chi_2)_{\kappa'} = (\chi_1,\chi_1,\chi_2)_{\kappa} + \chi_1 \cup (\kappa' -\kappa).
\]
\end{lem}
\begin{proof}
Immediate from the definition.
\end{proof}

\subsection{Vanishing of cyclic Massey products for absolute Galois groups} In this section, we recapitulate \cite[\S5.1]{LLSWW2020}.
From the previous section, we see the vanishing of the cyclic Massey product $(\chi_1,\chi_1,\chi_2)$ for \emph{some} proper defining system is equivalent to $\chi_2$ being in the image of the augmentation
\[
H^1(G,\F_p[H]/I^3) \xrightarrow{\iota} H^1(G,\F_p).
\]
Indeed, suppose $\iota(x)=\chi_2$ for some $x \in H^1(G,\F_p[H]/I^3)$ and let $x' \in H^1(G,\F_p[H]/I^2)$ be the image of $x$. By Lemma \ref{lem:proper defining}, $x'$ defines a proper defining system for the Massey product $(\chi_1,\chi_1,\chi_2)$ and the value is $\partial(x')$. Since $x$ is a lift of $x'$, it follows that $\partial(x')=0$ and the Massey product vanishes. Conversely, if $x' \in H^1(G,\F_p[H]/I^2)$ corresponds to a proper defining system such that the Massey product vanishes, then $\partial(x')=0$ so $x'$ lifts to $H^1(G,\F_p[H]/I^3)$.

Of course, a proper defining system for $(\chi_1,\chi_1,\chi_2)$ exists if and only if $\chi_1 \cup \chi_2 =0$. This discussion can be summarized in the following lemma (see also \cite[Proposition 5.1.3]{LLSWW2020}).

\begin{lem}
Let $\chi_1:G \to \F_p$ be a homomorphism and let $H=G/\ker(\chi_1)$ be the coimage. If the sequence
\begin{equation}
\label{eq:absolute gal type}
    H^1(G,\F_p[H]) \to H^1(G,\F_p) \xrightarrow{\cup \chi_1} H^2(G,\F_p)
\end{equation}
is exact, then for any $\chi_2 \in H^1(G,\F_p)$ such that $\chi_1 \cup \chi_2=0$, there is a proper defining system for which the cyclic Massey product $(\chi_1,\chi_1,\chi_2)$ vanishes.
\end{lem}

Now we specialize to a case where \eqref{eq:absolute gal type} is always exact: when $G$ is the absolute Galois group of a field (this is a standard property of local symbols---see \cite[\S XIV.1]{serre1979}, for example). This allows us not only to show that cyclic triple Massey product vanish, but also to compute explicit solutions, as in the following theorem, due to Sharifi (it is a special case of \cite[Theorem 4.3]{sharifi2007}). 

In order to state this theorem, we require Kolyvagin derivative operators. 
\begin{defn}[{\cite[p.~14]{sharifi2007}}]
Let $\sigma \in C_p$ denote a generator of a cyclic $p$-group $C_p$. The \emph{$i$th Kolyvagin derivative operator (with respect to $\sigma$)} is 
\[
D^i_\sigma = \sum_{j=0}^{p-1} {j \choose i} \sigma^{j} \in \F_p[C_p]. 
\]
\end{defn}

\begin{thm}[Sharifi]
\label{thm:sharifi} Let $F$ be a field containing a primitive $p$th root of unity and let $G_F$ be the absolute Galois group of $F$. Let $s,t \in F^\times$ and let $\chi_s,\chi_t: G_F \to \F_p$ be the associated Kummer characters. 
Suppose $\chi_s \cup \chi_t =0$, and let $\theta \in F(\sqrt[p]{s})^\times$ be such that $\mathrm{Nm}(\theta)=t$. 

Let $\sigma \in G_F$ satisfy $\chi_s(\sigma)=1$ so that the image of $\sigma$ in $\Gal(F(\sqrt[p]{s})/F)$ is a generator. Then there are cochains $c_1, c_2: G_F \to \F_p$ satisfying
\begin{align*}
    -dc_1 & = \chi_s \smile \chi_t \\
    -dc_2 & = \chi_s \smile c_1 + {\chi_s \choose 2} \smile \chi_t
\end{align*}
and such that, under the identification $H^1(F(\sqrt[p]{s}),\F_p) \cong F(\sqrt[p]{s})^\times \otimes \F_p$, we have
\begin{align*}
    c_1|_{G_{F(\sqrt[p]{s})}} & = D_\sigma^1(\theta) \\
    c_2|_{G_{F(\sqrt[p]{s})}} & = D_\sigma^2(\theta)
\end{align*}
where $D^i_\sigma \in \F_p[\Gal(F(\sqrt[p]{s})/F)]$ denotes the $i$th Kolyvagin derivative operator.
\end{thm}

\begin{proof}
Let $H=\Gal(F(\sqrt[p]{s})/F)$ and let $\mathrm{Sh}:H^1(F(\sqrt[p]{s}),\F_p) \isoto H^1(F,\F_p[H])$ be the isomorphism of Shapiro's lemma. The composition
\[
F(\sqrt[p]{s})^\times \otimes \F_p \cong H^1(F(\sqrt[p]{s}),\F_p) \xrightarrow{\mathrm{Sh}} H^1(F,\F_p[H]) \to H^1(F,\F_p) \cong F^\times \otimes \F_p
\]
is the norm map, so
\[
\mathrm{Sh}(\chi_\theta) = \chi_t + c_1 X + c_2 X^2 + \dots c_{p-1}X^{p-1}
\]
for some cochains $c_i$. By Lemma \ref{lem:proper defining}, the cochain $c_1$ is a proper defining system for the Massey product $(\chi_s,\chi_s,\chi_t)$, and $dc_2$ is the Massey product for that defining system. The desired formulas for $dc_1$ and $dc_2$ then follow from the definition of defining system.

To complete the proof, it suffices to explicitly write down the isomorphism $\mathrm{Sh}$. 
For this, apply Lemma \ref{lem:Shapiro} below with $G=G_F$, $X=H$, and $Y=\{1,\sigma, \sigma^2,\dots,\sigma^{p-1}\}$, so that the function $y:G_F \to Y$ is $y(g)=\sigma^{[\chi_s(g)]}$ where $[x] \in \Z$ is the smallest non-negative representative of $x \in \F_p$.
Then Lemma \ref{lem:Shapiro} states that $\mathrm{Sh}(\chi_\theta)$ is the class of the cocycle $\Phi \in Z^1(G_F,\F_p[H])$ defined by
\[
\Phi_g(\sigma^i) = \chi_\theta(\sigma^ig\sigma^{-[\chi_s(\sigma^ig)]})
\]
for $g \in G_F$ and $i \in \Z$. It remains to write $\Phi_g$ in terms of the basis $1,X,\dots,X^{p-1}$ of $\F_p[H]$ for $g \in G_{F(\sqrt[p]{s})}$.

If $g \in G_{F(\sqrt[p]{s})}$ and $i \in \{0,\dots,p-1\}$, then $[\chi_s(\sigma^ig)]=i$, so we have
\[
\Phi_g(\sigma^i) = \chi_\theta(\sigma^ig\sigma^{-i}).
\]
Now, let $\Bone_{\sigma^i} \in \F_p[H]$ denote the indicator function of $\sigma^i$ so that $\Bone_{\sigma^i} = (X+1)^i$. Then we have
\begin{align*}
\Phi_{g} &= \sum_{i=0}^{p-1}  {\chi_\theta}(\sigma^ig\sigma^{-i}) \Bone_{\sigma^i} \\
&= \sum_{i=0}^{p-1}  {\chi_\theta}(\sigma^ig\sigma^{-i}) \left( \sum_{j=0}^{p-1} {i \choose j} X^j \right) \\
& = \sum_{j=0}^{p-1} \left( \sum_{i=0}^{p-1} {i \choose j}  {\chi_\theta}(\sigma^ig\sigma^{-i}) \right) X^j \\
&=  \sum_{j=0}^{p-1} (D^j{\chi_\theta})(g) X^j. 
\end{align*}
Since the class of $g \mapsto \Phi_g$ is $\mathrm{Sh}(\chi_\theta)$, we have $c_i(g)=(D^i{\chi_\theta})(g)=\chi_{D^i(\theta)}(g)$ for all $g \in G_{F(\sqrt[p]{s})}$. This completes the proof.
\end{proof}

\begin{lem}
\label{lem:Shapiro}
Let $G$ be a group, $G' < G$ be a finite-index subgroup, and $X=G' \backslash G$. Let $\F_p[X]$ denote the $G$-module of left-$G'$-invariant functions on $G$. The isomorphism of Shapiro's Lemma is given by
\[
\mathrm{Sh}^{-1}: H^1(G,\F_p[X])  \isoto H^1(G',\F_p), \ (g \mapsto f_g) \mapsto (g' \mapsto f_{g'}(1)).
\]
The inverse is given as follows. Let $Y \subset G$ be set of coset representatives such that $1 \in Y$ and let $y: G \to Y$ be the function satisfying $G'g=G'y(g)$ for all $g \in G$, so that $y$ induces an isomorphism $y: X \isoto Y$. For $g \in G$ let $\tilde{g}=gy(g)^{-1} \in G'$. 
For a cocycle $(g' \mapsto \phi_{g'}) \in Z^1(G,\F_p)$ and $g \in G$, define $\Phi_g \in \F_p[X]$ by
\[
\Phi_g(x) = \phi_{\widetilde{y(x)g}}.
\]
Then $\mathrm{Sh}(\phi) = \Phi$.
\end{lem}
\begin{proof}
Let $\phi \in Z^1(G',\F_p)$. Note that $y(g')=1$ and $\widetilde{g'}=g'$ for all $g' \in G'$. It follows that
\[
(\mathrm{Sh}^{-1}(\mathrm{Sh}(\phi)))_{g'}=\mathrm{Sh}(\phi)_{g'}(1) = \phi_{\widetilde{y(1)g'}} = \phi_{g'}.
\]
This shows that $\mathrm{Sh}$ and $\mathrm{Sh}^{-1}$ are inverse functions on the level of cocycles. It remains only to show that $\mathrm{Sh}(\phi)$ is indeed a cocycle. It suffices to show that, for all $x \in Y$ and all $s,t \in G$,
\[
\Phi_{st}(x) = \Phi_{s}(x) + \Phi_t(xs).
\]
First note that, for all $a,b \in G$, there are equalties of cosets,
\[
G'y(ab)=G'ab =G'y(a)b =G'y(y(a)b),
\]
so $y(ab)=y(y(a)b)$. Then a simple computation shows that
\[
\widetilde{ab}= \tilde{a}\widetilde{y(a)b}.
\]
Applying this identity to $a=xs$ and $b=t$ gives 
\begin{align*}
    \Phi_{st}(x) & = \phi_{\widetilde{xst}} \\
    & = \phi_{\widetilde{xs}\widetilde{y(xs)t}} \\
    &= \phi_{\widetilde{xs}} + \phi_{\widetilde{y(xs)t}} \\
    &= \Phi_s(x) + \Phi_t(xs). \qedhere
\end{align*}
\end{proof}

\subsection{Computing $a\cand\up1$ and $b\cand\up2$}\label{subsec: computing a1 and b2}
We now apply Theorem \ref{thm:sharifi} to compute $a\cand\up1$ and $b\cand\up2$, candidate solutions to the differential equations in \eqref{eq:da1 and db2}. 

Let $K=\Q(\ell_1^{1/p},\zeta_p)$ be the splitting field of $b\up1$ and let $\sigma \in \Gal(K/\Q(\zeta_p))$ be the generator satisfying $\sigma(\ell_1^{1/p}) = \zeta_p \ell_1^{1/p}$. Let $L$ be the splitting field of $c\up1$ and let $c \in \Q(\zeta_p)$ be such that $L=\Q(\zeta_p,\sqrt[p]{c})$.
\begin{thm}
\label{thm:comp of cands}
Let $\gamma \in K^\times \otimes_\Z \F_p$ be such that $\mathrm{Nm}_{K/\Q(\zeta_p)}(\gamma)=c$. Then there are cochains $a\cand\up1, b\cand\up2: G_{\Q(\zeta_p)} \to \F_p$ satisfying
\begin{equation}
    \begin{split}
-da\up1\cand &= b\up 1\smile c\up 1\\
-db\up2\cand &= a\up 1\cand\smile b\up1 + b\up1\smile d\up 1\cand,
\end{split}
\end{equation}
where $d\up1\cand = b\up1 c\up1 -a\up1\cand$, and such that, under the identification $H^1(K,\F_p) \cong K^\times \otimes \F_p$, we have
\begin{equation}
\label{eq:a1 and b2 as Ds}
\begin{split}
    -a\up1\cand|_{G_K} &= D_\sigma^1(\gamma) \\
    -b\up2\cand|_{G_K} &= D_\sigma^2(\gamma)^{-2} D_\sigma^1(\gamma)^{-1}.
\end{split}    
\end{equation}
\end{thm}
\begin{proof}
Applying Theorem \ref{thm:sharifi}, we find that there are cochains \[a\up1\cand, Z: G_{\Q(\zeta_p)} \to \F_p\] satisfying
\[
\begin{split}
-da\up1\cand &= b\up 1\smile c\up 1\\
-dZ &= b\up1 \smile a \up1\cand + {b\up1 \choose 2} \smile c\up1
\end{split}
\]
and
\[
\begin{split}
    a\up1\cand|_{G_K} &= D_\sigma^1(\gamma) \\
    Z|_{G_K} &= D_\sigma^2(\gamma).
\end{split}
\]
By the commutativity relation (Lemma \ref{lem:commutativity}), if we define
\[
b\up2\cand = b\up1a\up1\cand -2Z - a\up1\cand
\]
then
\[
-db\up2\cand = a\up 1\cand\smile b\up1 + b\up1\smile d\up 1\cand.
\]
Since $(b\up1a\up1\cand)|_{G_K}=0$, we see that $b\up2\cand|_{G_K}= -2Z|_{G_K} - a\up1\cand|_{G_K}$, and the theorem follows.
\end{proof}

Notice that the theorem only defines cochains on $G_{\Q(\zeta_p)}$, whereas we will eventually want to work with cochains on $G_\Q$. Accounting for this will amount to keeping track of actions of $\Delta:=\Gal(\Q(\zeta_p)/\Q)$. 

We establish some notation. 
\begin{defn}
\label{defn: Delta isotypic} 
For a character $\psi: \Delta \to \F_p^\times$, let $\epsilon_\psi \in \F_p[\Delta]$ be the corresponding idempotent. For a $\F_p[\Delta]$-module $M$, let $M^{\Delta=\psi} = \epsilon_\psi M$. We call this the \emph{$\psi$-isotypic} summand of $M$. If $M$ is a Galois group $M = \Gal(F'/F)$ and $M$ is $\psi$-isotypic, then we call the extension $F'/F$ $\psi$-isotypic as well. 
\end{defn}

\begin{lem}\label{lem: Galois action a1}
Let $M$ be a $\F_p[\Gal(K/\Q)]$-module.
Suppose that $\theta \in M^{\Delta=\omega^i}$ for some $i \in \Z/(p-1)\Z$. Then $D_\sigma^1(\theta) \in M^{\Delta=\omega^{i-1}}$
\end{lem}
\begin{proof}
Note the conjugation action of $\Delta$ on $\Gal(K/\Q(\zeta_p))$ is through $\omega$, so $\delta \sigma \delta^{-1} = \sigma^{\omega(\delta)}$ for all $\delta \in \Delta$. Then a simple computation shows that $\delta D_{\sigma}^1 \delta^{-1} = \omega^{-1}(\delta) D_{\sigma}^1$. Hence we see that
\[
\delta \cdot D_\sigma^1(\theta) = \omega^{-1}(\delta) D_\sigma^1 (\delta \cdot \theta) = \omega^{i-1}(\delta) D_\sigma^1(\theta). \qedhere
\]
\end{proof}
In particular, if $\gamma$ in Theorem \ref{thm:comp of cands} is chosen to be an element in $(K^\times \otimes_\Z \F_p)^{\Delta=\omega^2}$, then $a\up1\cand|_{G_K} \in (K^\times \otimes_\Z \F_p)^{\Delta=\omega}$.

\section{Adjusting solutions to satisfy boundary conditions}
\label{sec:adjustments}
By Propositions \ref{prop: produce a1} and \ref{prop: produce b2}, we know that there exist solutions  
\[
a\up1 \in C^1(\Z[1/Np], \F_p), \quad b\up2 \in C^1(\Z[1/Np], \F_p(1))
\]
of the differential equations \eqref{eq:diffeq of a1} and \eqref{eq:diffeq for b2} that also satisfy local conditions delineated there. In this section, we explicitly construct these solutions and then prove the main result of this paper, Theorem \ref{thm: main}. We also include explicit algorithms for the main computations in the construction of $a\up1$ and $b\up2$.

\subsection{Set up for computations} By Theorem \ref{thm:comp of cands}, we already have constructions of 1-cochains 
\[
a\up1\cand : G_{\Q(\zeta_p)} \to \F_p, \qquad b\up2\cand: G_{\Q(\zeta_p)} \to \F_p(1)
\]
satisfying those same differential equations (see \eqref{eq:da1 and db2}), along with the explicit determination \eqref{eq:a1 and b2 as Ds} of the 1-cocycles $a\up1\cand|_{G_K}$ and $b\up2\cand|_{G_K}$. Moreover, Kummer theory, along with our choice of $\zeta_p$ that identifies $\F_p$ with $\F_p(1)$ and $\mu_p$, provides that the corresponding elements $a\up1|_{G_K}, b\up2|_{G_K} \in K^\times \otimes_\Z \F_p$ are uniquely determined. 

To satisfy the local conditions of Propositions \ref{prop: produce a1} and \ref{prop: produce b2}, we will add 1-cocycles to our explicitly constructed 1-cochains. Namely, we want to compute the following ``adjustment'' cocycles:
\begin{itemize}
    \item $a\up1\adj\in Z^1(\Q(\zeta_p),\F_p)$, which expresses the difference
    \[
    a\up1\adj := a\up1\vert_{G_{\Q(\zeta_p)}}-a\up1\cand,
    \] 
    \item $b\up2\adj\in Z^1(\Q(\zeta_p),\F_p(1))$, which expresses the difference 
    \[
    b\up2\adj := b\up2\vert_{G_{\Q(\zeta_p)}}-\tilde{b}\up2\cand,
    \] 
    where $\tilde{b}\up2\cand$ is constructed similarly to $b\up2\cand$ via Theorem \ref{thm:comp of cands} but also accounts for the dependence of (\ref{eq:da1 and db2}) on $a\up1$.
\end{itemize}

For each of these two adjustments, we establish a method for how to determine a corresponding element in $\Q(\zeta_p)^\times\otimes_\Z\F_p$ and then give algorithms that can be used for explicit computation. In particular, our algorithms involve extensive calculations within $S$-unit groups of Kummer extensions, so we recall two classical results from Kummer theory, without proof, that are quite useful in our setting. 

First, we implicitly use the following lemma when translating between the language of cocycles and $S$-units.
\begin{lem}[Kummer theory]
\label{lem: kummer dual}
Let $F$ be a number field that is Galois over $\Q$ and contains $\Q(\zeta_p)$, and let $F'/F$ be a $C_p$-extension, so that Kummer theory provides for the existence of some $h \in F^\times$ such that $F' = F(h^{1/p})$. Consider $\Delta = \Gal(\Q(\zeta_p)/\Q)$ to be a subset of $\Gal(F/\Q)$ under any section of the standard projection $\Gal(F/\Q) \rsurj \Delta$. Then $\Gal(F'/F)$ and $F^\times$ admit a natural $\Delta$-action that does not depend on the choice of section. The conjugation action of $\Delta$ on $\Gal(F'/F)$ is $\omega^i$-isotypic if and only if $h \otimes 1 \in F^\times \otimes_\Z \F_p$ is $\omega^{1-i}$-isotypic. 
\end{lem}

Second, the theorem below allows us to compute the splitting behavior of primes in a Kummer extension $F'/F$ in terms of the arithmetic of the base field $F$; it is one of the main advantages of computing in Kummer extensions and makes it feasible to test local conditions in number fields of degree $p^2(p-1)\geq 100$. 

\begin{thm}[{\cite[Theorem 4.12]{LemmermeyerReciprocity}}]
\label{kummertheorem}
	Let $p$ be a prime and let $F$ be a number field containing a primitive $p$th root of unity. Let $\alpha \in F^\times$ be $p$-power-free, and let $F'=F(\sqrt[p]{\alpha})$. Let $\mathfrak{p}$ be a prime of $F$. 
	\begin{enumerate}
		\item if $\mathfrak{p} \mid \alpha$, then $\mathfrak{p}$ is ramified in $F'/F$,
		\item if $\mathfrak{p} \nmid \alpha$ and $\mathfrak{p} \nmid p$, the $\mathfrak{p}$ splits if $\alpha$ is a $p$th power mod $\mathfrak{p}$ and is inert otherwise.
		\item if $\mathfrak{p} \nmid \alpha$ and $\mathfrak{p} \mid p$, let $a$ be the highest power of $\mathfrak{p}$ dividing $1-\zeta_p$ in $F$. Then
		\[
		\begin{array}{lll}
		\mbox{$\mathfrak{p}$ splits} &\mbox{if} & \mbox{$\alpha$ is a $p$th power modulo $\mathfrak{p}^{ap+1}$} \\
		\mbox{$\mathfrak{p}$ ramifies} &\mbox{if} & \mbox{$\alpha$ is not a $p$th power modulo $\mathfrak{p}^{ap}$} \\
		\mbox{$\mathfrak{p}$ is inert} & & \mbox{otherwise} \\
		
		\end{array}
		\]
	\end{enumerate}
\end{thm}

It remains to outline the data fixed in our set up for computing $a\up1\adj$ and $b\up2\adj$.  From a computational perspective, the pinning data in Definition \ref{defn: pinning} plays a much less salient role, with our algorithms depending directly only on our choices of $\zeta_p,\,\ell_p^{1/p}\in\overline{\Q}$. Our algorithms also require that we have fixed $S$-units, with $S$ taken to be the set of primes over $\ell_0$, for the following elements from Theorem \ref{thm:comp of cands}:
\begin{itemize} \item $c \in (\Q(\zeta_p)^\times \otimes_\Z \F_p)^{\Delta = \omega^2}$ such that $L=\Q(\zeta_p,\sqrt[p]{c})$,
    \item $\gamma \in (K^\times \otimes_\Z \F_p)^{\Delta=\omega^2}$ such that $\mathrm{Nm}_{K/\Q(\zeta_p)}(\gamma)=c$,
 \item $a\up1\cand|_{G_K} = D_\sigma^1(\gamma)\in (K^\times \otimes_{\Z}\F_p)^{\Delta=\omega}$,
\item $b\up2\cand|_{G_K} = D_\sigma^2(\gamma)^{-2} D_\sigma^1(\gamma)^{-1}\in (K^\times \otimes_{\Z}\F_p)$.

\end{itemize}
In computing these $S$-units, we must fix a choice of generators
\begin{itemize}
    \item $\sigma \in \Gal(K/\Q(\zeta_p))$ satisfying $\sigma(\ell^{1/p})=\zeta_p \ell^{1/p}$, 
    \item $\delta \in \Delta = \Gal(\Q(\zeta_p)/\Q)$.
\end{itemize} Note that throughout the algorithms in this section, we view $S$-units in any field $F^\times$ as elements in $F^\times\otimes_\Z\F_p$ by implicitly reducing modulo $p$-powers. 

A complete computational example is given in $\S$\ref{subsec: complete computation}; the reader might find it useful to work through this example alongside the remainder of this section.

\subsection{Adjustment for local conditions on $a\up1$}\label{subsec: a1 adjustments}
The cochains $a\up1$ and $a\up1\cand$ both satisfy the differential equation in \eqref{eq:diffeq of a1}. The difference is that $a\up1$ also satisfies the two local conditions of Proposition \ref{prop: produce a1}:
\begin{enumerate}[label=(\alph*)]
    \item $a\up1 |_{I_p} = - (b\up1 \cup x_{c})|_{I_p}$, and
    \item $a\up1|_{\ell_0}$ is on the line spanned by $\zeta_p \cup c_0|_{\ell_0}$.
\end{enumerate}

Recall the element $a_0 \in H^1(\Z[1/Np],\F_p)$, ramified only at $\ell_0$, that we selected in Definition \ref{defn: pinned cocycles}. In what follows, we will think of it as an element of $\Q(\zeta_p)^\times \otimes_\Z \F_p$ via
\[
a_0|_{G_{\Q(\zeta_p)}} \in H^1(\Q(\zeta_p),\F_p) \risom H^1(\Q(\zeta_p), \F_p(1)) \cong \Q(\zeta_p)^\times \otimes_\Z \F_p,
\]
where the isomorphism $H^1(\Q(\zeta_p),\F_p) \risom H^1(\Q(\zeta_p), \F_p(1))$ is drawn using our chosen primitive $p$th root of unity $\zeta_p$. 

We also want to point out to the reader that we evaluate the special element from Theorem \ref{thm:slope},
\[
\zeta_\mathrm{MT}' \lambda + \frac{1}{6} a_0|_{\ell_0} \in H^1(\Q_{\ell_0},\F_p) = \Hom(G_{\ell_0}, \F_p) = \Hom(G_{\ell_0}^\ab, \F_p),
\]
at elements of $K^\times \otimes_\Z \F_p$ by using the local Artin map
\begin{equation}
    \label{eq: localized Artin map at ell0}
    \mathrm{Art}_{\mathfrak{L}_0} : K^\times \to \Q_{\ell_0}^\times \buildrel{\mathrm{Art}}\over\to G_{\ell_0}^\ab 
\end{equation}
arising from the distinguished place $\mathfrak{L}_0$ of $K$ over the prime $\ell_0$. Specifically, under the local Artin map, we identify the basis $\{\lambda,a_0|_{\ell_0}\}$ of $H^1(\Q_{\ell_0},\F_p)$ with the basis $\{\ord_{\ell_0},\log_{\ell_0}\}$ of $\Hom(\Q_{\ell_0}^\times,\F_p)$, where $\ord_{\ell_0}$ is the usual $\ell_0$-valuation on $\Q_{\ell_0}^\times$, and $\log_{\ell_0}$ is the projection $\Q_{\ell_0}^\times\to\,\Z_{\ell_0}^\times\to\,\F_{\ell_0}^\times$ composed with the discrete logarithm $\F_{\ell_0}^\times \onto \F_p$ determined by $a_0|_{\ell_0}$. Then, for $y\in K^\times\otimes_\Z\F_p\hookrightarrow\Q_{\ell_0}^\times\otimes_\Z\F_p$, we obtain 
\begin{equation}\label{eq: mazur-tate line eq}
(\zeta_\mathrm{MT}' \lambda + \frac{1}{6} a_0|_{\ell_0})(\mathrm{Art}_{\mathfrak{L}_0}(y)) = \zeta_\mathrm{MT}'  \ord_{\ell_0}(y) + \log_{\ell_0} \left( \frac{y}{\ell_0^{\ord_{\ell_0}(y)}} \right)\in\F_p.
\end{equation}

We now determine $a\up1\adj = a\up1-a\up1\cand\in Z^1(\Q(\zeta_p),\F_p)$.

\begin{thm}\label{thm: a1 adjustments}
The element $a\up1\adj\in \Q(\zeta_p)^\times \otimes_\Z \F_p$ is given by $\zeta_p^i a_0|_{G_{\Q(\zeta_p)}}^j$, where 
\begin{itemize}
    \item $i \in \F_p$ is the unique element such that $K\left(\sqrt[p]{D_\sigma^1(\gamma)\zeta_p^i}\right)/K$ is unramified at $p$, and
    \item $j \in \F_p$ is the unique element such that the evaluation
    \[
    \left(\zeta_\mathrm{MT}' \lambda + \frac{1}{6} a_0|_{\ell_0} \right) \left(\mathrm{Art}_{\mathfrak{L}_0} \big(D_\sigma^1(\gamma)\zeta_p^i  a_0|_{G_{\Q(\zeta_p)}}^j\big)\right) =0 \in \F_p,
    \]
    where $\mathrm{Art}_{\mathfrak{L}_0}$ is as in \eqref{eq: localized Artin map at ell0}. 
\end{itemize}
\end{thm}

\begin{proof}
First note that, since $a\up1\vert_{G_K}, a\up1\cand|_{G_K} \in (K^\times \otimes_{\Z}\F_p)^{\Delta=\omega}$ by Lemma \ref{lem: kummer dual}, and since both $a\up1$ and $a\up1\cand$ are unramified outside $Np$,  $a\up1\adj$ enjoys the same properties. In particular, $a\up1\adj \in (\Z[\zeta_p,1/Np]^\times \otimes_\Z \F_p)^{\Delta=\omega}$, which is spanned by $\zeta_p$ and $a_0|_{G_{\Q(\zeta_p)}}$. Hence we have $a\up1\adj = \zeta_p^i a_0|_{G_{\Q(\zeta_p)}}^j$ for some $i$ and $j$.

This implies that $a\up1|_{G_K} = D^1_\sigma(\gamma)\zeta_p^i a_0|_{G_{\Q(\zeta_p)}}^j$. The first condition on $a\up1$ implies that $K\left(\sqrt[p]{a\up1|_{G_K}}\right)/K$ is unramified at $p$. Since $a_0$ is unramified at $p$, we see that this is equivalent to $K\left(\sqrt[p]{D_\sigma^1(\gamma)\zeta_p^i}\right)/K$ being unramified at $p$. This determines $i$.

By Theorem \ref{thm:slope}, the second condition on $a\up1$ is equivalent to
\[
\left(\zeta_\mathrm{MT}' \lambda + \frac{1}{6} a_0|_{\ell_0} \right) \left(\mathrm{Art}_{\mathfrak{L}_0} (a\up1|_{G_K})\right) =0.
\]
This determines $j$.
\end{proof}

\subsubsection*{Algorithms for computing $a\up1\adj$} 
To explicitly compute the first adjustment of $a\up1\cand$, we apply Theorem \ref{kummertheorem}(3), which states that if $a$ is the highest power of $\mathfrak{p}$ dividing $(1-\zeta_p)$ in $K$, then $\mathfrak{p}$ ramifies if the Kummer generator $\alpha$ is not a $p$th power modulo $\mathfrak{p}^{ap}$. The exponent $a$ depends on whether $p$ tamely or wildly ramifies in $K/\Q$, so we handle these cases separately in Algorithm \ref{alg:algorithm_a1_p}. 

%\Hmar{If Algorithm 1 appears on the next page in the final version of this draft, we should note that here.}

\begin{algorithm}\label{alg:algorithm_a1_p}
\KwIn{$a\up1\cand|_{G_K}\in (K^\times \otimes_{\Z}\F_p)^{\Delta=\omega}$}
\KwOut{$i\in\F_p$ such that $a\up1\cand|_{G_K}\zeta_p^i$ generates a Kummer extension that is unramified at $p$}
\caption{Adjustment of $a^{(1)}\cand$ for condition at $p$ in Theorem \ref{thm: a1 adjustments}}
\begin{enumerate}[leftmargin=1.75em]
     \item If $p$ \textit{tamely} ramifies in $K$, let $i\in\F_p$ be such that $a\up1\cand|_{G_K}\zeta_p^i$ is congruent to a $(p-1)$st root of unity in $\mathcal{O}_K/\mathfrak{p}^p$ for each $\mathfrak{p}$ above $p$ in $K$.
     \item If $p$ \textit{wildly} ramifies in $K$:
    \begin{enumerate}
        \item Let $\mathfrak{p}$ denote the prime above $p$ in $K$. 
        \item Compute a set of representatives for the $p$th-powers in $\mathcal{O}_K/\mathfrak{p}^{p^2}$.
        \item Let $i\in\F_p$ be such that $a\up1\cand|_{G_K}\zeta_p^i$ is a $p$th power $\mathcal{O}_K/\mathfrak{p}^{p^2}$.
    \end{enumerate}
   
    \item Return $i$.
\end{enumerate}
\end{algorithm}

For the second adjustment of $a\up1\cand$, it would be straightforward to implement the method outlined in Theorem \ref{thm: a1 adjustments} provided that we have computationally identified the distinguished prime of $K$ over $\ell_0$. However, even after computing $a\up1$, it would be unclear how to compute the value of $\alpha$ in Definition \ref{defn: alpha}. So, for computational purposes, we take a different approach that involves changing the pinning data to arrange for $\alpha=0$; this is permitted because $\alpha^2 + \beta$ is independent of the pinning data (Part I, Theorem \ref{thm: main invariant}) and guaranteed to be possible by the following lemma.

\begin{lem}
\label{lem: choice pinning data}
There is a choice of pinning data such that $\alpha = 0$. Moreover, when $\alpha = 0$, the primes over $\ell_0$ that are split in the Kummer extension generated by $a\up1|_{G_K}$ are exactly the $\Delta$-orbit of the distinguished place of $K$ over $\ell_0$.
\end{lem}

\begin{proof} 
Given any choice of pinning data, define $\alpha\in\F_p(1)$ as in Definition \ref{defn: alpha}.  Then, change our choice of decomposition group at $p$ and $p$th root of $\ell_1$ so that the new choice corresponds to the Kummer cocycle in $Z^1(\Z[1/Np],\F_p(1))$ given by 
\[
b\up1+\alpha(\omega-1).
\] 
By Part I, Lemma \ref{lem: change pth root of ell1}, this gives $\alpha =0$ for the new choice of cocycle $a\up1$. 

Now, let $K'/K$ be the Kummer extension generated by $a\up1|_{G_K}$. By the definition of $\alpha$, if $\alpha =0$, then $a\up1\vert_{\ell_0} = 0$. So the distinguished prime of $K$ over $\ell_0$ splits in $K'/K$. Because $K'/K$ is $\omega^0$-isotypic (in the sense of Definition \ref{defn: Delta isotypic}), all places in the $\Delta$-orbit of the distinguished prime also split in $K'/K$. Because $\ell_0$ ramifies in $L/\Q(\zeta_p)$ and the Galois closure of $K'/\Q$ is $M'=K'L$, this $\Delta$-orbit consists of exactly those places of $K$ over $\ell_0$ that split in $K'/K$. Note that $M'/\Q$ is Galois since it is cut out by the homomorphism \eqref{eq: 3d abc}.
\end{proof}

In light of this observation, we assume that we have arranged for $\alpha=0$ for the remainder of the paper. Moreover, when $\alpha =0$, Lemma \ref{lem: choice pinning data} suggests an alternative method for computing the value of the adjustment $j\in\F_p$ in Theorem \ref{thm: a1 adjustments} that does not involve the distinguished prime of $K$ over $\ell_0$:

\begin{lem}
\label{lem: splitting condition for a1} 
When $\alpha=0$, the element $j\in\F_p$ determined in Theorem \ref{thm: a1 adjustments} is also the unique element such that $a\up1\cand|_{G_K}\zeta_p^ia_0|_{G_{\Q(\zeta_p)}}^j$ generates a Kummer extension that is split at exactly one $\Delta$-orbit of primes above $\ell_0$ in $K$.
\end{lem}

\begin{proof} 
Let $j\in\F_p$ be determined as in Theorem \ref{thm: a1 adjustments}, and suppose that $j'\in\F_p$ is such that  
$a\up1\cand|_{G_K}\zeta_p^ia_0|_{G_{\Q(\zeta_p)}}^{j'}$ generates a Kummer extension that is split at exactly one $\Delta$-orbit of primes above $\ell_0$ in $K$. Denote this $\Delta$-orbit of primes by $L_{\mathrm{split}}$. 

By Part I, Lemma \ref{lem: change ell0}, the construction of $a\up1|_{G_K}$, and hence the value of $j\in\F_p$, is invariant if we change the distinguished prime of $K$ over $\ell_0$ within its $\Gal(K/\Q(\zeta_p))$-orbit. So, we can change the distinguished prime, if necessary, to be a prime of $L_{\mathrm{split}}$ without changing the value of $j\in\F_p$. In particular, after this change, we conclude that $j=j'$ by the uniqueness of $a\up1$ since both choices of $j$ and $j'$ satsify the conditions in Theorem \ref{prop: produce a1}.
\end{proof}

We implement the method outlined in Lemma \ref{lem: splitting condition for a1} in Algorithm \ref{alogorithm_a1_ell0}, which appears on the next page, 
%\Hmar{If Algorithm 2 appears on the next page in subsequent versions, we should note that here.} 
using Theorem \ref{kummertheorem}(2) to check the splitting behavior of primes above $\ell_0$ in the various Kummer extensions of $K$. Note that once we compute $j\in\F_p$ via Lemma \ref{lem: splitting condition for a1}, we have computationally identified the $\Delta$-orbit of the distinguished prime of $K$ over $\ell_0$. Without loss of generality, we can fix one of the split primes in the Kummer extension generated by $a\up1|_{G_K}$ to be the distinguished prime of $K$ over $\ell_0$ for future computations. 

\begin{algorithm}\label{alogorithm_a1_ell0}
\KwIn{$a\up1\cand|_{G_K}\zeta_p^i\in (K^\times \otimes_{\Z}\F_p)^{\Delta=\omega}$}

\KwOut{$j\in\F_p$ such that $a\up1|_{G_K}=a\up1\cand|_{G_K}\zeta_p^ia_0|_{G_{\Q(\zeta_p)}}^j$}
\KwOut{$\mathfrak{L}_0$, the distinguished prime of $K$ over $\ell_0$}

\caption{Adjustment of $a^{(1)}\cand$ for condition at $\ell_0$ in Theorem \ref{thm: a1 adjustments}}
\begin{enumerate}[leftmargin=1.75em]
    \item Compute an $S$-unit for $a_0|_{G_{\Q(\zeta_p)}} \in \Q(\zeta_p)^\times \otimes_\Z \F_p$.
    \item For $t=0,1,2,\dots,p-1$:
    \begin{enumerate}
    \item Initialize $L_{\mathrm{split}}$ to the empty set $\{\}$.
    \item For each prime $\mathfrak{L}$ in $K$ above $\ell_0$:
    \begin{enumerate}[leftmargin=2em]
        \item If $a\up1\cand|_{G_K}\zeta_p^ia_0|_{G_{\Q(\zeta_p)}}^t$ is a $p$th power mod $\mathfrak{L}$, append $\mathfrak{L}$ to $L_{\mathrm{split}}$.
    \end{enumerate}
    \item If $L_{\mathrm{split}}$ contains exactly one $\Delta$-orbit of primes above $\ell_0$:
    \begin{enumerate}[leftmargin=2em]
        \item Let $j=t$ and $\mathfrak{L}_0$ be any prime in $L_{\mathrm{split}}$.
        \item Break for loop.
    \end{enumerate}
    \end{enumerate}

    \item Return $j,\,\mathfrak{L}_0$.
\end{enumerate}
\end{algorithm}

\subsection{Adjusting $b\up2\cand$ along with $a\up1\cand$}
\label{subsec: adjusting a1 and b2 simultaenously}

Having computed $a\up1$, we turn to $b\up2$. Recall that in Theorem \ref{thm:comp of cands}, the differential equation that $b\up2\cand$ satisfies depends on the choice of $a\up1\cand$. So, before we can make adjustments for the local conditions on $b\up2$, we must do a preliminary adjustment to $b\up2\cand$ to ensure it satisfies the correct differential equation, now depending on $a\up1$. In particular, by Proposition \ref{prop: produce b2}, this preliminary adjustment is possible only when $a\up1|_{\ell_1}=0$.
\begin{lem}\label{lem: b2 adjustment xi}
When $a\up1|_{\ell_1}=0$, there is an element $\xi \in K^\times \otimes_\Z \F_p$ such that $\mathrm{Nm}_{K/\Q(\zeta_p)}(\xi) =a\up1\adj$ in $\Q(\zeta_p)^\times \otimes_\Z \F_p$.

Moreover, there is a cocycle $\tilde{b}\up2\cand: G_{\Q(\zeta_p)} \to \F_p(1)$ satisfing
\[
d\tilde{b}\up2\cand = a\up1 \smile b\up1 +b\up1 \smile d\up1,
\]
where $d\up1=b\up1c\up1 -a\up1$ and such that
\begin{equation}\label{eq: b_tilde_cand}
\tilde{b}\up2\cand|_{G_K} = (D_\sigma^2(\gamma)D_\sigma^1(\xi))^{-2} a\up1|_{G_K}^{-1}.
\end{equation}
\end{lem}
\begin{proof}
When $a\up1|_{\ell_1}=0$, we know $a\up1\adj\cup b\up1=0$ in $H^2(\Z[1/Np],\F_p(1)$. The result then follows immediately from Lemma \ref{lem:change proper def sys} and Theorem \ref{thm:sharifi}. 
\end{proof}

Since it is straightfoward to compute an $S$-unit for $\tilde{b}\up2\cand|_{G_K}\in (K^\times \otimes_{\Z}\F_p)^{\Delta = \omega^0}$, we assume that we have done so. While there is no result analogous to Lemma \ref{lem: Galois action a1} for ensuring $\tilde{b}\up2\cand$ is in the proper $\omega$-isotypic class, we can to obtain the desired $\Delta$-action on $\tilde{b}\up2\cand$ by multiplying it by an appropriate $p$-power $S$-unit in $K$.

\subsection{Adjustment for local conditions on $b\up2$}
\label{subsec: adjusting b2_tilde} 
Now, assuming that we have constructed $\tilde{b}\up2\cand$, it satisfies the same differential equation as $b\up2$, but $b\up2$ also satisfies the two local conditions stipulated in Proposition \ref{prop: produce b2}:
\begin{enumerate}[label=(\alph*)]
    \item $b\up2|_{\ell_0}$ is on the line spanned by $\zeta' \cup c_0|_{\ell_0}$ for some basis $\zeta' \in H^0(\Q_{\ell_0}, \F_p(2))$, and 
    \item $b\up2|_p$ is finite-flat in the sense of Proposition \ref{prop: produce b2}(b); we may and do use the much simpler finite-flatness criterion of Proposition \ref{prop: b2 f-f unramified test}, thanks to the equivalence drawn in Proposition \ref{P2: prop: ref flat rho2 step 2}.
\end{enumerate}
Let $b\up2\adj = b\up2 - \tilde{b}\up2\cand \in Z^1(\Q(\zeta_p),\F_p(1))$.

\begin{thm}
\label{thm: b2 adjustments}
The element $b\up2\adj \in  \Q(\zeta_p)^\times \otimes_\Z \F_p$ is given by $p^k \ell_0^m$ where
\begin{itemize}
    \item $k \in \F_p$ is the unique element such that $\tilde{b}\up2\cand|_{G_K}p^k \in K^\times \otimes_\Z \F_p$ is $1$ modulo $\p^{2(p-1)}$, where $\p \subset K$ is a prime above $p$, 
    \item $m \in \F_p$ is the unique element such that  
    \[
    \left(\zeta_\mathrm{MT}' \lambda + \frac{1}{6} a_0|_{\ell_0} \right) \left(\mathrm{Art}_{\mathfrak{L}_0} (\tilde{b}\up2\cand|_{G_K}p^k\ell_0^m)\right) =0,
    \]
    where $\mathrm{Art}_{\mathfrak{L}_0}$ is as in \eqref{eq: localized Artin map at ell0}. 
\end{itemize}
\end{thm}

\begin{proof}
We follow the proof of Theorem \ref{thm: a1 adjustments}, with appropriate modifications. Indeed, since $b\up2\vert_{G_K}, \tilde{b}\up2\cand|_{G_K} \in (K^\times \otimes_{\Z}\F_p)^{\Delta=\omega^0}$, we have $b\up2\adj \in \Z[1/Np]^\times \otimes_\Z \F_p$. We can drop the $\ell_1$-part of $b\up2\adj$ and aim to compute the $k$ and $m$ in $b\up2\adj = p^k\ell_0^m$, because, by the last claim in Proposition \ref{prop: produce b2}, $b\up2$ is only well defined up to coboundaries and multiples of $\ell_1$. Indeed, only the $\ell_0$-local behavior of $b\up2$ matters, and, as the existence of $\beta$ proved in Proposition \ref{prop: produce b2} corroborates, adjustments by powers of $\ell_1$ are $\ell_0$-locally trivial. 

The stated formula for the $k \in \F_p$, which comprises the first adjustment for $b\up2$, follows directly from the finite-flatness criterion of Proposition \ref{prop: b2 f-f unramified test} as made explicit in Lemma \ref{lem: b1 ramified f-f test} and the discussion before it. Note these results are stated for $\Q(\ell_1^{1/p})$ rather than $K$, so the condition $1$ modulo $\mathfrak{p}^{2(p-1)}$ in Theorem \ref{thm: b2 adjustments} accounts for multiplication by the ramification degree $(p-1)$ of the prime $p$ in $\Q(\zeta_p)/\Q$. 

The method for computing the second adjustment for $b\up2$, which determines $m$, can be done by relying on the expression for the slope of $c\up1\vert_{\ell_0}$ in Theorem \ref{thm:slope}, exactly as in Theorem \ref{thm: a1 adjustments}. 
\end{proof}

\subsubsection*{Algorithms for computing $b\up2\adj$} 
To explicitly compute the adjustments for the local conditions on $b\up2$, we use straightforward implementations of the methods outlined in Theorem \ref{thm: b2 adjustments}, starting with the adjustment for the finite-flat condition.

As discussed in \S \ref{subsec: finite-flat}, the level of difficulty in computing the adjustment for the finite-flat condition on $b\up2$ is controlled by whether $b\up1|_{I_p}=0$ and $a\up1|_p=0$. Ideally, the algorithm for this adjustment would isolate the most difficult case, $b\up1|_{I_p}\neq0$ and $a\up1|_p\neq0$, but in practice, checking whether $a\up1|_p=0$ can be an infeasible computation. Therefore, Algorithm \ref{algorithm_b2_ff}, which appears below, considers two cases based on whether $b\up1|_{I_p}=0$:
%\Hmar{If Algorithm 3 appears on the next page in the final version of this draft, we should note that here.}
\begin{itemize}
    \item When $b\up1|_{I_p}=0$, i.e., $p$ tamely ramifies in $K/\Q$, we can apply Proposition \ref{prop: b2 f-f unramified test} to see that $\tilde{b}\up2\cand|_{G_K}$ is finite-flat at $p$ as long as it is prime-to-$p$ as an $S$-unit, which is guaranteed by its construction.
    \item When $b\up1|_{I_p}\neq0$, i.e., $p$ wildly ramifies in $K/\Q$, we can apply Lemma \ref{lem: b1 ramified f-f test} to see that $\tilde{b}\up2\cand|_{G_K}$ is finite-flat at $p$ if it is congruent to 1 modulo $\mathfrak{p}^{2(p-1)},$ where $\mathfrak{p}$ is the prime above $p$ in $K$. 
\end{itemize}

\begin{algorithm}\label{algorithm_b2_ff}
\KwIn{$\tilde{b}\up2\cand|_{G_K}\in(K^\times \otimes_{\Z}\F_p)^{\Delta=\omega^0}$}

\KwOut{$k\in\F_p$ such that $\tilde{b}\up2\cand|_{G_K} p^k$ satisfies the finite-flat condition at $p$}
\caption{Adjustment of $\tilde{b}^{(2)}\cand$ for finite-flat condition in Theorem \ref{thm: b2 adjustments}}
\begin{enumerate}[leftmargin=1.75em]
      \item If $p$ \textit{tamely} ramifies in $K$, let $k=0$.
 \item If $p$ \textit{wildly} ramifies in $K$:
    \begin{enumerate}
        \item Let $\mathfrak{p}$ be the unique prime in $K$ above $p$.
        \item Let $k\in\F_p$ be such that $\tilde{b}\up2\cand|_{G_K}p^k\equiv 1\pmod{\mathfrak{p}^{2(p-1)}}$.
    \end{enumerate}
    \item Return $k$.
\end{enumerate}
\end{algorithm}

\begin{rem}
Ultimately, we want to determine $\beta$, which depends on the restriction of $b\up2$ to $\ell_0$. When $p$ is a $p$th power modulo $\ell_0$, the $p^k$ part of $b\up2\adj$ vanishes at $\ell_0$. That is, when $\log_{\ell_0}(p) = 0$, then $b_p\vert_{\ell_0} = 0$. Hence, when $p$ is a $p$th power modulo $\ell_0$, we do not need to run Algorithm \ref{algorithm_b2_ff} since the adjustment by $p^k$ is trivial at $\ell_0$. 
\end{rem}

To compute the second adjustment of $b\up2$, we follow the method outlined in Theorem \ref{thm: b2 adjustments}, which requires the distinguished prime $\mathfrak{L}_0$ identified in the output of Algorithm \ref{alogorithm_a1_ell0}. The computation of this last adjustment is given in Algorithm \ref{algorithm_b2_ell0}. 
%\Hmar{If Algorithm 4 appears on the next page in a subsequent versions, we should note that here.}

\begin{algorithm}\label{algorithm_b2_ell0}
\KwIn{$\tilde{b}\up2\cand|_{G_K} p^k\in(K^\times \otimes_{\Z}\F_p)^{\Delta=\omega^0}$}
\KwIn{$\mathfrak{L}_0$, the distinguished prime of $K$ over $\ell_0$}
\KwOut{$m\in\F_p$ such that $b\up2|_{G_K}=\tilde{b}\up2\cand|_{G_K}p^k\ell_0^m$}
\caption{Adjustment of $\tilde{b}^{(2)}\cand$ for condition at $\ell_0$ in Theorem \ref{thm: b2 adjustments}}
\begin{enumerate}[leftmargin=1.75em]
    \item Compute $\zeta_\mathrm{MT}', \,\lambda$, and $a_0|_{\ell_0}$.
    \item Let $a= \left(\zeta_\mathrm{MT}' \lambda + \frac{1}{6} a_0|_{\ell_0} \right) \big(\mathrm{Art}_{\mathfrak{L}_0}\big(\tilde{b}\up2\cand|_{G_K}p^k\big)\big)$. 
    \item Let $b= \left(\zeta_\mathrm{MT}' \lambda + \frac{1}{6} a_0|_{\ell_0} \right) (\mathrm{Art}_{\mathfrak{L}_0}(\ell_0))$.
    \item Return $m=-ab^{-1}$.
\end{enumerate}
\end{algorithm}

\subsection{Deducing the main theorem}
\label{subsec: P2 deducing main}

Now that we have established a method to construct the elements $a\up1\vert_K,b\up2\vert_K \in K^\times \otimes_\Z \F_p$ when $a\up1$ and $b\up2$ exist, we can prove Theorem \ref{thm: main}. More precisely, we deduce Theorem \ref{thm: main} from the main theorem of Part I (restated in this paper as Theorem \ref{thm: main of part 1}). 

Both of these theorems express a criterion for $\dim_{\F_p} R/pR > 3$, and our work is to draw an equivalence between these criteria. The criterion of Theorem \ref{thm: main} is expressed in terms of certain twisted-Heisenberg extension of $\Q$, which we now set up. Assuming $a\up1$ and $b\up2$ both exist, we construct a lattice of $C_p$-extensions 
\[
\begin{tikzcd}[every arrow/.append style={-}]
K'\arrow[swap,"a\up1"]{dr}&K''\arrow["b\up2"]{d}\\
&K=\Q(\zeta_p, \ell_1^{1/p})\arrow["b\up1"]{dl}\\
\Q(\zeta_p)&
\end{tikzcd}
\]
in which each $C_p$-extension is generated by the $p$-th root of the $S$-unit labeling it. Note that while this constructive definition of $K'/K$ and $K''/K$ suffices for the purposes of this section, we give a number-theoretic characterization of these extensions in Propositions \ref{prop: characterize M'} and \ref{prop: characterize M''}, respectively. 

Having defined these fields, we state this paper's main theorem. 
\begin{thm}
\label{thm: main}
We have $\dim_{\F_p}R/pR>3$ if and only if (i) and (ii) hold:
    \begin{enumerate}[label=(\roman*)]
        \item all primes of $K$ over $\ell_1$ split in $K'/K$;
        \item there exists some prime of $K$ over $\ell_0$ that splits in both $K'/K$ and $K''/K$.
    \end{enumerate}
In particular, when $\dim_{\F_p}R/pR=3$, we have $R=\mathbb{T}$.
\end{thm}

Theorem \ref{thm: main} follows directly from the following key lemma relating the criterion ``$a\up1\vert_{\ell_1} = 0$ and $\alpha^2 + \beta = 0$'' of Theorem \ref{thm: main of part 1} to the criterion of Theorem \ref{thm: main}. 

\begin{lem}
\label{lem: main conditions equivalent}
We have the following equivalences. 
\begin{enumerate}
    \item $a\up1\vert_{\ell_1} = 0$ if and only if all primes of $K$ over $\ell_1$ split in $K'/K$
    \item $\alpha^2 + \beta = 0$ if and only if there exists some prime of $K$ over $\ell_0$ that splits in both $K'/K$ and $K''/K$.
\end{enumerate}
\end{lem}

\begin{proof}[Proof of Lemma \ref{lem: main conditions equivalent}]
The first equivalence is standard. For the second equivalence, recall that we have arranged for $\alpha =0$ via our choice in pinning data. So, $\alpha^2 + \beta =0$ if and only if $\beta =0$. Thus, if $\alpha^2+\beta =0$, the distinguished prime of $K$ over $\ell_0$ satisfies the desired property: it splits in both $K'/K$ and $K''/K$. 

On the other hand, if we assume that there exists some place $\cL'_0$ of $K$ at $\ell_0$ that splits in both $K'/K$ and $K''/K$, our goal is to prove that $\beta=0$. By Lemma \ref{lem: choice pinning data}, we know that this place lies in the $\Delta$-orbit of the distinguised prime $\cL_0$ of $K$ over $\ell_0$, so let $\sigma \in \Delta$ such that $\sigma(\cL_0) = \cL'_0$ and consider the homomorphism 
\[
\nu = \begin{pmatrix}
\omega & b\up1 & \omega a\up1 & b\up2\\
0 & 1 & \omega c\up1 &  d\up1\\
0 & 0 & \omega & b\up1 \\
0 & 0 & 0 & 1
\end{pmatrix}.
\]
We see that the conjugate of $\nu$ by $\sigma$---that is, $\tau \mapsto \nu(\sigma^{-1}\tau\sigma)$---has its $b\up2$-coordinate vanishing at $\ell_0$. Its $b\up2$-coordinate would be given by $\omega(\sigma^{-1})\cdot b\up2$. Therefore, $b\up2$ also vanishes at $\ell_0$, which means $\beta = 0$. 
\end{proof}

\vspace{-0.5em}

\subsubsection*{Algorithms to verify conditions in Theorem \ref{thm: main}} 
To check condition $(i)$ in Theorem \ref{thm: main}, i.e, whether $a\up1|_{\ell_1}=0$, we apply Theorem \ref{kummertheorem}(2) in the Kummer extension $K'/K$ generated by $a\up1|_{G_K}\in K^\times \otimes_{\Z}\F_p$.  
%\Hmar{If either Algorithm 5 or 6 appears on the next page in the final version of this draft, we should note that here.}

\begin{algorithm}\label{alg:finalcheck1}
\KwIn{$a\up 1|_{G_K} \in(K^\times \otimes_{\Z}\F_p)^{\Delta=\omega}$}
\KwOut{True if $a\up1|_{\ell_1}=0$; False otherwise}
\caption{Verify condition $(i)$ in Theorem \ref{thm: main}}
\begin{enumerate}[leftmargin=1.75em]
    \item If $a\up1|_{G_K}$ is a $p$th power mod $\mathfrak{L}$ for each prime $\mathfrak{L}$ in $K$ above $\ell_1$, return True.
    \item Else, return False.
\end{enumerate}
\end{algorithm}

Assuming Algorithm \ref{alg:finalcheck1} returns True, we proceed to check condition $(ii)$ in Theorem \ref{thm: main}, i.e., whether $\beta =0$, by applying Theorem \ref{kummertheorem}(2) in the Kummer extension $K''/K$ generated by $b\up2|_{G_K}\in K^\times \otimes_{\Z}\F_p$. 
% This is Algorithm  \label{alg:finalcheck}, which appears on the next page. 

\begin{algorithm}
\label{alg:finalcheck}
\KwIn{$b\up 2|_{G_K} \in(K^\times \otimes_{\Z}\F_p)^{\Delta=\omega^0}$}
\KwIn{$\mathfrak{L}_0$, the distinguished prime of $K$ over $\ell_0$ (from Algorithm \ref{alogorithm_a1_ell0})}
\KwOut{True if $\beta = 0$; False otherwise}
\caption{Verify condition $(ii)$ in Theorem \ref{thm: main}}
\begin{enumerate}[leftmargin=1.75em]
    \item If $b\up2|_{G_K}$ is a $p$th power mod $\mathfrak{L}_0$, return True.
    \item Else, return False.
\end{enumerate}
\end{algorithm}
\begin{rem} While our goal in this section has been to verify whether $\alpha^2+\beta$ vanishes as an element in $\F_p(2)$, we could instead compute $\alpha^2+\beta$ as a canonical element in $\mu_p^{\otimes 2}$. (See  [Part I, \S\ref{sec: invariant is canonical}] for a precise explanation of what we mean by ``canonical" in this setting.) The key observation required to compute $\alpha^2+\beta$ as an element in $\mu_p^{\otimes 2}$ is that under the isomorphisms \[\F_p(1)\cong \mu_p,\;\;\;\F_p(-1)\cong\mu_p^{\otimes -1}\] induced by our choice of $\zeta_p\in\mu_p$, the ratio of $b_0(\gamma_0)$ and $c\up1(\gamma_0)$ is $\zeta_p\otimes\zeta_p\in\mu_p^{\otimes 2}$. Then, assuming that we have arranged for $\alpha=0$, we can compute $\alpha^2+\beta$ as $\zeta_p\otimes\zeta_p^e\in\mu_p^{\otimes 2}$, where $e$ is equal to $b\up2|_{\mathfrak{L}_0}(\gamma_0)$.
\end{rem}

\newpage

\section{Computed examples}
\label{sec:data} 

Using Sage \cite{SAGE}, we have computed whether the conditions in the main Theorem \ref{thm: main} hold for a wide selection of examples with $p=5,7$. To summarize our results, every example for which our algorithm completed within the allotted time is consistent with our conjecture that $R=\mathbb{T}$. Specifically, we either:
\begin{itemize}
    \item compute that $(i)$ and $(ii)$ of Theorem \ref{thm: main} are satisfied, and hence \[\dim_{\F_p}R/pR\geq 4,\] and independently compute that $\mathrm{rank}_{\Z_p}(\bT)\geq 4$, or
    \item compute that $(ii)$ of Theorem \ref{thm: main} is not satisfied, and hence $R = \bT$.
\end{itemize}
\begin{rem}
A particularly interesting computational observation is that condition $(i)$ in Theorem \ref{thm: main} has been satisfied in every example computed to date. We have been unable to explain why we might always expect $a\up1|_{\ell_1}=0$ from a theory perspective but hope to either do so in future work or find an example in which $a\up1|_{\ell_1}\neq 0$.
\end{rem}

\subsection{Scope of computations} Our program for checking the conditions in Theorem \ref{thm: main}, available online at \url{https://github.com/cmhsu2012/RR3}, is written for Sage Version 9.2 and implements the algorithms outlined in \S \ref{sec:adjustments} using the $S$-units interface to Pari/GP. All of our computations were carried out using either the Strelka Computer Cluster\footnote{The Strelka Computer Cluster is located at Swarthmore College. Its technical specifications can be found at \url{https://kb.swarthmore.edu/display/ACADTECH/Strelka+Computer+Cluster}.} or the SMP Cluster\footnote{The HTC Cluster is located at the Center for Research Computing at the University of Pittsburgh. Its technical specifications can be found at \url{https://crc.pitt.edu/resources}.} with an allotted computing time of 3 days per example. 

We have attempted to verify whether the conditions in Theorem \ref{thm: main} hold for the triples $(p,\ell_0,\ell_1)$ that satisfy Assumption \ref{assump: main} and are in the following ranges:
\begin{itemize}
\item $(5,\ell_0,\ell_1)$  with $\ell_0\leq 100$ and $\ell_1\leq 1000$, 
\item $(7,\ell_0,\ell_1)$ with $\ell_0\leq 50$ and $\ell_1\leq 500$. 
\end{itemize}
For computational convenience when reducing $S$-units modulo $p$-powers, our program also requires that $p$ does not divide the class number of $K = \Q(\zeta_p,\ell_1^{1/p})$. In these ranges, this additional assumption excludes only one triple, $(7,29, 347)$.

The most common Sage error that prevented the program from finishing–-other than the program simply timing out–-was the Sage interface for the Pari S-unit functionalities that led to high inefficiency or even memory overflow in some cases.\footnote{This problem is in the process of being fixed by Sage and Pari developers and is logged at \url{https://trac.sagemath.org/ticket/31327}.} When $p\geq 11$, our method for computing the $S$-unit $a_0|_{G_{\Q(\zeta_p)}}\in\Q(\zeta_p)^\times\otimes_\Z\F_p$ used in Algorithm \ref{alogorithm_a1_ell0} becomes infeasible for the 3-day time constraint. 

\subsection{A complete example.}\label{subsec: complete computation}

Let $p=5,\,\ell_0=11,$ and $\ell_1 = 23$. To set-up for our computations, we construct the number field $K = \Q(\zeta_5,\sqrt[5]{23})$. Taking $S$ denote the set of primes over $\ell_0$, we also construct $S$-unit groups $U_{\Q(\zeta_5),S}$ and $U_{K,S}$, which have respective ranks 5 and 29 as $\Z$-modules. In particular, in Sage, we represent lines in $U_{\Q(\zeta_5),S}\otimes_\Z\F_{5}$ and $U_{K,S}\otimes_\Z\F_5$ as elements of $\mathbb{P}^5(\F_5)$ and $\mathbb{P}^{29}(\F_5)$, respectively. Lastly, we check that $5$ wildly ramifies in $K/\Q$, which affects the method for computing the local adjustments in Algorithms \ref{alg:algorithm_a1_p} and \ref{algorithm_b2_ff}.

To construct $a\up1$ and $b\up2$, we start by computing $a_0|_{G_{\Q(\zeta_5)}}=(1,0,2,3,4,1)$ and $c=(3,4,4,4,1,1)$ in $U_{\Q(\zeta_5),S}\otimes_\Z\F_5$ through a brute-force check of the lines in $\mathbb{P}^5(\F_5)$, searching for the Kummer extensions of $\Q(\zeta_5)$ uniquely characterized by the definitions of $a_0$ and $c$. Next, we translate the conditions $\gamma\in (K^\times\otimes_\Z\F_5)^{\Delta=\omega^2}$ and $\Nm_{K/\Q(\zeta_5)}(\gamma)=c$ into a system of linear equations and solve to obtain 
\[\gamma=(2, 0, 1, 4, 0, 4, 1, 4, 2, 0, 0, 0, 0, 0, 4, 0, 4, 0, 0, 0, 0, 0, 0, 1,\dots, 0)\in U_{K,S}\otimes_\Z\F_5.\]
Using the formulas in Theorem \ref{thm:comp of cands}, we can now compute $a\up1\cand$ and $b\up2\cand$ and proceed to our computation of the required local adjustments. 

The diagram below follows the organization and notation of \S\ref{sec:adjustments}, giving the output at each step of our implementation. To ease notation in this diagram, we omit subscripts, such as ``$|_{G_K}$'', from the names of 1-cocycles. While subtleties of course arise, especially in making sure that the algorithms are compatible with each other, our implementation is largely straightforward, so we refer the reader to our Sage code for further details.

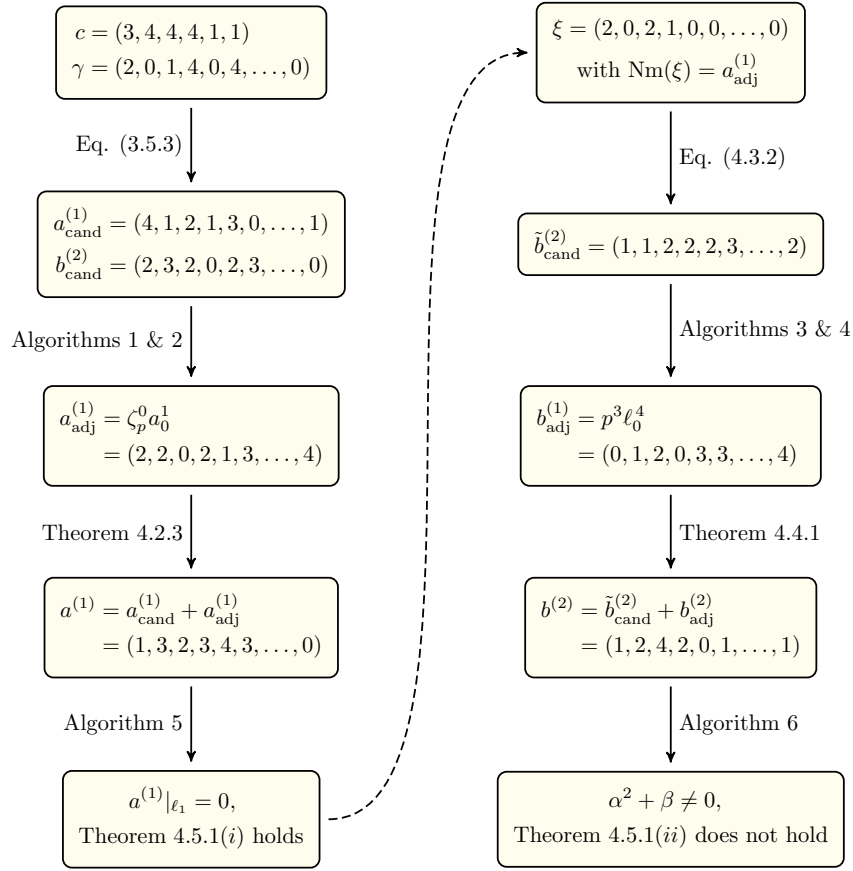
\newcommand{\outputnode}[3]{
\draw (#1) node [draw, thick,rounded corners,inner sep=7,outer sep=4,align=center,fill=yellow!08] (#2) {#3};
}
\tikzstyle{connector}=[black, thick,-stealth']
\tikzstyle{arrowlabel}=[black,right,align=left]
\vspace{.35cm}
\begin{center}
\scalebox{.85}{\begin{tikzpicture}

\outputnode{0,12}{0}{$\begin{aligned}c&=(3,4,4,4,1,1)\\
\gamma&=(2,0,1,4,0,4,\dots,0)
\end{aligned}$}

\outputnode{0,9}{1}{$\begin{aligned}a\up1\cand &= (4,1,2,1,3,0,\dots,1)\\
b\up2\cand &=(2,3,2,0,2,3,\dots,0)\end{aligned}$}

\outputnode{0,6}{2}{$\begin{aligned}a\up1\adj&=\zeta_p^{0}a_0^1\\&=(2,2,0,2,1,3,\dots,4)\end{aligned}$}

\outputnode{0,3}{3}{$\begin{aligned}a\up1&=a\up1\cand+a\up1\adj\\&=(1,3,2,3,4,3,\dots,0)\end{aligned}$}

\outputnode{0,0}{8}{$\begin{aligned}&a\up1|_{\ell_1}=0,\\\text{Theo}&\text{rem }\ref{thm: main}(i)\text{ holds}\end{aligned}$}

\outputnode{7.5,12}{4}{$\begin{aligned}\xi&=(2,0,2,1,0,0,\dots,0)\\&\text{ \;with } \mathrm{Nm}(\xi) = a\up1\adj\end{aligned}$}

\outputnode{7.5,9}{5}{$\tilde{b}\up2\cand=(1,1,2,2,2,3,\dots,2)$}

\outputnode{7.5,6}{6}{$\begin{aligned}b\up1\adj&=p^3\ell_0^{\,4}\\&=(0,1,2,0,3,3,\dots,4)\end{aligned}$ }

\outputnode{7.5,3}{7}{$\begin{aligned}b\up2&=\tilde{b}\up2\cand+b\up2\adj\\&=(1,2,4,2,0,1,\dots,1)\end{aligned}$}

\outputnode{7.5,0}{9}{$\begin{aligned}&\alpha^2+\beta\neq0,\\\text{Theorem }&\ref{thm: main}(ii)\text{ does not hold}\end{aligned}$}

\draw [connector] (0) -- node[arrowlabel,left,align=right] {Eq. (\ref{eq:a1 and b2 as Ds})} (1);
\draw [connector] (1) -- node[arrowlabel,left,align=right] {Algorithms \ref{alg:algorithm_a1_p} \& \ref{alogorithm_a1_ell0}} (2);
\draw [connector] (2) -- node[arrowlabel,left,align=right] {Theorem \ref{thm: a1 adjustments}} (3);
\draw [connector] (3) -- node[arrowlabel,left,align=right] {Algorithm \ref{alg:finalcheck1}} (8);

\draw [connector] (4) -- node[arrowlabel] {Eq. (\ref{eq: b_tilde_cand})} (5);
\draw [connector] (5) -- node[arrowlabel] {Algorithms \ref{algorithm_b2_ff} \& \ref{algorithm_b2_ell0}} (6);
\draw [connector] (6) -- node[arrowlabel] {Theorem \ref{thm: b2 adjustments}} (7);
\draw [connector] (7) -- node[arrowlabel] {Algorithm \ref{alg:finalcheck}} (9);

\draw (8.east)++(3,0) node[coordinate] (A) {};
\draw (4.west)++(-3,0) node[coordinate] (B) {};

\draw [black, dash pattern=on 4pt off 2pt, thick,-stealth'] (8.east) .. controls (A) and (B) .. (4.west);
\end{tikzpicture}}
\end{center}

\vspace{.35cm}
\noindent Since Theorem \ref{thm: main}$(ii)$ fails, we conclude $\dim_{\F_p}(R/pR)=3$, and hence, $R=\mathbb{T}$.

\newpage

\subsection{Tables}

For each choice of $(p,\ell_0,\ell_1)$ in Tables 1-6, we provide the following data related to the construction of $K'/K$ and $K''/K$:

\begin{itemize}
    \item The first three columns specify primes $(p, \ell_0, \ell_1)$ satisfying Assumption \ref{assump: main}. 
    \item The next two columns of ``$\beta$ difficulty factors'' indicate the two influences, outlined in \S\ref{subsec: finite-flat}, on the difficulty of the computations of the correct adjustments to $b\up2$ to find $\beta$; these computations are not strictly necessary to verify the conditions in Theorem \ref{thm: main} but can simplify the finite-flat adjustment computations significantly. Note that DNC stands for ``did not complete" within 3 days.  
        \item The next four columns are the adjustments for the local conditions on $a\up1$ and $b\up2$ discussed in \S \ref{sec:adjustments}: 
    \begin{itemize}
        \item $\zeta_p^i$ and $a_0^j = a_0|_{G_\Qz}^j$ give $a\up1\adj$, influencing $\alpha$;
        \item $p^k$ and $\ell_0^m$ give $b\up2\adj$, influencing $\beta$.
    \end{itemize}
    Note that we record the exponents $i,j,k,m$ of these adjustments in the tables.
    \item The second rightmost column says whether condition $(ii)$ in Theorem \ref{thm: main} holds. Since condition $(i)$ in Theorem \ref{thm: main} has been satisfied in all computed examples so far, we do not record this in our tables. As such, the second rightmost column gives the conclusion of Theorem \ref{thm: main}. 
    \item As a check on our main computations, the rightmost column gives the rank of $\mathbb{T}$, computed independently using modular symbols. Note that $``\geq 4^\ast"$ means that this computation did not complete within 3 days but allowed for the Hecke rank to be at least 4 after checking the first 20 Hecke operators.
\end{itemize}

Table 7 provides some additional examples in which large values of $\ell_1$ prevent a direct computation of the Hecke rank via modular symbols. So, although our computations have not been independently verified for these examples, we can conclude that the Hecke rank is 3 when condition $(ii)$ in Theorem \ref{thm: main} fails.

\begin{table}
\begin{center}

\caption{$p=5,\ell_0=11$}
{\footnotesize
\begin{tabularx}{4.5925in}{|c|c|c|c|c|c|c|c|c|c|c|}
\hline
    \multicolumn{3}{|c|}{\multirow{2}{*}{primes}} & \multicolumn{2}{c|}{\multirow{2}{*}{$\beta$ difficulty factors}} & \multicolumn{4}{c|}{local adjustments} & \multirow{2}{*}{conclusion} & \multirow{2}{*}{Hecke rank}
    \\ 
    \cline{6-9}
    \multicolumn{3}{|c|}{} & \multicolumn{2}{c|}{} & \multicolumn{2}{c|}{$a\up 1 \rightsquigarrow \alpha$} &\multicolumn{2}{c|}{$b\up2 \rightsquigarrow \beta$} & & 
    \\ \hline
      $p$ &$\ell_0$ &$\ell_1$ &$p \text{ in }K$ &$a^{(1)}|_p$ & $\zeta_p^i$ &$a_0^j$&$p^k$&$\ell_0^m$ &$\alpha^2+\beta = 0$? &$\mathrm{rk}(\mathbb{T})$ \\\hline \hline
    5&11 &23 &wild & $\neq 0$ &0 &1 &3&4 &no&3 \\\hline
    5&11 &43 &tame & $\neq 0$&3 &2  &0&4 &yes&$\geq 4$ \\\hline
  5&11 &67 &wild & $\neq 0$ &0 &0&1&3 & no&3 \\\hline
  5&11 &197 & wild & $\neq 0$ &0 &2& 1&4& yes&$\geq 4$  \\\hline
    5&11 &263  &wild & = 0&0 &2 &4&3&no&3\\\hline
    5&11 &307  &tame & = 0 &1&3 &0&0&no&3\\\hline
    5&11 &373  &wild & $\neq 0 $ &0&4 &0&3&no&3\\\hline
    5&11 &397 &wild & $\neq 0$ &0&4 &2&3&no&3\\\hline
    5&11 &593  &tame & = 0 &0&3 &0&2&no&3\\\hline
  5&11 &683 & wild & = 0 &0 &4& 3&0& yes&$\geq 4$  \\\hline
  5&11 &727 & wild & $\neq 0$ &0 &1& 1&3& yes&$\geq 4$  \\\hline
      5&11 &857 & tame & $\neq 0$ &2 &0& 0&4& no&3  \\\hline
    5&11 &967 & wild & $\neq 0$ &0 &0& 2&2& no&3  \\\hline
      5&11 &1013  &wild & $\neq 0$ &0&3 &3&1&no&3\\\hline
     
\end{tabularx}}
\end{center}
\end{table}

\begin{table}
\begin{center}

\caption{$p=7,\ell_0=29$}
{\footnotesize
\begin{tabularx}{4.5325in}{|c|c|c|c|c|c|c|c|c|c|c|}\hline
    \multicolumn{3}{|c|}{\multirow{2}{*}{primes}} & \multicolumn{2}{c|}{\multirow{2}{*}{$\beta$ difficulty factors}} & \multicolumn{4}{c|}{local adjustments} & \multirow{2}{*}{conclusion} & \multirow{2}{*}{Hecke rank}
    \\ 
    \cline{6-9}
    \multicolumn{3}{|c|}{} & \multicolumn{2}{c|}{} & \multicolumn{2}{c|}{$a\up1 \rightsquigarrow \alpha$} &\multicolumn{2}{c|}{$b\up2 \rightsquigarrow \beta$} & & 
    \\ \hline
      $p$ &$\ell_0$ &$\ell_1$ &$p \text{ in }K$ &$a^{(1)}|_p$ & $\zeta_p^i$ &$a_0^j$&$p^k$&$\ell_0^m$ &$\alpha^2+\beta = 0$? &$\mathrm{rk}(\mathbb{T})$ \\\hline \hline

    7&29 &17  & wild & DNC& 0&3 &0& 0&no&3\\\hline
    7&29 &157  &wild & DNC& 0& 6&4&4 &no&3\\\hline
    7&29 &521 & tame & $\neq 0$ & 3& 6& 0&2&yes&$\geq 4$ \\\hline
\end{tabularx}}

\end{center}
\end{table}

\begin{table}
\begin{center}

\caption{$p=5,\ell_0=41$}
{\footnotesize
\begin{tabularx}{4.5325in}{|c|c|c|c|c|c|c|c|c|c|c|}\hline
    \multicolumn{3}{|c|}{\multirow{2}{*}{primes}} & \multicolumn{2}{c|}{\multirow{2}{*}{$\beta$ difficulty factors}} & \multicolumn{4}{c|}{local adjustments} & \multirow{2}{*}{conclusion} & \multirow{2}{*}{Hecke rank}
    \\ 
    \cline{6-9}
    \multicolumn{3}{|c|}{} & \multicolumn{2}{c|}{} & \multicolumn{2}{c|}{$a\up1\rightsquigarrow \alpha$} &\multicolumn{2}{c|}{$b\up2\rightsquigarrow \beta$} & & 
    \\ \hline
      $p$ &$\ell_0$ &$\ell_1$ &$p \text{ in }K$ &$a^{(1)}|_p$ & $\zeta_p^i$ &$a_0^j$&$p^k$&$\ell_0^m$ &$\alpha^2+\beta = 0$? &$\mathrm{rk}(\mathbb{T})$ \\\hline \hline
      5&41 &73  &wild & $\neq 0$ &0&4&2&1&yes&$\geq$ 4\\\hline
5&41 &83  &wild & $\neq 0$ &0&4&3&3&no&3\\\hline
5&41 &137  &wild & $\neq 0$ &0&3&2&0&no&3\\\hline
5&41 &163 &wild & $\neq 0$ &0&2&3&1&no&3\\\hline    
5&41 &167 &wild & $\neq 0$&0&2&4&0&no&3\\\hline
5&41 &173 &wild & = 0 &0&0&3&2&no&3\\\hline
5&41 &383 &wild & = 0 &0&3&2&2&no&3\\\hline
5&41 &547 &wild & $\neq 0$ &0&3&0&0&no&3\\\hline
5&41 &577 &wild & $\neq 0$&0&0&2&0&yes&$\geq 4^\ast$\\\hline
%5&41 &653 &wild & = 0 &0&2&1&0&yes&size error\\\hline
5&41 &683 &wild & $\neq 0$ &0&2&3&0&no&3\\\hline
%5&41 &823 &wild & = 0 &0&0&1&3&no&size error\\\hline
5&41 &983 &wild & $\neq 0$ &0&1&4&1&yes&$\geq 4^\ast$\\\hline
    
\end{tabularx}}

\end{center}
\end{table}

\begin{table}
\begin{center}

\caption{$p=5,\ell_0=61$}
{\footnotesize
\begin{tabularx}{4.5325in}{|c|c|c|c|c|c|c|c|c|c|c|}\hline
    \multicolumn{3}{|c|}{\multirow{2}{*}{primes}} & \multicolumn{2}{c|}{\multirow{2}{*}{$\beta$ difficulty factors}} & \multicolumn{4}{c|}{local adjustments} & \multirow{2}{*}{conclusion} & \multirow{2}{*}{Hecke rank}
    \\ 
    \cline{6-9}
    \multicolumn{3}{|c|}{} & \multicolumn{2}{c|}{} & \multicolumn{2}{c|}{$a\up1 \rightsquigarrow \alpha$} &\multicolumn{2}{c|}{$b\up2 \rightsquigarrow \beta$} & & 
    \\ \hline
      $p$ &$\ell_0$ &$\ell_1$ &$p \text{ in }K$ &$a^{(1)}|_p$ & $\zeta_p^i$ &$a_0^j$&$p^k$&$\ell_0^m$ &$\alpha^2+\beta = 0$? &$\mathrm{rk}(\mathbb{T})$ \\\hline \hline

5&61 &13  &wild & $\neq 0$ &0&2&2&3&no&3\\\hline
5&61 &47  &wild & $\neq 0$ &0&0&1&0&yes&$\geq$ 4\\\hline
5&61 &197  &wild & = 0 &0&3&4&3&no&3\\\hline
5&61 &257  &tame & $\neq 0$ &2&0&0&4&no&3\\\hline
5&61 &337  &wild & = 0 &0&0&1&1&no&3\\\hline
5&61 &353  &wild & $\neq 0$&0&3&4&3&no&3\\\hline
5&61 &367  &wild & $\neq 0$&0&1&4&2&no&3\\\hline
5&61 &487  &wild & $\neq$ 0 &0&4&4&3&yes&$\geq 4^\ast$\\\hline
5&61 &563 &wild & $\neq$ 0 &0&4&3&1&no&3\\\hline
5&61 &733 &wild & $\neq$ 0 &0&0&3&1&no&3\\\hline
%5&61 &743 &tame & = 0 &2&0&0&4&no&size error\\\hline
5&61 &853 &wild & $\neq$ 0 &0&4&1&2&yes&$\geq 4^\ast$\\\hline
%5&61 &883 &wild & $\neq$ 0 &0&0&4&2&no&size error\\\hline
5&61 &977 &wild & $\neq$ 0 &0&3&2&2&yes&$\geq 4^\ast$\\\hline
%5&61 &997 &wild & $\neq$ 0 &0&0&3&2&no&size error\\\hline
\end{tabularx}}

\end{center}
\end{table}

\begin{table}
\begin{center}

\caption{$p=5,\ell_0=71$}
{\footnotesize
\begin{tabularx}{4.5325in}{|c|c|c|c|c|c|c|c|c|c|c|}\hline
    \multicolumn{3}{|c|}{\multirow{2}{*}{primes}} & \multicolumn{2}{c|}{\multirow{2}{*}{$\beta$ difficulty factors}} & \multicolumn{4}{c|}{local adjustments} & \multirow{2}{*}{conclusion} & \multirow{2}{*}{Hecke rank}
    \\ 
    \cline{6-9}
    \multicolumn{3}{|c|}{} & \multicolumn{2}{c|}{} & \multicolumn{2}{c|}{$a\up1 \rightsquigarrow \alpha$} &\multicolumn{2}{c|}{$b\up2 \rightsquigarrow \beta$} & & 
    \\ \hline
      $p$ &$\ell_0$ &$\ell_1$ &$p \text{ in }K$ &$a^{(1)}|_p$ & $\zeta_p^i$ &$a_0^j$&$p^k$&$\ell_0^m$ &$\alpha^2+\beta = 0$? &$\mathrm{rk}(\mathbb{T})$ \\\hline \hline

5&71 &23  &wild & $\neq 0$ &0&0&0&2&no&3\\\hline
5&71 &37  &wild & $\neq 0$ &0&4&2&1&no&3\\\hline
5&71 &97  &wild & = 0 &0&3&2&3&no&3\\\hline
5&71 &103  &wild & $\neq 0$ &0&2&2&3&no&3\\\hline
5&71 &193  &tame & $\neq 0$ &2&2&0&3&no&3\\\hline
5&71 &233  &wild & = 0 &0&3&1&2&yes&$\geq 4^\ast$\\\hline
5&71&283&wild&$\neq$ 0&0&1&3&0&no&3\\\hline
5&71&307&tame&$\neq$ 0&3&2&0&3&no&3\\\hline
5&71&463&wild&$\neq$ 0&0&4&0&0&no&3\\\hline
%5&71 &613  &wild & = 0 &0&4&2&1&no&size error\\\hline
%5&71 &673  &wild & $\neq$ 0 &0&2&2&1&no&size error\\\hline
%5&71 &733  &wild & $\neq$ 0 &0&4&1&3&no&size error\\\hline
5&71 &853  &wild & $\neq$ 0 &0&2&0&4&no&3\\\hline
\end{tabularx}}

\end{center}
\end{table}

\begin{table}
\begin{center}

\caption{$p=7,\ell_0=43$}
{\footnotesize
\begin{tabularx}{4.6225in}{|c|c|c|c|c|c|c|c|c|c|c|}\hline
    \multicolumn{3}{|c|}{\multirow{2}{*}{primes}} & \multicolumn{2}{c|}{\multirow{2}{*}{$\beta$ difficulty factors}} & \multicolumn{4}{c|}{local adjustments} & \multirow{2}{*}{conclusion} & \multirow{2}{*}{Hecke rank}
    \\ 
    \cline{6-9}
    \multicolumn{3}{|c|}{} & \multicolumn{2}{c|}{} & \multicolumn{2}{c|}{$a\up1 \rightsquigarrow \alpha$} &\multicolumn{2}{c|}{$b\up2 \rightsquigarrow \beta$} & & 
    \\ \hline
      $p$ &$\ell_0$ &$\ell_1$ &$p \text{ in }K$ &$a^{(1)}|_p$ & $\zeta_p^i$ &$a_0^j$&$p^k$&$\ell_0^m$ &$\alpha^2+\beta = 0$? &$\mathrm{rk}(\mathbb{T})$ \\\hline \hline

7&43 &37  & wild & DNC& 0&5 &\phantom{aa}0\phantom{aa}& 1&yes&$\geq$ 4\\\hline
7&43 &79  & tame & $\neq 0$& 2&4 &\phantom{aa}0\phantom{aa}& 3&no&3\\\hline
\end{tabularx}}

\end{center}

\end{table}

\begin{table}
\begin{center}

\caption{Examples without an independent computation of Hecke rank}
{\footnotesize
\begin{tabularx}{3.8175in}{|c|c|c|c|c|c|c|c|c|c|}
\hline
    \multicolumn{3}{|c|}{\multirow{2}{*}{primes}} & \multicolumn{2}{c|}{\multirow{2}{*}{$\beta$ difficulty factors}} & \multicolumn{4}{c|}{local adjustments} & \multirow{2}{*}{conclusion} 
    \\  
    \cline{6-9}
    \multicolumn{3}{|c|}{} & \multicolumn{2}{c|}{} & \multicolumn{2}{c|}{$a\up 1 \rightsquigarrow \alpha$} &\multicolumn{2}{c|}{$b\up2 \rightsquigarrow \beta$} & 
    \\ \hline
      $p$ &$\ell_0$ &$\ell_1$ &$p \text{ in }K$ &$a^{(1)}|_p$ & $\zeta_p^i$ &$a_0^j$&$p^k$&$\ell_0^m$ &$\alpha^2+\beta = 0$? \\\hline \hline

5&41 &653 &wild & = 0 &0&2&1&0&yes\\\hline

5&41 &823 &wild & = 0 &0&0&1&3&no\\\hline

5&61 &743 &tame & = 0 &2&0&0&4&no\\\hline

5&61 &883 &wild & $\neq$ 0 &0&0&4&2&no\\\hline

5&61 &997 &wild & $\neq$ 0 &0&0&3&2&no\\\hline

5&71 &613  &wild & = 0 &0&4&2&1&no\\\hline
5&71 &673  &wild & $\neq$ 0 &0&2&2&1&no\\\hline
5&71 &733  &wild & $\neq$ 0 &0&4&1&3&no\\\hline
\end{tabularx}}
\end{center}
\end{table}

\clearpage

%\appendix

\section{Algebraic number theory} 
\label{sec: P2 ANT}

While the $C_p$-extensions $K'/K$ and $K''/K$ in Theorem \ref{thm: main} are defined in terms of the cochains $a\up1$ and $b\up2$, an intrinsic characterization is desirable. These fields are not Galois over $\Q$; accordingly, they depend on choices, such as the pinning data. In this section, we will characterize the Galois closures (over $\Q$) $M'$ of $K'$ and $M''$ of $K''$, and characterize the isomorphism class of subfields of $M''$ to which $K'$ and $K''$ belong. This establishes the descriptions of $K'$ and $K''$ used to state the main theorem in the introduction (Theorem \ref{thm: main intro}). Additionally, this allows us to rephrase the conditions of the main Theorem \ref{thm: main} in order to refer to the decomposition subgroups of $\Gal(M''/\Q)$ for primes over $\ell_0$. This translation appears in Theorem \ref{thm: main ANT} as well as in another form in Theorem \ref{thm: HL P2}.

\subsection{The extensions $M'/M$ and $K'/K$} 

To describe $M'/M$ and $K'/K$, the 3-dimensional $G_{\Q,Np}$-representation of \eqref{eq: 3d abc} is a convenient tool; it is 
\begin{equation}
    \label{eq: 3d abc again}
\upsilon = \begin{pmatrix}
\omega & b\up1 & \omega a\up1\cand \\
0 & 1 & \omega c\up1 \\
0 & 0 & \omega
\end{pmatrix} : G_{\Q,Np} \to \GL_3(\F_p), 
\end{equation}
where we view $a\up1\cand$ as a general choice of matrix entry making $\upsilon$ a homomorphism. We recall that $K/\Q(\zeta_p)$ is cut out by $b\up1$, $L/\Q(\zeta_p)$ is cut out by $c\up1$, $M=KL/\Q(\zeta_p)$ is cut out by $(b\up1,c\up1)$, $K'/K$ is cut out by $a\up1$, and $M'/\Q$ is cut out by the entire representation $\upsilon$ when we let $a\up1\cand = a\up1$. 

To state the proposition, some notation and terminology is needed. Let $F/M$ be its maximal extension that is ramfied only at primes dividing $Np$, abelian, of exponent $p$, and Galois over $\Q$. We also call a number field extension $A'/A$ \emph{$\ell_0$-split} to briefly say that all primes of $A$ over $\ell_0$ are split in $A'/A$. 
\begin{prop}
\label{prop: characterize M'}
$M'/M$ is the unique unramified and $\ell_0$-split $C_p$-extension contained in $F$ whose Galois group is coinvariant under the conjugation action of $\Gal(M/\Q)$ on $\Gal(F/M)$. The isomorphism class of $K'/K$ is characterized by being a $C_p$-extension contained in $M'$ and not equal to $M/K$.
\end{prop}

\begin{proof}
Let $M'\cand/M$ be a $C_p$-extension contained in $F/M$ whose Galois group is coinvariant under the conjugation action of $\Gal(M/\Q)$ on $\Gal(F/M)$. A central extension sequence  arises from $M'\cand$,
\[
1 \to \Gal(M'\cand/M) \to \Gal(M'\cand/\Q) \to \Gal(M/\Q) \to 1. 
\]
Using the semi-direct product decomposition
\[
\Gal(M/\Q) = \Gal(M/\Q(\zeta_p)) \rtimes \Gal(M/\Q(\ell_1^{1/p}, c\up1) =: \Gal(M/\Q(\zeta_p)) \rtimes \Delta, 
\]
and the description of group cohomology of a semi-direct product of \cite{tahara1972}, we find that the element of $H^2(\Gal(M/\Q), \Gal(M'\cand/M))$ determined by $\Gal(M'\cand/\Q)$ arises from $H^2(\Gal(M\cand/\Q(\zeta_p), \Gal(M'\cand/M))^\Delta$. We use the natural isomorphism $\Delta \cong \Gal(\Q(\zeta_p)/\Q)$ and represent its actions on $\F_p$-vector spaces according the usual notation $\F_p(i)$ for $i \in \Z$. Since $\Gal(M/\Q(\zeta_p)) \simeq \F_p(-1) \times \F_p(1)$ and $\Gal(M'\cand/M) \simeq \F_p(0)$, standard calculations yield that 
\[
H^2(\Gal(M\cand/\Q(\zeta_p), \Gal(M'/M))^\Delta
\]
is $1$-dimensional. Since the class of $\Gal(M'\cand/\Q)$ as a central extension is non-trivial, as is the extension generated by the group $\Gal(M/\Q)$, they are equal up to a scalar and therefore isomorphic. The upshot is that we obtain a faithful matrix representation of $\Gal(M'\cand/\Q)$, just like the faithful representation \eqref{eq: 3d abc again} of $\Gal(M'/\Q)$. 

The upper right coordinate of this matrix representation yields a new $a\up1\cand : G_{\Q,Np} \to \F_p$ such that $-da\up1\cand = b\up1\smile c\up1$.  The conclusion of what we have argued so far is that $C_p$-extensions of $M$ contained in $F/M$ whose Galois groups that are coinvariant for the action of $\Gal(M/\Q)$ correspond with solutions $a\up1\cand$ to $-da\up1\cand = b\up1\smile c\up1$. This correspondence is bi-directional, and it is not necessary to refine this statement in order to obtain a bijection.

Let us fix a corresponding pair $M'\cand$ and $a\up1\cand$, also letting $K'\cand/K$ be the $C_p$-extension of $K$ cut out by $a\up1\cand\vert_{G_K}$, and review the local conditions characterizing $a\up1$ that are stated in Proposition \ref{prop: produce a1}. In order to complete the proof using this correspondence, we claim that $a\up1\cand$ satisfies the local conditions of Proposition \ref{prop: produce a1} if and only if $M'\cand/M$ is unramified. Then the uniqueness of $M'/M$, claimed here, will follow from the uniqueness of $a\up1$ proved in Proposition \ref{prop: produce a1}. 

We first prove that $M'\cand/M$ is unramified and $\ell_0$-split if and only if the local conditions of Proposition \ref{prop: produce a1} hold true for $a\up1\cand$. We will use the following implication of the fact that both $M'$ and $M$ are Galois over $\Q$ throughout the argument: $M'/M$ being unramified at a single prime of $M$ over a rational prime $q$ is equivalent to $M'/M$ being unramified at all primes of $M$ over $q$. A similar statement applies to the $\ell_0$-split condition. We also implicitly use the fact that $a\up1\cand\vert_{G_M}$ cuts out $M'/M$. 

\noindent
\textbf{Unconditionally, $M'/M$ is unramified at all primes of $M$ over $\ell_1$.} Becuase $a\up1\cand\vert_{\ell_1} : G_{\ell_1} \to \F_p$ is a cocycle (since $b\up1\vert_{\ell_1} = 0$), it is automatically unramified. Thus $a\up1\cand\vert_{G_M}$ is unramified at $\ell_1$. 

\noindent 
\textbf{$M'/M$ is $\ell_0$-split if and only if condition (c) holds.} Condition (c) of Proposition \ref{prop: produce a1}, imposed on $a\up1\cand$, is equivalent to the two non-zero homomorphisms
\[
c\up1\vert_{\ell_0} : G_{\ell_0} \to \F_p(-1), \quad a\up1\vert_{\ell_0} : G_{\ell_0} \to \F_p
\]
having identical kernels. Since $M/K$ is cut out by $c\up1\vert_{G_K}$, we deduce that $a\up1\vert_{G_M}$ vanishes at the distinguished prime over $\ell_0$ if and only if $M'/M$ is $\ell_0$-split.

\noindent
\textbf{$M'/M$ is unramified at all primes of $M$ over $p$ if and only if condition (a) holds.} 
Condition (a) of \ref{prop: produce a1} reads that $(a\up1\cand + b\up1 \smile x_c)\vert_{I_p} = 0$. Let $L_c \subset L$ be the subfield fixed by $\F_p^\times$ under the isomorphism
\[
\Gal(L/\Q) \isoto \F_p^\times \ltimes \F_p \xrightarrow[\sm{1}{\omega c\up1}{0}{\omega}]{\sim} \ttmat{1}{*}{0}{*} \subset \GL_2(\F_p). 
\]
This $L_c/\Q$ has degree $p$ and its Galois closure is $L/\Q$. It is a brief exercise to check that $a\up1\vert_{G_p'}$ is a cocycle and condition (a) is equivalent to $a\up1\cand\vert_{I_p'} = 0$, where $I_p' \subset G_p' \subset G_{\Q,Np}$ is an alternate choice of decomposition group and inertia group at $p$ such that the alternate distinguished prime has trivial ramification degree in $L_c/\Q$. Using the facts about $M'/M$ listed above, we conclude that $a\up1\cand\vert_{I_p'} = 0$ if and only if $M'/M$ is unramified at all primes of $M$ over $p$.

Finally, we establish the claimed characterization of $K'/K$ as a $C_p$-extension contained in $M'/K$. Under the isomorphism $\upsilon$, the subgroup $\Gal(M'/K) \subset \Gal(M'/\Q)$ is identified with the subgroup of $\GL_3(\F_p)$ isomorphic to $\F_p \oplus \F_p$ that differs from the identity matrix in the $b\up1$ and $a\up1$-coordinates. Consider the action of $\Gal(K/\Q)$ by conjugation on the $p+1$ subgroups of $\Gal(M'/K)$ of order $p$. One of them is fixed (the one concentrated in the $a\up1$-coordinate), and has fixed field $M$. The remaining $p$ are a single orbit. Therefore, any of their fixed fields are isomorphic, and one of them is $K'$. 
\end{proof}

\subsection{The extensions $M''/M'$ and $K''/K$}

Recall that the cochain $b\up2$ satisfying differential equation \eqref{eq:diffeq for b2} exists if and only if $a\up1\vert_{\ell_1} = 0$. When $b\up2$ does not exist, we consider $K''$ undefined and let $M'' = M'$. 

For the rest of this section, we assume that $b\up2$ does exist. Let $K''/K$ be the extension cut out by $b\up2\vert_{G_K}$ as usual, and let $M''$ denote the Galois closure of $K''$ over $\Q$. Our goal is to describe these extensions using the 4-dimensional $G_{\Q,Np}$-representation of \eqref{eq: 4d abcd}, 
\begin{equation}
    \label{eq: 4d nu}
    \nu := 
\begin{pmatrix}
\omega & b\up1 & \omega a\up1 & b\up2\\
0 & 1 & \omega c\up1 &  d\up1 \\
0 & 0 & \omega & b\up1 \\
0 & 0 & 0 & 1
\end{pmatrix} : G_{\Q,Np} \to \GL_4(\F_p). 
\end{equation}
Note that $d\up1 = a\up1-b\up1 c\up1$, so that $M'' = \oQ^{\ker \nu}$ is a $C_p$-extension of $M' = \oQ^{\ker \upsilon}$ cut out by $b\up2\vert_{G_{M'}}$. 
In addition to satisfying the differential equation required to make $\nu$ a homomorphism, we recall that $b\up2 \in C^1(\Z[1/Np], \F_p(1))$ is characterized by the local conditions of Proposition \ref{prop: produce b2}, but only up to addition by the subspace of cocycles spanned by $\{b\up1, dx\}$ for some choice of non-zero $x \in \F_p(1)$.  However, since $b\up1\vert_{G_K} = 0$, such changes to $b\up2$ do not change the kernel of $\nu$. We therefore may and do regard the Galois extension $M''/\Q$ as well-defined. 

To state the proposition, let $F'/M'$ be its maximal extension that is ramfied only at primes dividing $Np$, abelian, of exponent $p$, and Galois over $\Q$. Let $\m^\fl$ be the modulus of $M'$ (in the sense of ray class field theory) defined as the product of all squares of primes of $M'$ over $p$; that is, 
\[
\m^\fl := \prod_{\p \mid (p)} \p^2. 
\]
Finally, when $V$ is an irreducible $\F_p$-linear representation of $\Gal(M'/\Q)$, we say that an intermediate field $F''$, $F'\supset F'' \supset M$, is \emph{$V$-equivariant} when $\Gal(F'/M) \otimes_{\F_p} V \rsurj \Gal(F''/M) \otimes_{\F_p} V$ factors through the coinvariants of the $\Gal(M/\Q)$-action (by conjugation in $\Gal(F/\Q)$). 
\begin{prop}
\label{prop: characterize M''}
Assume that $a\up1\vert_{\ell_1} = 0$. Then $M''/M'$ is the unique $C_p$-extension contained in $F'$ such that it has conductor $\m^\fl$, it is $\ell_0$-split, and its Galois group's $\Gal(M'/\Q)$-action is $\F_p(1)$-equivariant. The isomorphism class of $K''/K$ is characterized by being a $C_p$-extension contained in $M''$, not contained in $M'$, and having cardinality $p$. 
\end{prop}

We can rephrase the proposition in terms of a ray class group. Let $C$ denote the maximal quotient of the ray class group of $M'$ of conductor $\m^\fl$ such that it has exponent $p$ and such that the images of prime ideals over $\ell_0$ vanish. We use $C_\omega$ to denote the maximal quotient of $C$ whose $\Gal(M'/\Q)$-action is $\omega$-isotypic. 
\begin{cor}
\label{cor: aup1 vanishing}
The group $C_\omega$ has order $p$ if $a\up1\vert_{\ell_1} = 0$ and has order $1$ if $a\up1\vert_{\ell_1} \neq 0$. The extension of $M'$ associated to $C_\omega$ is $M''$. 
\end{cor}

To prove Proposition \ref{prop: characterize M''}, first we prove the following local lemma, which justifies calling $\m^\fl$ the ``finite-flat modulus.'' We write ``$v_p$'' for the normalized valuation on $\oQ_p$, that is, the valuation $v_p : \oQ_p^\times \to \Q$ such that $v_p(p) = 1$. Also, write $\pi_A$ for the uniformizer of a finite extension $A/\Q_p$. In this lemma, we allow $p \geq 3$, in contrast with our usual assumption that $p \geq 5$. 
\begin{lem}
\label{lem: conductor of finite-flat extension}
Let $p$ be an odd prime. Let $F/\Q_p$ be a Galois extension of degree $p^d(p-1)$ that contains $\Q_p(\zeta_p)$ and is contained in a finite extension of $\Q_p$ cut out by the Galois action on the $\bar\Q_p$-points of a finite-flat group scheme over $\Z_p$ of exponent $p$. Let $F \supset H \supset \Q_p$ such that $[F : H] = p$, $H/\Q_p$ is Galois, and $F/H$ is totally ramified. Then the conductor of $F/H$ is $(\pi_H)^2$. 
\end{lem}

\begin{rem}
As the proof will explain, the statement of the lemma is equivalent to the following formula for the valuation of the different of $F/\Q_p$. If the ramification degree of $F/\Q_p$ is written $p^e(p-1)$, then 
\begin{equation}
    \label{eq: different formula}
v_p(\mathrm{Diff}(F/\Q_p)) = \frac{p^{e+1}-2}{p^e(p-1)}.
\end{equation}
\end{rem}

\begin{proof}
The key input is Fontaine's upper bound on the different of an extension $F/\Q_p$ cut out by the action on a finite-flat group scheme: as a particular case of \cite[\S0.1, Corollaire, pg.\ 516]{fontaine1985}, we find that
\[
v_p(\mathrm{Diff}(F/\Q_p)) < \frac{p}{p-1}. 
\]
Because differents are multiplicative in towers, this bound on the different also applies to subextensions of $F/\Q_p$. 

We prove \eqref{eq: different formula} by induction. The base case $e=0$ follows from the standard calculation that $\mathrm{Disc}(\Q_p(\zeta_p)/\Q_p) = (p^{p-2})$; it follows that $\mathrm{Diff}(\Q_p(\zeta_p)/\Q_p) = ((\zeta_p-1)^{p-2})$, which has absolute valuation $(p-2)/(p-1)$ as desired. 

Now we deduce the truth of \eqref{eq: different formula} for $e'=e+1$ in place of $e$ from its truth as written. Let $F/\Q_p$ be as in the lemma, with ramification degree $p^{e'}(p-1)$. Since $\Gal(F/\Q_p)$ is solvable and we can always decompose an extension into an unramified extension followed by a totally ramified extension, we may choose an intermediate field $H$ such that $F \supset H \supset \Q_p(\zeta_p)$, $F/H$ is of degree $p$ and ramified, and $\Gal(F/H) \subset \Gal(F/\Q_p)$ is a normal subgroup. Therefore we can apply \eqref{eq: different formula} to $H$ and conclude that 
\[
v_p(\mathrm{Diff}(H/\Q_p)) = \frac{p^{e+1}-2}{p^e(p-1)}.
\]
Because the different is multiplicative in towers, Fontaine's bound on $\mathrm{Diff}(F/\Q_p)$ implies that
\begin{align*}
&v_p(\mathrm{Diff}(F/H)) = \\
&v_p(\mathrm{Diff}(F/\Q_p)) - v_p(\mathrm{Diff}(H/\Q_p)) < \frac{p}{p-1} - \frac{p^{e+1}-2}{p^e(p-1)} = \frac{2p}{p^{e'}(p-1)}. 
\end{align*}
On the other hand, because $F/H$ is abelian and ramified, a standard result bounding the possible differents of wildly ramified extensions (see e.g.\ \cite[Thm.\ 2.6, Ch.\ III]{neukirch1999}) states that 
\[
v_p(\mathrm{Diff}(F/H)) \geq \frac{p}{p^{e'}(p-1)}.
\]
Altogether, letting $m$ be the integer satisfying $\mathrm{Diff}(F/H) = (\pi_F)^m$, the bounds above dictate that
\[
p \leq m \leq 2p - 1
\]

To determine $m$, we apply the conductor-discriminant formula for $F/H$, which states that $\mathrm{Cond}(F/H)^{p-1} = \mathrm{Disc}(F/H)$. Note also that $\mathrm{Disc}(F/H) = (\pi_H)^m$. Therefore $(p-1) \mid m$ as well. Because $p$ is odd, the bounds on $m$ imply that $m=2(p-1)$. 

Applying the calculation of $m$,  \eqref{eq: different formula} follows by calculating
\begin{align*}
v_p(\mathrm{Diff}(F/\Q_p)) = v_p(\mathrm{Diff}(H/\Q_p)) +  v_p(\mathrm{Diff}(F/H)) \\ = \frac{p^{e+1}-2}{p^e(p-1)} + \frac{2p-2}{p^{e'}(p-1)} = \frac{p^{e'+1}-2}{p^{e'}(p-1)}
\end{align*}
and $\mathrm{Cond}(F/H) = (\pi_F)^2$, as desired. 
\end{proof}

\begin{proof}[{Proof of Proposition \ref{prop: characterize M''}}]
An argument similar to the one appearing in the beginning of the proof of Proposition \ref{prop: characterize M'} proves that a $C_p$-extension of $M'$ contained in $F'$ and with the $\F_p(1)$-coinvariance property of Proposition \ref{prop: characterize M''} exists if and only if a matrix representation of $G_{\Q,Np}$ the form $\nu$ cuts it out. Proposition \ref{prop: produce b2} proves that $\nu$ exists if and only if $a\up1\vert_{\ell_1} = 0$, and that, in that case, there exists a choice of its $b\up2$-coordinate with the properties (a) and (b) of Proposition \ref{prop: produce b2}. 

What we will prove is that properties (a) and (b) of a candidate solution $b\up2\cand$ to differential equation \eqref{eq:diffeq for b2} listed in Proposition \ref{prop: produce b2} hold true if and only if the extension $M''\cand/M'$ cut out by $b\up2\cand\vert_{G_{M'}}$ satisfies the properties listed in Proposition \ref{prop: characterize M''}. We have already observed that, while there is a torsor of possibilities for $b\up2\cand$, the extension $M''/M'$ cut out by $b\up2\vert_{G_{M'}}$ is nonetheless well-defined. Therefore, the uniqueness of $M''/M'$ will follow. 

\noindent
\textbf{Unconditionally, $M''\cand/M'$ is unramified at all primes of $M'$ over $\ell_1$.} 
Under our assumption that $a\up1\vert_{\ell_1} = 0$, which implies that $d\up1\vert_{\ell_1}=0$, the cochain $b\up2\cand\vert_{\ell_1} : G_{\ell_1} \to \F_p$ is a cocycle, since $-db\up2$ equals $b\up1 \smile d\up1 + a\up1 \smile b\up1$. Therefore, $b\up2\cand\vert_{\ell_1}$ is in the span of $b\up1$ up to 1-coboundaries. Since 1-coboundaries on $G_{\ell_1}$ valued in $\F_p(1)$ vanish on inertia and $b\up1\vert_{G_{M'}} = 0$, we conclude that $M''\cand/M'$ is unramified at the distinguished prime over $\ell_1$. Because both $M'$ and $M''\cand$ are Galois over $\Q$, it follows that all primes of $M'$ over $\ell_1$ are unramified in $M''\cand$. 

\noindent
\textbf{$M''\cand/M'$ is $\ell_0$-split if and only if condition (a) of Proposition \ref{prop: produce b2} holds if and only if there exists a $\rho_{2,\mathrm{cand}} : G_{\Q,Np} \to E_2^\times$ as in \eqref{eq:rho2} with $b\up2\cand$ as its $b\up2$-coordinate.}
A very similar argument to the case of $M'\cand/M$ argued in the proof of Proposition \ref{prop: characterize M'} applies to prove the first equivalence. The second equivalence follows directly from [Part I, Lem.\ \ref{P1: lem: exists 2nd-order 1-reducible}(2)]. 

\noindent
\textbf{$M''\cand/M'$ has conductor dividing $\m^\fl$ and is $\ell_0$-split if and only if $\rho_{2,\mathrm{cand}}$ as in \eqref{eq:rho2} exists and also satisfies condition (b) of Proposition \ref{prop: produce b2}.} Condition (b) states that there exists some $\rho_2$ as in \eqref{eq:rho2} such that $\rho_2\vert_p$ is finite-flat and has $b\up2\cand$ as its $b\up2$-coordinate. Using the previous $\ell_0$-local claim, we assume that the corresponding pair $(b\up2\cand, M''\cand)$ occurs as the $b\up2$-coordinate of some $\rho_{2,\mathrm{cand}}$, so that it only remains to address the $p$-local conditions

Assume condition (b) is true. At the distinguished prime $v''$ of $M''\cand$ over $p$, $(M''\cand)_{v''}/\Q_p$ is contained in $\oQ_p^{\ker\rho_{2,\mathrm{cand}}\vert_p}$. Condition (b) implies that $\oQ_p^{\ker\rho_{2,\mathrm{cand}}\vert_p}/\Q_p$ is cut out by the $G_p$-action on the $\oQ_p$-points of a finite-flat group scheme over $\Z_p$ with exponent $p$. The same statement applies to the subfield $M' \subset M''$ with distinguished prime $v'$ over $p$. 
Therefore, by Lemma \ref{lem: conductor of finite-flat extension}, the conductor of $(M''\cand)_{v''}/M'_{v'}$ is either $v'^2$ or $1$. Because both $M''\cand/\Q$ and $M'/\Q$ are Galois extensions, this local conductor calculation applies to all primes of $M'$ over $p$. In other words, $\mathrm{Cond}(M''\cand/M') \mid \m^\fl$, as desired.  

We next prove the converse: assume that $M''\cand/M'$ has conductor dividing $\m^\fl$, having been cut out by $b\up2\cand\vert_{G_{M'}}$ where the only assumption on $b\up2\cand$ is that it is an element of $C^1(\Z[1/Np], \F_p(1))$ satisfying the differential equation \eqref{eq:da1 and db2}. The set of solutions $b\up2\cand$ is a torsor under $Z^1(\Z[1/Np], \F_p(1))$, which has basis $\{dx, b_p, b\up1, b_0\up1\}$. By \cite[Lem.\ C.4.1]{WWE3}, the set of solutions $b\up2\cand$ making $\nu\vert_p$ finite-flat are a torsor under the subspace $Z^1(\Z[1/Np], \F_p(1))^\fl = \langle dx, b\up1, b_0\up1\rangle$ computed in [Part I, Lem.\ \ref{lem: global Kummer theory}], namely,
\[
b\up2 + Z^1(\Z[1/Np], \F_p(1))^\fl. 
\]
Lemma \ref{lem: conductor of finite-flat extension} implies that the $C_p$-extensions of $M'_{v'}$ cut out by $b\vert_{G_{M_{v'}'}}$ for any $b\in b\up2 + Z^1(\Z[1/Np], \F_p(1))^\fl$ satisfies the conductor bound stated in the Lemma. Conversely, one can calculate that the conductor of $b_p\vert_{G_{\Q_p(\zeta_p)}}$ does not cut out a $C_p$-extension satisfying the conductor bound, which implies the same result for the $C_p$-extension of $M'_{v'}$ cut out by $b\vert_{G_{M'_{v'}}}$ for any 
\[
b \in \left[b\up2 + Z^1(\Z[1/Np], \F_p(1))\right] \smallsetminus \left[b\up2 + Z^1(\Z[1/Np], \F_p(1))^\fl\right]. 
\]
Therefore, because $M''\cand/M'$ has conductor bounded by $\m^\fl$, $b\up2\cand \in b\up2 + Z^1(\Z[1/Np], \F_p(1))^\fl$, which is equivalent to $\nu\vert_p$ being finite-flat. According to the torsor structures on $\Pi_2^{\det}$ and $\Pi_2^{\det,p}$ described in Part I, Lemma \ref{lem: rho2 constant determinant} and Part I, Proposition \ref{P1: prop: exists unique beta}, one can adjust $\rho_2$ only in its $a\up2$, $c\up2$, and $d\up2$-coordinates to produce a $\rho_{2,\mathrm{cand}}'$ that is finite-flat at $p$ with $b\up2$-coordinate $b\up2\cand$. This completes the claimed equivalence. 

It only remains to prove the claimed characterization of $K''/K$. Under the embedding $\nu : \Gal(M''/\Q) \rinj \GL_4(\F_p)$ of \eqref{eq: 4d nu}, the abelian subgroup $\Gal(M''/K)$ admits an isomorphism
\begin{equation}
\label{eq: M''/K to V}
\left.\ttmat{a\up1}{b\up2}{c\up1}{-a\up1}\right\vert_{G_K} : \Gal(M''/K) \isoto V \subset M_2(\F_p)
\end{equation}
to the subspace of trace $0$ matrices $V \subset M_2(\F_p)$. Likewise, $\nu$ produces an isomorphism
\[
\ttmat{\omega}{b\up1}{0}{1} : \Gal(K/\Q) \isoto B := \ttmat{*}{*}{0}{1} \subset \GL_2(\F_p). 
\]
It follows from the shape of $\nu$ in \eqref{eq: 4d nu} that the natural conjugation action of $\Gal(K/\Q)$ on $\Gal(M''/K)$ matches the usual adjoint action of $B$ on $V$ via these isomorphisms. 

Our characterization of $K''/K$ will follow from the following description of orbits of the action of $\Gal(K/\Q)$ by conjugation on the $p^2 + p +1$ subgroups of $\Gal(M''/K)$ isomorphic to $C_p \times C_p$, one of which has fixed field equal to $K''$. Under the isomorphisms above, the following description of five orbits  of this action, labeled (a)-(e), matches the listing labeled (a)-(e) in Lemma \ref{lem: 2d orbits}.

\begin{enumerate}[label=(\alph*)]
\item One orbit is a singleton: the one concentrated in the $a\up1$ and $b\up2$-coordinates. Its fixed field is $M$. 
\item There is an orbit of cardinality $p$, one of which is the subgroup $\Gal(M''/K') \subset \Gal(M''/K)$, which is concentrated in the $c\up1$ and $b\up2$-coordinates. The fixed fields of its conjugates are the subfields of $M''$ isomorphic to $K'$, all of which are contained in their common Galois closure, $M'$.
\item Another orbit of cardinality $p$ includes the subgroup $\Gal(M''/K'')$ that is concentrated in the $c\up1$ and $a\up1$-coordinates. The fixed fields of its conjugates are the subfields of $M''$ isomorphic to $K''$; none of these are contained in $M'$.
\item An orbit of cardinality $p(p-1)/2$. 
\item Another orbit of cardinality $p(p-1)/2$. 
\end{enumerate}
Because of the running assumption that $p \geq 5$, we conclude that the isomorphism class of $K''$ is the only isomorphism class of $C_p$-extensions of $K$ contained in $M''$ that has cardinality $p$ and also is not contained in $M'$. 
\end{proof}

\subsection{Orbits in Grassmannians of the trace zero adjoint representation of the conjugation action of the Borel subgroup}

While there is usually a running assumption that $p \geq 5$, in this section only we also allow $p=3$. Consider the adjoint action of the matrix subgroup 
\[
B = \F_p^\times \ltimes \F_p \cong \ttmat{\ast}{\ast}{0}{1} \subset \GL_2(\F_p)
\]
on the $3$-dimensional $\F_p$-vector subspace $V \subset M_2(\F_p)$ consisting of matrices of trace $0$. This induces an action of $B$ on the set of two-dimensional subspaces of $V$, which we now describe. We will use the expression in coordinates $a,b,c \in \F_p$ according to the basis
\[
V \ni v = b \cdot \ttmat{0}{1}{0}{0} + a \cdot 
\ttmat{1}{0}{0}{-1} + c \cdot \ttmat{0}{0}{1}{0} = (b,a,c).
\]
We will also use the self-adjointness of the trace pairing on $V$. 
\begin{lem}
\label{lem: 2d orbits}
There are 5 orbits of the adjoint action of $B$ on the set of two-dimensional subspaces of $V$. A complete list of orbits and their cardinalities, along with at least one representative of the orbit, follows. 
\begin{enumerate}[label=(\alph*)]
\item $\langle (1,0,0), (0,1,0)\rangle$, the upper-triangular matrices, comprises an orbit of cardinality $1$
\item $\langle (1,0,0), (0,a,1)\rangle$ for $a \in \F_p$, the 2-dimensional subspaces that contain $\sm{0}{1}{0}{0}$ but are not contained in the upper-triangular matrices, together comprise an orbit of size $p$
\item $\langle (0,a, 1), (4a, 1, 0)\rangle$ for $a \in \F_p$, which together comprise an orbit of size $p$ 
\item $\langle (0,1,0), (-1, 0, 1)\rangle$ is a representative of an orbit of size $p(p-1)/2$
\item $\langle (0,1,0), (-1, 0, c)\rangle$, where $c \in \F_p^\times$ is not a square, is a representative of an orbit of size $p(p-1)/2$
\end{enumerate}
When $p \geq 5$, the orbit (c) is the only orbit of cardinality $p$ that has no representative containing $\sm{0}{1}{0}{0}$. 
\end{lem}

Lemma \ref{lem: 2d orbits} follows directly from the following lemma characterizing the orbits of the $B$-action on $1$-dimensional subspaces of $V$, along with the perfect self-duality of $V$ with respect to the trace pairing. We omit the proof deducing Lemma \ref{lem: 2d orbits} from the following lemma, but list the orbits there in the same order that their duals were listed in Lemma \ref{lem: 2d orbits}.  
\begin{lem}
\label{lem: 1d orbits}
There are 5 orbits of the adjoint action of $B$ on the set of one-dimensional subspaces of $V$. A complete list of orbits and their cardinalities, along with at least one representative of the orbit, follows. 
\begin{enumerate}[label=(\alph*)]
\item $\langle (1,0,0)\rangle$, the strictly upper-triangular matrices, comprise an orbit of cardinality $1$
\item $\langle (b,1,0)\rangle$ for $b \in \F_p$, the lines that are upper-triangular but not strictly upper-triangular, together comprise an orbit of size $p$
\item $\langle (-a^2,a, 1)\rangle$ for $a \in \F_p$, which together comprise an orbit of size $p$
\item $\langle (1,0,1)\rangle$ is a representative of an orbit of size $p(p-1)/2$
\item $\langle (b,0,1)\rangle$, where $b \in \F_p^\times$ is not a square, is a representative of an orbit of size $p(p-1)/2$
\end{enumerate}
When $p \geq 5$, the orbit (c) is the only orbit of cardinality $p$ whose members do not consist entirely of upper-triangular matricies. 
\end{lem}

\begin{proof}
The orbit listed in (a) is, indeed, an orbit, because conjugation by $B$ preserves the properties ``upper-triangular'' and ``strictly upper-triangular.'' For this reason, along with the fact that the stabilizer of $\langle (0,1,0)\rangle \subset V$ is the diagonal subgroup $\F_p^\times$ of $B$, we apply the orbit-stabilizer lemma to deduce that there is an orbit as listed in (b). 

The diagonal subgroup $\F_p^\times$ of $B$ fixes $\langle (0,0,1)\rangle \subset V$, but one easily checks that its normal subgroup $\F_p \subset B$ acts faithfully on the orbit of $\langle (0,0,1)\rangle$. By the orbit-stabilizer lemma, we have the orbit (c). 

One readily checks that the stabilizer of $\langle(b,0,1)\rangle$, for any $b \in \F_p^\times$, is the subgroup $\sm{\pm 1}{}{}{1} \subset B$. In addition, the maximal subgroup of $B$ preserving the set of all lines of the form $\langle (b,0,1)\rangle$, for $b \in \F_p^\times$, is  the diagonal subgroup $\F_p^\times \subset B$. Because the representatives listed in (d) and (e) are not in the same orbit under the adjoint action of the diagonal subgroup of $B$, we deduce from the orbit-stabilizer lemma that they each are representatives of orbits of size $p(p-1)/2$. 

Because the total number of lines in $V$ is $p^2+p + 1$, which equals the sum of the cardinalities of orbits listed in (a)-(e),  the list of orbits is complete. 
\end{proof}

\subsection{The value and vanishing of $\alpha^2+\beta$ in terms of algebraic number theory}
In line with the long tradition of theorems relating congruences between cusp forms and Eisenstein series with an algebraic number-theoretic condition that started with Ribet's converse to Herbrand's theorem, we would like to rephrase the conditions for $\bT$ to have minimal rank, proved in the main Theorem \ref{thm: main}, in terms of algebraic number theory. This has precedent, for example, the main theorem of \cite{WWE3} that shows that $\bT_{\ell_0}$ has minimal rank if and only if a certain class group has $p$-torsion of minimal rank. 

We have already seen in Proposition \ref{prop: characterize M''} that condition (i) of Theorem \ref{thm: main} is equivalent to the existence of $p$-torsion in a ray class group of $M'$. It remains to interpret condition (ii), the vanishing of $\alpha^2+\beta$, in terms of algebraic number theory. This is already partially complete in this expression of condition (ii) of Theorem \ref{thm: main} in terms of prime decomposition in the extensions $K'/K$ and $K''/K$ cut out by $a\up1\vert_{G_K}$ and $b\up2\vert_{G_K}$, respectively. Therefore our goal is to ``remove $a\up1$ and $b\up2$'' from this expression by substituting universal characterizations of $K', K''$ as completed in Propositions \ref{prop: characterize M'} and \ref{prop: characterize M''}. Additionally, we characterize whether or not $\alpha^2+\beta \in \mu_p^{\otimes 2}$ is a square, where \emph{squares} have the form $\zeta \otimes \zeta$ for some $\zeta \in \mu_p$. 

We begin with the following description of the $C_p \times C_p$-extension $K'K''/K$, a corollary of the proof of Proposition \ref{prop: characterize M''}. 
\begin{cor}
The composite $K'K'' \subset M''$ is a member of the unique isomorphism class of $C_p \times C_p$-extensions of $K$ contained in $M''$ that has cardinality $p$ and does not contain $M$.
\end{cor}

\begin{proof}
Under the homomorphism $\nu$ of \eqref{eq: 4d nu}, $K'K'' \subset M''$ is the fixed field of the cyclic order $p$ subgroup of $\Gal(M''/\Q)$ concentrated in the $c\up2$-coordinate. This subgroup belongs to the conjugacy class listed as item (c) in Lemma \ref{lem: 1d orbits}. As we see in that lemma, this is the only conjugacy class of cardinality $p$ that has members that are non-trivial in the $c\up2$-coordinate, and therefore does not fix $M$. 
\end{proof}

\begin{thm}
\label{thm: main ANT}
The following algebraic number theoretic conditions characterize $\alpha^2+\beta \in \mu_p^{\otimes 2}$. 
\begin{enumerate}
\item $\alpha^2 + \beta$ vanishes if and only if the isomorphism class consisting of decomposition subfields of $M''/\Q$ at primes over $\ell_0$ has cardinality $p$ if and only if this isomorphism class equals the isomorphism class of $K'K''$. 
\item $\alpha^2 + \beta \in \mu_p^{\otimes 2}$ is a non-vanishing square if and only if the isomorphism class of decomposition subfields of $M''/\Q$ at the prime $\ell_0$ is equal to the fixed fields of the order $p$ subgroups of $\Gal(M''/K)$ given in part (d) of Lemma \ref{lem: 1d orbits} under the isomorphism $V \cong \Gal(M''/K)$. 
\item $\alpha^2 + \beta \in \mu_p^{\otimes 2}$ is a non-square if and only if the isomorphism class of decomposition subfields of $M''/\Q$ at the prime $\ell_0$ is equal to the fixed fields of the order $p$ subgroups of $\Gal(M''/K)$ given in part (e) of Lemma \ref{lem: 1d orbits} under the isomorphism $V \cong \Gal(M''/K)$. 
\end{enumerate}
\end{thm}

\begin{proof}
Characterization (1) of the vanishing of $\alpha^2+\beta$ follows directly from Theorem \ref{thm: main}. Indeed, if there exists a prime of $M''$ over $\ell_0$ such that the prime of $K$ lying below it splits in both $K'/K$ and $K''/K$, then the decomposition subgroup of this prime fixes $K'K''$. (Recall that $\ell_0$ splits completely in $K/\Q$.) 

Let $D_{\ell_0} \subset \Gal(M''/\Q)$ denote a decomposition subgroup arising from the choice of a prime of $M''$ over $\ell_0$. It has order $p$. In order to prove characterizations (2) and (3), determining the image of $\alpha^2+\beta \in \mu_p^{\otimes 2}\smallsetminus \{1 \otimes 1\}$ in the doubleton set $(\mu_p^{\otimes 2}\smallsetminus \{1 \otimes 1\})/\Gal(\Q_p(\zeta_p)/\Q)$, we simply need to review definitions of $\alpha, \beta$, and the isomorphism \eqref{eq: M''/K to V}. Putting these together, isomorphism \eqref{eq: M''/K to V} simplifies when restricted to $D_{\ell_0}$ according to
\[
\left.\ttmat{a\up1}{b\up2}{c\up1}{-a\up1}\right\vert_{D_{\ell_0}} = 
\left.\ttmat{\alpha c\up1}{\beta c\up1}{c\up1}{-\alpha c\up1}\right\vert_{D_{\ell_0}}.
\]
The existence of a choice of prime over $\ell_0$ such that $a\up1\vert_{\ell_0} = 0$ and equivalently $\alpha=0$, discussed in Lemma \ref{lem: choice pinning data}, implies that we can replace $D_{\ell_0} \subset \Gal(M''/\Q)$ by a conjugate subgroup such that 
\[
\left.\ttmat{a\up1}{b\up2}{c\up1}{-a\up1}\right\vert_{D_{\ell_0}} = 
\left.\ttmat{0}{(\alpha^2+\beta) c\up1}{c\up1}{0}\right\vert_{D_{\ell_0}}.
\]
Therefore, because $c\up1\vert_{\ell_0} \neq 0$, there is a generator of $D_{\ell_0}$ whose image in $V$ is $\sm{0}{\alpha^2+\beta}{1}{0}$. Looking at the list of conjugacy classes of order $p$ subgroups of $\Gal(M''/K)$ listed in Lemma \ref{lem: 1d orbits}, one can read off characterizations (2) and (3). 
\end{proof}

\subsection{A terminal result of Part I and II}

In this final section, we present the preceding results, especially Corollary \ref{cor: aup1 vanishing} and Theorem \ref{thm: main ANT}(1), in a form that allows us to deduce a ``terminal result'' of this pair of papers, which is stated below as Theorem \ref{thm: HL P2}. 

Let $N_{\ell_0}$ be the normalizer of a decomposition group $D_{\ell_0} \subset \Gal(M''/\Q)$, and let $n_{\ell_0}$ denote the order of $N_{\ell_0}$. This integer $n_{\ell_0}$ is independent of the choice of $D_{\ell_0}$. Here is a complete description of $n_{\ell_0}$. Recall that $M''=M'$ if and only if $a\up1\vert_{\ell_1} = 0$. 
\begin{prop}
\label{prop: cases for n_ell_0}
If $M''=M'$, then $n_{\ell_0} = p^2(p-1)$. If $M''/M'$ has degree $p$, then $n_{\ell_0}$ is described by the following two cases. 
\begin{itemize}
    \item $n_{\ell_0} = 2p^3$ if and only if the conjugacy class of the subgroup $D_{\ell_0} \subset \Gal(M''/\Q)$ has order $p(p-1)/2$. 
    \item $n_{\ell_0} = p^3(p-1)$ if and only if the conjugacy class of the subgroup $D_{\ell_0} \subset \Gal(M''/\Q)$ has order $p$. 
\end{itemize}
\end{prop}
\begin{proof}
First we claim that $D_{\ell_0}$ has order $p$ and is contained in $\Gal(M''/K)$. The claim about the order follows from the facts that 
\begin{itemize}
    \item $\ell_0$ splits completely in $K/\Q$ (see [Part I, Lem.\ \ref{lem: log ell1 is zero}])
    \item the primes over $\ell_0$ are ramified in $M/K = KL/K$ because the same is true for $L/\Q(\zeta_p)$ (see the description of $L/\Q(\zeta_p)$ in \S\ref{subsec: P2 main results}) 
    \item the primes over $\ell_0$ are split in $M'/M$ and $M''/M'$ according to Propositions \ref{prop: characterize M'} and \ref{prop: characterize M''}.
\end{itemize}

According to the claim, in the case $[M'' : M'] = p$, $D_{\ell_0} \subset \Gal(M''/K)$ is a line in a 3-dimensional $\F_p$-vector space. Moreover, its conjugacy class in $\Gal(M''/\Q)$ is described by Lemma \ref{lem: 1d orbits} under the isomorphism \eqref{eq: M''/K to V} from $\Gal(M''/K)$ to $V$. Because $[M'' : \Q] = p^4(p-1)$ and the normalizer $N_{\ell_0}$ is the stablizer of the conjugacy action of $\Gal(M''/\Q)$ on $D_{\ell_0}$, the proposition's claims in the case $[M'' : M']=p$ follow from the orbit-stabilizer lemma. 

In the case $M''=M'$, we use the representation $\upsilon : \Gal(M'/\Q) \to \GL_3(\F_p)$ of \eqref{eq: 3d abc again} to draw an isomorphism
\[
(a\up1, c\up1)\vert_{G_K} : \Gal(M'/K) \to \F_p \oplus \F_p.
\]
The claim implies that each choice of decomposition group $D_{\ell_0} \subset \Gal(M'/K)$ is a line. Because of the form of $\upsilon$, a much simpler analysis than Lemma \ref{lem: 1d orbits} yields that the conjugacy classes of order $p$ subgroups of $\Gal(M'/\Q)$ contained in $\Gal(M'/K)$ are described as follows in terms of $\upsilon$.
\begin{itemize}
    \item The subgroup concentrated in the $a\up1$-coordinate comprises a singleton orbit.
    \item The subgroups containing elements with non-zero $c\up1$-coordinate comprise an orbit of order $p$. 
\end{itemize}
Because $c\up1\vert_{\ell_0} \neq 0$, the conjugacy class containing $D_{\ell_0} \subset \Gal(M'/\Q)$ has cardinality $p$. Therefore, the proposition's claim in the case $M''=M'$ follows from the fact that $[M':\Q] = p^3(p-1)$ and an applicaiton of the orbit-stabilizer lemma. 
\end{proof}

\begin{thm}
\label{thm: HL P2}
If $n_{\ell_0} = p^2(p-1)$ or $n_{\ell_0} = 2p^3$, then the $\Z_p$-rank of $R$ is $3$, there is a unique newform of level $N$ congruent to the Eisenstein series, and $R \cong \bT$. If $n_{\ell_0} = p^3(p-1)$, then $\dim_{\F_p} R/pR > 3$.
\end{thm}
\begin{proof}
The conclusions of this theorem are the same as those of Theorem \ref{thm: main}, so we only need to check that the condition ``$n_{\ell_0} = p^3(p-1)$'' is equivalent to the condition that ``both conditions (i) and (ii) of Theorem \ref{thm: main} are true.'' Indeed, Proposition \ref{prop: cases for n_ell_0} also tells us that $n_{\ell_0} = p^2(p-1)$ or $n_{\ell_0} = 2p^3$ when $n_{\ell_0} \neq p^3(p-1)$.  

Corollary \ref{cor: aup1 vanishing} draws an equivalence between condition (i) and the condition that $[M'': M']=p$. Given that (i) is true, Theorem \ref{thm: main ANT}(1) draws an equivalence between condition (ii) and the condition that the conjugacy class of subgroups of $\Gal(M''/\Q)$ containing (any choice of) $D_{\ell_0}$ has cardinality $p$. By Proposition \ref{prop: cases for n_ell_0}, we conclude that condition (ii) is equivalent to $n_{\ell_0} = p^3(p-1)$ upon the assumption that (i) is true. And (i) is also necessarily true when $n_{\ell_0} = p^3(p-1)$, according to Proposition \ref{prop: cases for n_ell_0}. 
\end{proof}

% \tocless{\section*}{Conflicts of Interest Statement}
%  The authors assert that there are no conflicts of interest.

% \vspace{.5cm}
% \tocless{\section*}{Data Availability Statement}
%  All data generated or analysed during this study are included in this  published article (and its supplementary information files).

\vspace{.5cm}
\bibliographystyle{alpha}
\bibliography{CWEbib-2020-RR3}

\end{document}